\documentclass[11pt, reqno]{amsart}
\usepackage{amsmath, amssymb}
\usepackage[all]{xypic}
\usepackage{graphics}
\usepackage{blindtext}
\usepackage{tabularx}

\usepackage{times}
\usepackage[T1]{fontenc}
\usepackage{mathrsfs}
\usepackage{epsfig}
\usepackage{color}

\newtheorem{theorem}{Theorem}[section]

\newtheorem{defi}[theorem]{Definition}
\newtheorem{lemma}[theorem]{Lemma}
\newtheorem{prop}[theorem]{Proposition}

\def\slfrac#1#2{\hbox{\kern.1em %
 \raise.5ex\hbox{\the\scriptfont0 #1}\kern-.11em %
 /\kern-.15em\lower.25ex\hbox{\the\scriptfont0 #2}}}

\newcommand{\rra}{{a}}
\newcommand{\rrc}{{c}}

\newcommand{\Cbddz}{{C_{bdd}^0}}
\newcommand{\Cbdd}{{C_{bdd}}}
\newcommand{\CNd}{C_{N,d}}
\newcommand{\HNd}{{\mathcal H}_{N,d}}
\newcommand{\WNd}{{\mathcal W}_{N,d}}
\newcommand{\Wdd}{{\mathcal W}_{d,d}}
\newcommand{\SNd}{{\mathcal S}_{N,d}}

\newcommand{\bpsi}{{\mbox{\boldmath $\psi$}}}

\newcommand{\eqn}[1]{(\ref{#1})}
\newcommand{\hsp}{\hspace*{\parindent}}

\newcommand{\eeq}{\end{equation}}
\newcommand{\beql}[1]{\begin{equation}\label{#1}}
\newcommand{\bsq}{{\vrule height .9ex width .8ex depth -.1ex }}
\newcommand{\sgn}{ {\rm sgn}}

\newcommand{\CC}{{\mathbb C}}
\newcommand{\FF}{{\mathbb F}}
\newcommand{\HH}{{\mathbb H}}

\newcommand{\QQ}{{\mathbb Q}}
\newcommand{\RR}{{\mathbb R}}

\newcommand{\ZZ}{{\mathbb Z}}
\newcommand{\bone}{{\bf I}}

\newcommand{\bv}{{\bf v}}

\newcommand{\bx}{{\bf x}}

\newcommand{\bU} {{\rm U}}
\newcommand{\bV} {{\rm V}}

\newcommand{\sA}{{\mathcal A}}

\newcommand{\sD}{{\mathcal D}}
\newcommand{\sE}{{\mathcal E}}
\newcommand{\sF}{{\mathcal F}}
\newcommand{\sH}{{\mathcal H}}
\newcommand{\sK}{{\mathcal K}}
\newcommand{\sM}{{\mathcal M}}
\newcommand{\sN}{{\mathcal N}}

\newcommand{\sR}{{\mathcal R}}
\newcommand{\sS}{{\mathcal S}}

\newcommand{\sW}{{\mathcal W}}
\newcommand{\sX}{{\mathcal X}}

\newcommand{\hA}{{\rm A}} 
\newcommand{\hB}{{\rm B}}

\newcommand{\hJ}{{\rm J}}
\newcommand{\hM}{{\rm M}}
\newcommand{\hR}{{\rm R}}

\newcommand{\hT}{{\rm T}}
\newcommand{\hU}{{\rm U}}
\newcommand{\hV}{{\rm V}}
\newcommand{\hX}{{\rm X}}
\newcommand{\hY}{{\rm Y}}
\newcommand{\hZ}{{\rm Z}}

\newcommand{\fe} {{\mathfrak f}}

\newcommand{\thV}{{\tilde{\rm V}}}

\title{The Lerch Zeta Function and the Heisenberg Group}
\author{Jeffrey C. Lagarias}
\thanks{This research  was partially supported by NSF grants  DMS-1401224 and 1701576}
\address{Department of Mathematics, University of Michigan,
Ann Arbor, MI 48109-1043,USA}
\email{lagarias@umich.edu}

\date{December 21, 2020}

\begin{document}

\begin{abstract}
This paper gives a representation-theoretic interpretation
of the Lerch zeta function and related Lerch $L$-functions
twisted by Dirichlet characters.
These functions  are associated to a four-dimensional solvable
real Lie group $H^{J}$, called  here the sub-Jacobi group,
which is a semi-direct product of $GL(1, \RR)$
with the Heisenberg group $H(\RR)$.
The Heisenberg group action on $L^2$-functions on the Heisenberg nilmanifold $H(\ZZ)\backslash H(\RR)$
 decomposes as $\bigoplus_{N \in \ZZ} \sH_N$,
where each space $\sH_N~ (N \neq 0)$ consists of $|N|$ copies
of an irreducible infinite-dimensional representation of $H(\RR)$ with
central character $e^{2 \pi i Nz}$.
The paper shows that  show one can further decompose $\sH_N
(N \ne 0)$
into irreducible $H(\RR)$-modules  $\sH_{N,d}(\chi)$ indexed by
Dirichlet characters $(\bmod~ d)$ for $d \mid N$, each of
which carries an irreducible $H^J$-action.
On each $\sH_{N,d}(\chi)$ there is an action of
certain two-variable Hecke operators $\{\hT_m: m \ge 1\}$; 
these Hecke operators 
 have a natural global definition on all of 
$L^2(H(\ZZ)\backslash H(\RR))$, including 
the space of one-dimensional representations $\sH_0$.
For $H_{N,d}(\chi)$
with $N \neq 0$  
suitable Lerch $L$-functions
on the critical line $\frac{1}{2} + it$ 
form  a complete family of generalized eigenfunctions
(purely continuous spectrum) 
for a certain linear partial
differential operator $\Delta_L$.
These Lerch $L$-functions are also simultaneous eigenfunctions 
for all two-variable Hecke operators $T_m$ and their adjoints $T_m^{\ast}$, 
provided $(m, N/d) = 1$. 
Lerch $L$-functions are characterized by this Hecke eigenfunction property.
\end{abstract}

\maketitle
\setlength{\baselineskip}{1.0\baselineskip}

%
%
%

\section{Introduction}\label{sec1}
\hsp
\setcounter{equation}{0}
The {\em Lerch zeta function}
is defined by
\beql{101}
\zeta (s,a,c) = \sum_{n=0}^\infty e^{2\pi ina} (n+c)^{-s} ~.
\eeq
It is named after 
M. Lerch~\cite{Le87}, who in 1887 derived
a three term functional equation that it satisfies.
The parameter value $a=0$ reduces to the Hurwitz
zeta function, and the further specialization to
 $(a, c) = (0, 1)$ gives the Riemann zeta function. It is
well known that the functional equations for the
Hurwitz and Riemann zeta function can be derived
by specialization from that of the Lerch zeta function.

This paper addresses  the question: Where does the Lerch
zeta function fit in the framework of automorphic representation
theory?
 It provides the following answer,
as a special case  $(N,d)=(1,1)$ of the more general Theorem \ref{th95}.\\

\begin{theorem}\label{th11}  The  two symmetrized Lerch-zeta
functions 
$$
L^{\pm}(s, a, c) = \zeta(s, a, c) \pm e^{2\pi i a} \zeta(s, 1-a, 1-c)
$$
 are  `Eisenstein series
for the real Heisenberg group $H(\RR)$ with respect to the discrete
subgroup given by the integer Heisenberg group $H(\ZZ)$, in the
following sense. 
\begin{enumerate}
\item
They form a continuous family of generalized  eigenfunctions in the $s$-parameter
with respect to a ``Laplacian operator" 
$\Delta_L= \frac{1}{2\pi i} \frac{\partial}{\partial \rra}
\frac{\partial}{\partial \rrc}
+ \rrc \frac{\partial}{\partial \rrc} + \frac{1}{2}$,
having eigenvalue $\frac{1}{2}-s$.
\item
 The operator $\Delta_L$ 
defines a left-invariant vector
field on $H(\RR)$, and acts 
on a Hilbert space $\sH_1$ inside $L^2( H(\ZZ) \backslash H(\RR))$,
which is invariant under the right $H(\RR)$-action.
\item
The Hilbert space $\sH_1$
carries an irreducible action of $H(\RR)$,
the Schr\"{o}dinger representation, with central character $e^{2\pi i z}$. 
\item
The operator $\Delta_L$ is   specified as an unbounded operator  on $\sH_1$ by
a dense domain\\
 $\sW(\sD_{1,1}) \subset \sH_1$ with respect to which it is
skew-adjoint. It has pure continuous spectrum with 
generalized eigenfunctions  given by $L^{\pm}(\frac{1}{2} + i\tau , a, c)$,
for $\tau \in \RR$,  with a specified spectral measure.
\end{enumerate}
\end{theorem}

More generally, this paper provides a complete spectral decomposition of
$L^2(  H(\RR)/ H(\ZZ))$ with respect to this vector field.
The Hilbert space $L^2(  H(\RR)/ H(\ZZ))$  decomposes into a countable
direct sum of pieces  $\sH_N$ indexed by the (integer) value $N$ of the central character.
This paper treats the case of nonzero $N$  and shows that all such spaces
 carry pure continuous spectra of various
multiplicities, with generalized eigenfunctions of $\Delta_L$ being symmetrized Lerch
functions twisted by  (primitive or imprimitive)  Dirichlet characters, which we term
{\em Lerch $L$-functions}. We give a decomposition of  $\sH_N$ for $N \ne \pm 1$
 into smaller pieces, labeled $\sH_{N, d}(\chi)$ where $d \mid N$ and $\chi$ is a (primitive or imprimitive)
Dirichlet character $(\bmod \, d)$, and each  piece carries an irreducible representation of $H(\RR)$
with central character value $N$.  A variant of  the operator  $\Delta_L$ acts on each space  $\sH_{N,d}(\chi)$
(on a suitable dense domain) and its spectrum is pure continuous, with associated generalized eigenfunctions given 
by a pair of  Lerch $L$-functions on the critical line. The remaining case $N=0$ will be treated in
a sequel paper.

We also introduce two-variable ``Hecke operators" $\hT_m$ acting on these spaces  $\sH_N$, where for $N=\pm 1$  the functions
$L^{\pm}(s, a, c)$ are simultaneous eigenfunctions of all the $\hT_m$. For $|N| \ge 2$
all the  operators $\{\hT_m: m \ge 1\}$ form a commuting family acting on $\sH_N$, however  these operators
are normal operators on $\sH_N$  exactly when $(m, |N|)=1$. 
The smaller spaces $\sH_{N, d}(\chi)$ are  invariant under the action of the Hecke operators
$\hT_m$ having $(m, \frac{N}{d}) =1$.  The associated pair of Lerch $L$-functions on this space are simultaneous eigenfunctions of
all the Hecke operators $\hT_m$ and their adjoints $\hT_m^{\ast} $ that have $(m, N) =1$. 
 These spaces are not invariant under
other Hecke operators having $(m, \frac{N}{d}) >1$, and to restore full Hecke algebra invariance requires
combining some of these spaces into larger spaces, as indicated below.
One can  characterize the Lerch $L$-functions in terms of the Hecke operator action, using results of \cite{LL4}.

The spectral analysis on the  Hilbert space $\sH_0$, corresponding to  $N=0$ 
 has a different structure (see (\cite{LH2}). 
Under the  action of $\Delta_L$ alone it splits into one-dimensional eigenspaces.
The dilation action of the sub-Jacobi group is broken, but there  exist  two-variable Hecke operators $\hT_m$
retaining part of this action.
These operators are not normal operators for $m \ge 2$. Their action is 
 not invertible
on parts of $\sH_0$.

\subsection{Previous work}\label{sec11}

In  four  joint papers with Winnie Li (\cite{LL1}--\cite{LL4})
the author studied properties of the Lerch zeta function. 
The paper  \cite{LL1}   derived  two symmetric four-term functional 
equations for the Lerch
zeta function, valid for values of $(a,c)$
in the closed unit square 
$\Box:= \{ (a,c):~ 0 \le a, c  \le 1 \},$
following an approach used in  Tate's thesis.
 They  apply to the two
functions
\begin{equation}\label{100}
L^{\pm}(s, a,c) := \zeta(s, a, c) \pm e^{-2 \pi i a} \zeta(s, 1-a, 1-c),
\end{equation}
which for $\Re(s) >1$ are given by the absolutely convergent series
$$
L^{\pm}(s, a,c)= \sum_{n \in \ZZ} (\sgn (n+c))^k e^{2 \pi i n a} |n+c|^{-s},
$$
with $(-1)^k = \pm$, for $k=0$ or $1$, using the convention that $\sgn(0) =0$.
 
The four-term functional
equations take the form
$$
\hat{L}^{\pm}(s, a, c) = \epsilon(\pm)e^{- 2\pi i a c} \hat{L}^{\pm}(1-s, 1-c, a),
$$
in which $\epsilon(+)=1$ and $\epsilon(-) = i$, and the hat indicates a
completion of the function at the archimedean place, which adds a 
factor $\pi^{-s/2} \Gamma(s/2)$, resp. $\pi^{-s+1)/2}  \Gamma( (s+1)/2)$ according
as sign is $+$ or $-$.
They were  first obtained  by Weil \cite{We76} in 1976.
Paper I observed that the Lerch zeta  function is discontinuous 
on  the boundary of the square, indicating a singular
nature of the values $a=0$, resp. $c=0$, and  we characterized
the behavior of the function as the boundary is
approached.

The papers  \cite{LL2} and  \cite{LL3} 
 obtained an analytic continuation of  $\zeta (s,a,c)$ 
in all three complex variables to a (nearly) maximal domain
of holomorphy; the functions are multivalued and their
monodromy was explicitly determined.
Integer values of $a$ and  nonpositive integer
values of $c$ are singular values omitted from the analytic
continuation; this explains some of the discontinuities
observed in \cite{LL1}. 
It  also observed
that these functions for fixed $s$ give
eigenfunctions of a linear partial differential 
differential operator in the $(a, c)$-variables,
$$
D_L := \frac{1}{2\pi i} \frac{\partial}{\partial \rra}
\frac{\partial}{\partial \rrc}
+ \rrc \frac{\partial}{\partial \rrc} .
$$
which states
$$ 
D_L \zeta(s, \rra, \rrc) = - s \zeta(s, \rra, \rrc).
$$
It also showed that the  monodromy functions are also
eigenfunctions of this operator. The operator
$\Delta_L = D_L + \frac{1}{2}{\bf I}$ in the Theorem above
is a shifted version of this operator, having
$ \Delta_L \zeta(s, \rra, \rrc) = - (s- \frac{1}{2})  \zeta(s, \rra, \rrc),$
see Section \ref{sec92}
for its general definiton.

The paper \cite{LL4} introduced certain
two-variable operators which were termed  ``Hecke 
operators", and studied their action on the Lerch zeta function.
 These operators have the form 
$$
\hT_m f(a,c) := \frac{1}{m} \sum_{k=0}^{m-1} 
f \left ( \frac{a+k}{m}, mc \right) ~
$$
which dilate in one direction and contract and shift in
another. When these operators are  applied term-by-term  to the Dirichlet
series representation \eqn{101},
one obtains 
$
\hT_m (\zeta)(s, a, c) = m^{-s} \zeta(s, a, c),
$
more generally one has
$$
\hT_m(L^{\pm})(s, a, c)= m^{-s} L^{\pm}(s, a, c).
$$
Thus for fixed $s$ the  Lerch zeta function is a simultaneous eigenfunction
of all $\{\hT_m: m \ge 1\}$; conversely, we show 
for all $s \in \CC$ that  these simultaneous
eigenfunction conditions plus some "twisted periodicity"
conditions and integrability conditions characterize
a two-dimensional vector space spanned by
$L^{\pm}(s, a, c)$, which includes the Lerch zeta function, see \cite[Theorem 6.2]{LL4}.


\subsection{Present paper}\label{sec12} 

In the present paper we 
view  $(a, c)$ as  real variables, and 
characterize  the symmetrized Lerch zeta function 
viewed as related to  actions on the 
Heisenberg group.

The functions themselves  can more generally be viewed as 
associated to certain automorphic representations of
a  four dimensional solvable real Lie group $H^J(\RR)$. This group  $H^J(\RR)$
is an extension of 
the real Heisenberg group $H(\RR)$ by the multiplicative group $\RR^{\ast}$.
It has faithful matrix representations given in  Appendix A.
We call  $H^J(\RR)$   the {\em sub-Jacobi group} because
it is a subgroup
of the {\em Jacobi group} $Sp(2, \RR) \rtimes H(\RR)$.
We show that the Lerch zeta function, along with  generalizations
of it that involve twisting by Dirichlet characters, called 
{\em  Lerch $L$-functions}, are associated to representations
of the sub-Jacobi  group  acting on certain
$H(\RR)$-invariant
subspaces of  $L^2(H(\ZZ)\backslash H(\RR), d\mu)$,
which carry a representation of the larger 
real Lie group $H^J(\RR)$. 
 Given a (primitive or imprimitive) Dirichlet
character $\chi$ $(\bmod~d)$, with $d$ dividing $N$, in Section \ref{sec9}
we define the {\em Lerch $L$-functions} $L_{N,d}^{\pm}(\chi, s, a, c)$ by 
 \beql{eq131}
 L_{N,d}^{\pm}(\chi, s, a,c) :=  \sum_{n \in \ZZ} \chi( \frac{nd}{N}) (\sgn(n+Nc))^k e^{2 \pi i na}
 |n+Nc|^{-s}\,
 \eeq
 in which $(-1)^k-= \pm$ with $k=0$ or $1$. We extend these functions to functions on the 
 Heisenberg group by inserting a central character
 \beql{eq132}
 L_{N,d}^{\pm}(\chi, s, a,c, z) := e^{2 \pi i Nz} L_{N,d}^{\pm}(\chi, s, a, c),
 \eeq
 and continue to call them {\em Lerch $L$-functions}, since $z=0$ recovers \eqref{eq131}. 
The case $N=d=1$ with $\chi= \chi_0$ the trivial character and $z=0$ recovers the functions
 $L^{\pm}(s, a,c)$ studied in  \cite{LL1}.
We  show that the functions $L_{N,d}^{\pm}(\chi, s, a,c,z)$ for fixed $s \in \CC$ are well-defined on the 
 associated space  $\sN_3:= \Gamma\backslash H(\RR)$, which is often called
the {\em Heisenberg nilmanifold}.

We give an  automorphic interpretation of  
the Lerch zeta function  and, more generally, of Lerch $L$-functions
twisted by Dirichlet characters, in terms of 
representations on Heisenberg modules. Here
the  Lerch zeta function plays the role of an
Eisenstein series, in the sense that it parametrizes on
the critical line $s= \frac{1}{2} + i \tau$ the continuous
spectrum of the operator $\Delta_L $, acting on the
Heisenberg module $\sH_1$.  The operator $\Delta_L$ can be identified
with a certain left-invariant differential operator on the 
Heisenberg group, which can be  can be put in the 
form $\frac{1}{4 \pi i}(XY + YX)$, where $X$ and $Y$
are standard left-invariant vector fields.
This operator encodes a  ``shift by $\frac{1}{2}$''
which moves the line of
unitarity from the imaginary axis to the critical line $Re(s) = \frac{1}{2}$, 
coming from  the Heisenberg commutation relations.
More generally, Lerch
$L$-functions play a  similar Eisenstein series role on certain submodules of 
the Heisenberg module $\sH_N$ 
for various $N$, denoted $\sH_{N, d}(\chi)$ with $d \mid |N|$
and $\chi$ a Dirichlet character $ (\bmod \, d)$.

We also give  a generalization 
of the two-variable Hecke operators studied
in \cite{LL4} to all Heisenberg modules $\sH_{N}$, given by
$$
\hT_m f(a,c,z) := \frac{1}{m} \sum_{k=0}^{m-1} 
f \left ( \frac{a+k}{m}, mc, z \right). ~
$$
These operators  act as correspondences on certain single-valued functions defined almost
everywhere on the real Heisenberg group. 
We make  a detailed  study of
the action of  these two-variable Hecke operators on the
different Heisenberg modules. In each case for each $s \in \CC$
there is a
two-dimensional space $\sE_s(\sH_{N, d}(\chi))$
of simultaneous eigenfunctions of these operators.

We note that Lerch $L$-functions 
multiplied by a suitable central character may be lifted to (discontinuous) 
functions on the real Heisenberg group, with twisted
boundary conditions reflecting the Heisenberg group action.
That is, they may be  viewed as  sections of a
vector bundle $B_N$ over the manifold
$\sX= H(\ZZ) \backslash H(\RR)$, having singularities
above certain points. \\

\subsection{Related work}\label{sec13} 

There is a very large literature on the harmonic analysis of the 
Heisenberg group, see the survey of Howe \cite{Ho80} and the
book of Thangavelu \cite{Than98}.

The  Heisenberg nilmanifold $\sN_3:= \Gamma\backslash H(\RR)$,
with $\Gamma= H(\ZZ)$ is known to provide  an appropriate setting for the action of Jacobi
theta functions $\theta(z, \tau)$, with both variables $(z,\tau) \in \CC \times{\mathfrak h}$.
The Lerch $L$-functions $L_{N, d}^{\pm}(s; a, c, z)$
give well-defined functions on $\sN_3$.
Much work was done  in the 1970's relating  theta functions,  nilmanifolds, and abelian
varieties, see for   
Auslander \cite{Au77}, Auslander and Brezin \cite{AB73}, Auslander and Tolimieri \cite{AT75}, 
Brezin \cite{Br70}, \cite{Br72},  and Tolimieri \cite{To77a}, \cite{To77b}, \cite{To78}.
See also the survey of  Auslander and Tolimieri \cite{AT82}, and  the book of 
Corwin and Greenleaf \cite{CG90}.  

Harmonic analysis on the generalized  Heisenberg nilmanifold  $\sN_n = H_n(\ZZ) \backslash H_n(\RR)$,
a $(2n+1)$-dimensional manifold, 
has been  extensively studied. 
This includes work on  $H^p$-spaces (Kor\'{a}nyi \cite{Kor73})
and  expansions
taken with respect to various differential and integral operators,
including the Radon transform (Strichartz \cite{Str91}),
the Kohn Laplacian (Folland \cite{Fol04}), the horizontal Heisenberg sublaplacian
(Thangavelu \cite{Than09}). A spectral theory for a general class of such
operators was given by Ponge \cite{Pon08}. \medskip

\paragraph{\bf Notation.}
We write the complex variable $s = \sigma + i \tau$ where
$\sigma, \tau$ are real. We reserve $t$ for a real variable,
viewed in the multiplicative group $\RR^{\ast}$.
The additive Fourier transform of a function $f(x) \in L^2(\RR, dx)$
is  $\sF f(y) := \int_{-\infty}^\infty f(x) e^{-2 \pi i xy} dx$,
following Tate~\cite{Ta50}. A Hilbert space inner product
is denoted $ \langle cdot, \cdot \rangle$, with $\langle f, g\rangle = \overline{\langle g,f\rangle}$
and $\langle f, \alpha g\rangle = \bar{\alpha}\langle f, g\rangle$. We use the 
convention that divisibility relations $d |N$ have
$d >0$ but $N$ may have a sign, and we interpret
it as $d$ divides $|N|$.\smallskip

\paragraph{\bf Acknowledgments.}
This work grew out of a joint project  with Winnie Li  \cite{LL1}-\cite{LL4}
on the Lerch zeta function. I thank her for helpful discussions.
I thank E. M. Kiral for useful references. I am very grateful to the 
careful reviewer
for many useful comments, insights and important corrections. 
Initial work on this project was done at AT\& T Labs-Research. 
The author's work related to this project has been supported by
several NSF grants, currently  grant DMS-1701576.

%
%
%
\section{Main Results}\label{sec2}
\setcounter{equation}{0}

We coordinatize the  real Heisenberg group $H(\RR)$ 
by $[a, c, z]$ using  the
(asymmetric) three-dimensional representation
$$
 H(\RR) = \{ 
[a,c,z] =
\left[ \begin{array}{ccc}
1 & c & z \\
0 & 1 & a \\
0 & 0 & 1 \\
\end{array}
\right] ~:~ a, c, z \in \RR \}.
$$
and has a normalized  two-sided Haar measure
$d\mu= da~dc~dz.$
The group law is
$$
 [a, c, z] \circ [a', c', z'] = [a+a', c+c', z + z' + ca'].
$$
The subgroup $H(\ZZ)$ consists of those $[a, c, z]$
which are all integers, and the left-coset space
$H(\ZZ)\backslash H(\RR)$ is compact with volume $1$.\smallskip

In Section \ref{sec3} we recall 
that the Hilbert space 
$\sH =L^2 (H ( \ZZ )\backslash H( \RR ), d \mu )$
with $H(\RR)$ acting by $\rho_h(F)(g) = F(gh)$ 
has a canonical ``Fourier" decomposition under the unitary characters $e^{2 \pi i Nz}$ of
the center $Z(H(\RR))$ which restrict to the identity on $H(\ZZ)$ as 
\beql{200} 
\sH :=L^2 (H ( \ZZ )\backslash H( \RR ), d \mu )   =  \bigoplus_{N \in \ZZ} \sH_N
\eeq
It is known that  $N \neq 0$ the
space $\sH_N$ is isotypic, with multiplicity $|N|$ of
the irreducible representation with central character
$e^{2 \pi i N z}$. 
We also introduce a {\em Heisenberg-Fourier} operator
\beql{R-operator}
\hR(F)([\rra, \rrc, z]):= F(\alpha(\rra, \rrc, z))= 
F(- \rrc, \rra , z - \rra \rrc ) ~.
\eeq
which is of order $4$, and which gives rise to a unitary operator 
acting on $\sH$.\smallskip

In Section \ref{sec4} we introduce and study basic properties of
the two-variable Hecke operators $\{ \hT_m: m \in \ZZ \backslash \{0\} \}$
given by
$$
\hT_m f(a,c,z) := \frac{1}{m} \sum_{k=0}^{m-1} 
f \left ( \frac{a+k}{m}, mc, z \right) ~
$$
  We show that these operators induce
bounded operators on the Hilbert space $\sH=L^2( H(\ZZ) \backslash H(\RR), d\mu)$,
and we consider their adjoint operators $\hT_m^{\ast}$ as well.
We show that the operators  can be viewed as acting on
functions on the  
real Heisenberg group $H(\RR)$ that are $H(\ZZ)$ -invariant,
and furthermore that they leave each Hilbert space  $\sH_N$ for $N \in \ZZ$ invariant.
(Lemma \ref{le31}). 
We show the adjoint operators $\hT_m^{\ast} $ are bounded and are conjugate
to the operators $\hT_m$ under the $\hR$-operator action, as follows.\medskip

%

{\bf Theorem 4.3.} 
{\em
Let $N \in \ZZ$, allowing $N=0$. Then:

(i) The adjoint operator $\hT_m^\ast$ of $\hT_m$ on $\sH$ 
leaves each  $\sH_N$ invariant   and acts on $\sH_N$ by
$$
\hT_m^\ast (F) (\rra, \rrc, z) = 
\frac{1}{m} \sum_{k=0}^{m-1} e^{2 \pi i k N\rra}
F ( ma, \frac{c+k}{m}, z ) ~.
$$

(ii) On each $\sH_N$  each $\hT_m^{\ast} $   satisfies 
with respect  to the Heisenberg-Fourier operator $\hR$
the relation
$$
\hT_m^\ast = \hR^\ast \circ \hT_m \circ \hR ~.
$$
where $\hR^{\ast} = \hR^3$.

(iii) For $N \in \ZZ$, 
and    $m \ge 1$, with $d= (m, N)$, then for $F \in \sH_N$  
$$
\hT_m^\ast \circ \hT_m (F)(\rra , \rrc, z ) = \frac{1}{m}
\sum_{\ell =0}^{d-1} F(\rra + \frac{\ell}{d}, \rrc, z)
$$
and
$$
\hT_m \circ \hT_m^\ast (F)(\rra, \rrc, z) = 
\frac{1}{m} \sum_{\ell =0}^{d-1} e^{2 \pi i (\frac{N\ell}{d}) a}F(\rra , \rrc+ \frac{\ell}{d}, z)
$$
The operators $\hT_m$ and $\hT_m^{\ast}$  commute on 
$\sH_N$ when $d= (m,~N) = 1$, and then  satisfy
$$
\hT_m \circ \hT_m^{\ast} =\hT_m^\ast  \circ \hT_m= \frac{1}{m} \bone.
$$
They do not commute on $\sH_N$ when $d >1$.
}\medskip   

The  operators $\hT_m$ do not commute with the Heisenberg group action on
these functions, but transform 
in a simple way under the a Heisenberg group automorphism  
\beql{318ccc}
\beta(t)[a, c, z] = [\frac{1}{t} a, tc, z],
\eeq
as follows.

%

{\bf Theorem 4.4.}
{\em 
For any  $h \in H(\RR)$ and $F \in \sH$
there holds
$$
\rho_h \circ \hT_m(F)(\rra, \rrc, z) = 
\hT_m \circ \rho_{\beta(m)h}(F)(\rra, \rrc, z).
$$
and
$$
\rho_h \circ \hT_m^{\ast}(F)(\rra, \rrc, z) = 
\hT_m^{\ast} \circ \rho_{\beta(1/m)h}(F)(\rra, \rrc, z).
$$
}\medskip

In Section \ref{sec5} we establish 
a decomposition of each of the spaces $\sH_N$ for $N \ne 0$ into
irreducible Heisenberg modules, compatible with
the Hecke operators above. 
This decomposition is
associated to multiplicative (Dirichlet) characters,
with the submodules being indexed by a divisor $d | N$ and a Dirichlet
character $\chi~(\bmod~d)$, written $\sH_{N,d}(\chi)$.\medskip

%

{\bf Theorem 5.5.}
{\em
{\rm (Multiplicative Decomposition)}
For $N \neq 0$, the Hilbert space $\sH_N$ has an 
orthogonal direct sum decomposition
$$
\sH_N = \bigoplus_{d| |N|} \left( \bigoplus_{\chi \in (\ZZ / d \ZZ )^\ast}
\HNd (\chi ) \right) \,.
$$
Here $\chi$ runs over all Dirichlet characters, primitive and imprimitive.
Each $\HNd (\chi )$ is invariant under the $H(\RR )$-action 
on $\sH_N$ and is an irreducible representation of $H(\RR )$ with 
central character $e^{2\pi iNz}$.
}\medskip

We call this decomposition the  {\em multiplicative decomposition} of $\sH_N$.
These spaces are studied using   {\em twisted Weil-Brezin maps} $\sW_{N, d}(\chi): L^2(\RR, dx) \to \sH_{N,d}(\chi)$ defined  
for  Schwartz functions $f(x) \in \sS(\RR)$ by
\beql{203b}
\sW_{N, d}(\chi)(f)(a, c, z) := 
\sqrt{\CNd} e^{2 \pi iN z} \sum_{n \in \ZZ}
\chi \left( \frac{nd}{N} \right) f(n+ Nc) e^{2 \pi in \rra},
\eeq
which extends to  an isometric isomorphism of Hilbert spaces (Lemma \ref{le41}).
Here $C_{N,d}= \frac{N}{\varphi(d)}$ is a normalizing constant.
The two-variable Hecke operators $\hT_m$ with $(m, \frac{N}{d})=1$
leave the spaces  $\sH_{N,d}(\chi)$ invariant, 
and the twisted Weil-Brezin map shows that they intertwine
with operators corresponding to an $\RR^{\ast}$-action on $L^2(\RR, dx)$.
(Theorem~\ref{th42}).
However the  two-variable Hecke operators  $\hT_m$ with $(m, \frac{N}{d})>1$
do not leave all the spaces $\sH_{N,d}(\chi)$ invariant. 
Theorem~\ref{th43} gives a {\em coarse multiplicative
decomposition} of $\sH_N~(N \ne 0)$
into simultaneous invariant subspaces for all Hecke
operators $\{\hT_m:~m \in \ZZ \backslash \{0\}\}$.
The factors in this decomposition are
indexed by the {\em primitive Dirichlet characters $\chi (\bmod \fe)$} for 
each $\fe~|~ N.$ \medskip

%

{\bf Theorem 5.7.}
{\rm (Coarse Multiplicative Decomposition)}
{\em  
Let $N \ne 0$  and 
to each  primitive character $\chi ~(\bmod~\fe)$
with $\fe | N$ assign the Hilbert space
$$
\sH_{N}(\chi; \fe) := \bigoplus_{{d} \atop{\fe|~d~|N}} 
\sH_{N,d}(\chi |_d).
$$
in which  $\chi|_d$ denotes the (generally imprimitive)  character $(\bmod~d)$ 
associated to $\chi (\bmod \,\fe)$.
Then the Hilbert space $\sH_N$ has
the orthogonal direct sum decomposition
$$
\sH_N = \bigoplus_{\chi,~\fe} \sH_N(\chi; \fe),
$$
in which $\chi$ runs over all primitive characters $(\bmod~\fe)$
for all $\fe | N$. Each Hilbert space $\sH_N(\chi; \fe)$ is  
invariant under all
two-variable Hecke operators $\{\hT_m: ~m \in \ZZ \backslash \{0\} \}$.
}\medskip


In  Section \ref{sec6} we study another decomposition
of Heisenberg modules $\sH_N$ into irreducible submodules,
introduced in 1973 by  Auslander and Brezin \cite{AB73}.
This decomposition is associated to
additive characters rather than multiplicative characters.
More generally, their decompositions are based on
a choice of  ``distinguished subgroup.''
We determine the  action of the two-variable Hecke operators 
on the resulting  
modules $\sH(\psi)$. These modules are usually not 
invariant under the Hecke  operator action. 
We  show that the image of one module $\sH(\psi)$ under a given
$\hT_m$ is always contained in another module $\sH(\psi')$
and that  if $(m, N)=1$ this action is a permutation
(Theorem \ref{th52}). 
The results of this section are included for comparison 
and contrast 
 with the multiplicative decomposition in Section \ref{sec5}. \smallskip

In Section \ref{sec7}  we study the $\RR^{\ast}$-action on the
multiplicative submodules $\sH_{N,d}(\chi)$.
This action gives a semidirect product with the 
$H(\RR)$-action, and defines an action on $\sH_{N,d}(\chi)$ of a particular
four-dimensional solvable real Lie group 
$H^J$, which we call the {\em sub-Jacobi group}. \medskip

%

{\bf Theorem 7.2.}
{\rm (Sub-Jacobi group action)} 
{\em
For $N \neq 0$, each positive $d | N$
and each Dirichlet character
$\chi~(\bmod~ d)$ the Hilbert space $\HNd (\chi)$ carries an irreducible
unitary representation of a four-dimensional solvable real
Lie group $H^J$ with central character $e^{2 \pi i N z}$.
Two such representations are unitarily equivalent
$H^J$-modules if and only if they have the same value of $N$.
}\medskip

This result shows  that a large subspace  of
$L^2(H(\ZZ)\backslash H(\RR), d\mu)$ carries an $H^J$-action, namely the
orthogonal complement of $\sH_0$. The remaining space  $\sH_0$,
which is invariant under the action of the one-dimensional representations
of $H(\RR)$ on $L^2(H(\ZZ)\backslash H(\RR), d\mu)$, 
does not carry an $H^J$-action.
For $N \ne 0$  the individual  spaces $\sH_{N, d}(\chi)$ in $\sH_N$ carry 
isomorphic  $H^J$-actions. The two-variable Hecke operator eigenvalues 
 $\hT_m$ for $(m, |N|) =1$ are sufficient to distinguish the spaces $\sH_N(\chi; \fe)$ in the coarse multiplicative decomposition of $\sH_N$
given in Theorem \ref{th43} from each other, because they 
uniquely determine the associated primitive character $(\chi, \sf)$,
see Theorem \ref{th99} (1) below.
The  individual Hilbert spaces $\sH_{N, d}(\chi)$ may be distinguished  from each other
by  their associated 
spectral measures 
associated to the operator $\Delta_L$ described in Section \ref{sec9},
which  determines a particular pair of  Lerch $L$-functions $L_{N,d}^{\pm}(\chi, s, a, c, z)$,
as defined in Section \ref{sec9}.



In  Section \ref{sec8}  we study the action of the automorphism
$$
\alpha(\rra, \rrc, z) := (-\rrc, \rra, z - \rra\rrc)
$$
of the Heisenberg group acting as an operator $\hR$ 
on the Heisenberg modules $\sH_N$.
On $\sH_1$ under the Weil-Brezin map $\sW= \sW_{1,1}$ 
(defined in Section \ref{sec52}) this operator intertwines 
with the additive Fourier transform.
We show
that its action on $\sH_N$ intertwines with the 
dilation $\hU(N)$ ( see \eqref{601}) composed with the additive Fourier
transform, 
and that this intertwining action mixes various
of the spaces $\sH_{N,d}(\chi),$ over all 
different (imprimitive) $\chi (\bmod~d)$  associated to  
a fixed  primitive character $\chi (\bmod~\fe)$ with
$\fe | d | N$. We show that the $\hR$-operator respects
the coarse multiplicative decomposition of $\sH_N$
in the following way. \medskip

%

{\bf Theorem 8.3.} 
{\em
For $N \ne 0$, the Heisenberg-Fourier operator $\hR$
restricted to the invariant subspace $\sH_N$ is a 
unitary operator $\hR_N$ which acts to permute the
Hilbert spaces $\sH_N(\chi; \fe)$ given by the
coarse multiplicative decomposition of $\sH_N$.
It satisfies 
$$
 \hR_N( \sH_N(\chi; \fe)) = \sH_N(\bar{\chi}; \fe).
 $$
 }\smallskip

In Section \ref{sec9}  we define Lerch $L$-functions 
$L_{N,d}^{\pm}(\chi, s, \rra, \rrc, z)$
and show that they are analogous to Eisenstein series
in three distinct ways:
\begin{enumerate}
\item[(i)]  as giving a continuous spectral
decomposition of a Laplacian-like operator, 
\item[(ii)] as being a simultaneous
eigenfunction of a family $\{\hT_m: m \ge 1 \,\,\mbox{with} \, (m, N) =1 \}$ of 
Hecke-like operators,
\item[(iii)] as satisfying suitable functional equations
relating $s$ to $1-s$.
\end{enumerate}
In Section \ref{sec91}  we define $L_{N,d}^{\pm}(\chi, s, \rra, \rrc, z)$  as functions on the Heisenberg group,
for $0 < \Re(s) < 1$, via
$$L_{N,d}^{\pm}(\chi, s, \rra, \rrc, z) = e^{2\pi iNz} L_{N,d}^{\pm}( \chi, s, a, c),$$
where $L_{N,d}^{\pm}( \chi, s, a, c)$ is given by
 the conditionally convergent series,
$$
L_{N,d}^{\pm}(\chi, s, a, c) := 
 \sum_{n \in \ZZ} \chi(\frac{nd}{N})(sgn(n+N\rrc))^{k}e^{2 \pi i n \rra}~|n+N\rrc |^{-s}.
$$
In Theorem \ref{th91a} we give a meromorphic continuation
of these functions  in the  $s$-variable and show they satisfy twisted-periodicity
formulas  making them well-defined functions
  on the Heisenberg nilmanifold  $\sN_3$.

 In  Section \ref{sec92} we treat
Lerch $L$-functions
 as simultaneous eigenfunctions of the differential operators
$$
\Delta_L =\frac{1}{2\pi i} \frac{\partial}{\partial \rra} 
\frac{\partial}{\partial \rrc} + N \rrc  \frac{\partial}{\partial \rrc}
+ \frac{N}{2}\bone
$$
and $\frac{\partial}{\partial z}$, and show that the two 
families of functions
$L_{N,d}^{\pm}( \chi, \frac{1}{2} + i\tau, a, c.z)$,
viewing $\tau$ as a parameter, $-\infty < \tau < \infty$,
are a complete set of generalized eigenfunctions for $\Delta_L$,
with
$$
\Delta_L L_{N,d}^{\pm}( \chi, \frac{1}{2} + i\tau, a, c, z) =
-i\tau  L_{N,d}^{\pm}( \chi, \frac{1}{2} + i\tau, a, c, z),
$$
on a suitable dense domain inside the Hilbert space $\sH_{N, d}(\chi).$
A main result of this paper concerns the spectral interpretation of
Lerch $L$-functions as Eisenstein series. 
In stating  it we use
the twisted Weil-Brezin maps $\sW_{N, d}(\chi): L^2(\RR, dx) \to \sH_{N,d} (\chi)$
given in Definition \ref{de52} and 
Lemma 5.3.

\medskip

%

{\bf Theorem 9.5.} 
 {\rm (Eisenstein Series Interpretation of Lerch $L$-functions)}\\
{\em Let $N \ne 0$ and $d \ge 1$ with $d \mid N$.

$~~~~~$ (1) Consider  the unbounded
 operator $\Delta_L= \frac{1}{2\pi i} \frac{\partial}{\partial \rra} 
\frac{\partial}{\partial \rrc} + N \rrc  \frac{\partial}{\partial \rrc}
+ \frac{N}{2}$
 on the dense domain 
$$
 \sD_{N,d}(\chi) := \sW_{N,d}(\chi) (\sD) \,  
 $$
 in the Hilbert space $\sH_{N,d}(\chi)$, in which $\sD$ denotes  the maximal domain for 
$D= x \frac{\partial}{\partial x}+ \frac{1}{2} \bone$
on $L^2(\RR, dx)$. 
The operator $(\Delta_L, \sD_{N,d}(\chi))$
commutes with all elements of the unitary group 
$\{ \hV(t): ~ t \in \RR^{\ast} \}$
defined by \eqref{601} and \eqref{602}. 

(2) The operator
$(\Delta_L, \sD_{N,d}(\chi))$ 
is skew-adjoint on $\sH_{N,d}(\chi)$,
and its associated spectral multiplier
function  on $L^2(\ \ZZ/ 2\ZZ \oplus \RR, d\tau)$
is $a_0(\tau)= -i \tau$ and  $a_1(\tau_1) = -i \tau_1$.

(3) The two families of Lerch $L$-functions 
$L_{N,d}^{\pm}( \chi, \frac{1}{2} + i\tau, a, c,z)$,
parameterize the (pure) continuous spectrum 
of  $(\Delta_L, \sD_{N,d}(\chi))$ on $\sH_{N,d}(\chi)$, 
giving   a complete set of generalized eigenfunctions,
as $\tau$ varies  over $\RR$.
All functions  $F(a, c, z)$ in the dense subspace  $\sS(\sH_{N,d}(\chi))$ have
a convergent spectral representation 
\begin{eqnarray*}
F(a, c, z)  &=&  \frac{1}{4\pi} \int_{-\infty}^\infty 
\hat{F}^{+}(\frac{1}{2}+ i\tau)L_{N,d}^{+}( \chi, \frac{1}{2} - i\tau, a, c.z)
d \tau  \\
&&\quad\quad\quad+
\frac{1}{4\pi} \int_{-\infty}^\infty\hat{F}^{-}(\frac{1}{2}+ i\tau)
L_{N,d}^{-}( \chi, \frac{1}{2} - i\tau, a, c.z) d\tau, 
\end{eqnarray*}
in which 
$\hat{F}^{+}(s) = \sM_0( \WNd(\chi)^{-1}(F))(s)$
and $\hat{F}^{-}(s) = \sM_1( \WNd(\chi)^{-1}(F))(s)$,
where $\sM_{0}$ and $\sM_{1}$  are two-sided Mellin transforms defined by \eqref{620}. 
}\medskip

The domain $\sD$ is the maximal domain for the unique skew-adjoint
extension of the operator $D= x \frac{\partial}{\partial x}+ \frac{1}{2} \bone$
from the dense subspace of Schwartz functions $\sS(\RR)$ inside $L^2(\RR, dx)$,
see Section \ref{secB3}.

The operator $\Delta_L$ is the image
under the twisted Weil-Brezin map $\WNd(\chi)$ of   a scaling $ND$ of the dilation invariant Laplacian
$D=  x \frac{d}{dx} + \frac{1}{2}$
acting on $L^2(\RR_{>0}, dx)$, which itself corresponds to
the dilation  invariant Laplacian
$\tilde{D} = x \frac{d}{dx}$ acting on $L^2(\RR_{>0}^{*}, \frac{dx}{x})$. 

 In Section \ref{sec93} we show   these functions  can be
characterized  as simultaneous joint
eigenfunctions for the set  of two-variable Hecke
operators $\hT_m$ , viewing them as
tempered distributions using the generalized Weil-Brezin maps.\medskip
We introduce a notion of {\em Lerch $L$-distribution}
 and observe
the distributional variant agrees with Lerch $L$-functions in the critical strip
 $0 < \Re(s) < 1$, where these distributions correspond to  locally $L^{1}$-functions.
\medskip
 
%

{\bf Theorem 9.9.}
{\rm ( Hecke operator  tempered distribution eigenspace)}
{\em
Let $N$ be a nonzero integer, 
and $\chi$ be a Dirichlet character $(mod~d)$, with
$d | N$. 

(1) For  each fixed $s \in \CC$,
let $\sE_s(\sH_{N, d}(\chi))$ be 
the vector space  of tempered distributions 
$\Delta \in \sS'(\sH_{N, d}(\chi))$  such that
$$
\hT_m(\Delta) = \chi(m)m^{-s} \Delta, ~~~\mbox{for~all}~~ 
 m \ge 1~~ \mbox{with}~~(m, N) = 1.
 $$
Then $\sE_s(\sH_{N, d}(\chi))  $ is a two-dimensional
vector space, and
is spanned by an even homogeneous tempered distribution of homogeneity
$|t|^{1-s}$ and an odd homogeneous tempered distribution of
homogeneity $\sgn(t)|t|^{1-s}.$ 

(2) For all non-integer $s \in \CC$ the two Lerch
$L$-distributions $L_{N,d}^{\pm}({\chi}, s, \rra, \rrc, z)$
are nonzero  even and odd homogeneous distributions spanning
 $\sE_s(\sH_{N, d}(\chi))$, respectively. 
For $0 < \Re(s) < 1$ these two distributions
are induced by the Lerch $L$-functions
$L_{N,d}^{\pm}({\chi}, s, a, c, z)$, which  
both lie in $L^1( \sH_{N, d}(\chi))$.
}\medskip

In Section \ref{sec94} we show these functions satisfy suitable
functional equations taking $s \mapsto 1-s$. These
functional equations are more complicated than that
for Dirichlet $L$-functions.\medskip

%

{\bf Theorem 9.10.} 
{\rm (Generalized Lerch  Functional Equations)}
{\em
Suppose that $N \ne 0$.
Let $\chi$ be a primitive character $(\bmod~\fe)$ and suppose
that $\fe | d$ and $d|N$, and let $ \chi |_d$ denote the (generally imprimitive)
character $(\bmod~d)$ associated to $\chi$.
Then for  $0 < \Re(s) < 1$ 
 fixed $a, c \in \RR \smallsetminus \ZZ$) the two
 Lerch $L$-functions $L_{N,d}^{\pm}(\chi |_d, s, a, c, z)$ 
associated to  $\sH_{N,d}(\chi |_d)$ satisfy
the functional equations
$$
\hR (L_{N,d}^{\pm})(\chi |_d, 1-s, a, c, z) =  
\chi(-1) \tau(\chi) |N|^{s-1} \gamma^{\pm}(s)  
\left(\sum_{ \tilde{d} | N} C_{N, d}(\tilde{d}, \chi)
L_{N,\tilde{d}}^{\pm}(\bar{\chi}|_{\tilde{d}}, s, a, c, z) \right),
$$
in which $\hR f(a, c, z) = f(-c, a, z - ac)$
and $\gamma^{\pm}(s)$ are Tate-Gelfand-Graev gamma
functions, and the coefficients $C_{N, d}(\tilde{d}, \chi)$
vanish whenever $\fe \nmid \tilde{d}$.
}\medskip
It is possible to extend the functional equations to $s \in \CC$ by analytic continuation
as entire functions. \smallskip

 These functional equations  correspond to
the action of the $\hR$ operator studied in Section \ref{sec8}, 
which mixes  together Lerch $L$-functions 
involving all imprimitive
characters coming from a fixed primitive
character $\chi~ (\bmod~e)$ of conductor $e | N$.
The Tate-Gelfand-Graev gamma functions are defined
in \eqref{eq927}.

In Section \ref{sec10} we make  concluding and summarizing remarks.\smallskip

There are two appendices (Sections \ref{sec11} and \ref{sec120}).
Appendix A gives facts about the asymmetric form
of the Heisenberg group used in this paper, and its
relation to the symmetric Heisenberg group. It also
discusses properties of the sub-Jacobi group $H^J$.
Appendix B discusses dilation-invariant operators 
and their spectral theory, particularly the operator
$x \frac{d}{dx} + \frac{1}{2}$ on $L^2(\RR,dx)$,
based on work of Burnol.

%
%
%
\section{Preliminary  Results }\label{sec3}
\setcounter{equation}{0}

We recall basic facts about 
the Heisenberg nilmanifold $X=  H(\ZZ) \backslash H(\RR)$ and functions on it. 
It was studied by Brezin \cite{Br70} and Auslander and Brezin \cite{AB73}.
Note that $H(\ZZ)$  is not a lattice nilpotent group in $H(\RR)$, in the sense that $\log H(\ZZ)$
is not a subgroup in the vector space $L( H(\RR)$, its real Lie algebra.  
However  is of index $2$ in the  lattice subgroup
$$
 \Gamma = \{ 
[a,c,z] =
\left[ \begin{array}{ccc}
1 & \ZZ & \frac{1}{2} \ZZ \\
0 & 1 & \ZZ \\
0 & 0 & 1 \\
\end{array}
\right]  \}.
$$
of $H(\FF)$ see Auslander \cite[p.233]{Au73a}. Lattice nilpotent groups were studied
in Moore \cite{Mo65}. 
Useful references are  Auslander and Tolimieri \cite{AT75} and Thangavelu \cite{Than09}. 
%
\subsection{Decomposition of $H(\ZZ) \backslash H(\RR)$}\label{sec31}

We may view  the Hilbert space 
$\sH =L^2 (H ( \ZZ )\backslash H( \RR ), d \mu )$ 
 as a space of (equivalence classes of) measurable functions on $H( \RR )$ 
satisfying the periodicity condition
\beql{eq34}
F(\gamma g) = F(g) \quad\mbox{for all} \quad \gamma \in H( \ZZ ) ~,
\eeq
and square-summable on a fundamental domain of $H(\ZZ)\backslash H(\RR)$.
The Hermitian inner product on this Hilbert space is
\begin{eqnarray}\label{eq35}
\langle F_1, F_2 \rangle & := & 
\int_{H(\ZZ)\setminus H(\RR )} F_1 (g) \overline{F_2 (g)} dg \nonumber \\
& = & \int_0^1 \int_0^1 \int_0^1
F_1 ( \rra, \rrc , z) \overline{F_2 (\rra, \rrc , z)}
d \rra d \rrc dz ~.
\end{eqnarray}
The space of continuous bounded functions $\Cbdd  (H( \RR )) \cap L^2 ( H( \ZZ ) \backslash H(\RR ), d \mu )$
 is dense in $L^2 (H(\ZZ )\backslash H(\RR ), d \mu )$,
 and the formulas below hold for them.

The group $H(\RR )$ acts on the Hilbert space $\sH=L^2 ( H( \ZZ ) \backslash H( \RR ), d \mu )$ 
on the right as
\beql{eq36}
\rho_h (F)(g) = F(gh ) \quad g, h \in H(\RR ) ~.
\eeq
Indeed  
$\rho_{h_1} \circ \rho_{h_2}(F)(g) = \rho_{h_2}(F)(g h_1)=
F( (g h_1) h_2)= F (g (h_1 h_2))= \rho_{h_1h_2}(F)(g).
$
Thus if $g= [a, c, z]$ and $h=[a', c', z']$ then
$$
\rho_h(F)(a, c, z) = F(a+a', c+c', z+ z' + c a').
$$

Since $[0,0,1] \in H(\ZZ)$ elements of the Hilbert space $\sH$ satisfy the periodicity property
\beql{eq36bb}
F([a, c, z]) = F([0,0,1] \circ [a, c, z] ) = F(a, c, z+1).
\eeq
in the $z$-variable, for almost all $[a, c, z]$.  Thus we can decompose $\sH$ into Fourier
eigenspaces in the $z$-direction.
obtaining  the
orthogonal decomposition
\beql{eq37}
\sH = \bigoplus_{N \in \ZZ} \sH_N ~,
\eeq
in which $\sH_N$ consists of those  $L^2$-functions having 
$e^{2 \pi i N z}$ as central character; that is, 
\beql{eq38a}
F( \rra, \rrc , z) = e^{2 \pi i Nz} F( \rra, \rrc, 0) ~,
\eeq
almost everywhere. 
The conditions for membership $F( \rra, \rrc, z) \in \sH_N$ 
can be expressed in terms of values $z=0$, and  given in
terms of the left action of $H(\ZZ)$, as follows. Such a function
must satisfy  (almost everywhere) 
\begin{eqnarray}\label{eq38b}
F(\rra +1, \rrc, 0) & = & F([1,0,1]\circ [\rra, \rrc, 0]) 
 \nonumber\\
&=& 
F( \rra, \rrc, 0),
\end{eqnarray}
a condition which  is independent of $N$. 
Such a function must also satisfy (almost everywhere) 
\begin{eqnarray}
\label{eq38c}
F( \rra, \rrc +1, 0) & = &  F([1,1,0] \circ [\rra, \rrc, -\rra]) 
\\
&=& e^{-2 \pi iN\rra} F( \rra, \rrc, 0) ~, \nonumber
\end{eqnarray}
a condition which depends on $N$.
Since $H(\ZZ)$ is generated by 
$[0,0,1], [0, 1, 0], [1,0,0]$ any function on $\sH(\RR)$
satisfying properties 
\eqn{eq38a}, \eqn{eq38b}, \eqn{eq38c} almost everywhere
will be invariant under the left action of $H(\ZZ)$ on $\sH_N$ almost everywhere.

It is known that for  $N \ne 0$ the  action of $H(\RR)$ on the space $\sH_N$  decomposes into
$|N|$ copies of the unique (infinite-dimensional) unitary irreducible representation $\pi_N$
of $H(\RR)$ having central character $e^{2 \pi i Nz}$( see \cite[Theorem 10.4.2]{DeiEch09}, also \cite{Br70}, \cite{AT75}).

For $N=0$ the central character is trivial and  the functions in $\sH_0$ are constant
in the $z$-direction, so the Hilbert space
$ L^2 (\RR^2 / \ZZ^2, d \rra d \rrc )$ can serve as
 a representation module for $H( \RR )$.
It has a natural orthogonal basis
\beql{eq39}
\phi_{mn} ( \rra, \rrc, z) :=
e^{2 \pi i(m \rra + n \rrc )}, \quad
(m,n) \in \ZZ^2 ~,
\eeq
which decomposes it into one-dimensional representations of $H(\RR)$.
We will need some function spaces sitting inside $\sH$

%
\subsection{Heisenberg-Fourier operator}\label{sec32}

The map  $\alpha: H(\RR) \to H(\RR)$ 
given by 
\beql{eq310b}
\alpha ([\rra, \rrc, z]) := [-\rrc, \rra , z - \rra \rrc ].
\eeq
is an automorphism of $H(\RR)$, i.e. $\alpha(g) \circ \alpha(h) = \alpha(gh)$.
It  is of order $4$. 
The map $\alpha$ induces an operator acting on functions
in $\sH$.  
We let  $\Cbdd (H(\RR ))$ denote the set of bounded continuous functions 
on the real Heisenberg group $H(\RR )$. This space includes all $H(\ZZ)$-periodic
continuous functions, which form a dense subset of functions in $\sH$.

%
\begin{defi} ~\label{de31}
{\em 
The {\em Heisenberg-Fourier operator} $\hR: C_{bdd}^{0}(H(\RR) ) \to C_{bdd}^0(H(\RR))$
is given by 
\beql{eq311}
\hR (F) (\rra, \rrc , z) := F(\alpha(\rra, \rrc, z))= 
F(- \rrc, \rra , z - \rra \rrc ) ~.
\eeq
It satisfies $\hR^4= \bone$ on this domain.
}
\end{defi}
\smallskip

%

\begin{lemma}\label{le42a}
The operator $\hR$ restricted to  
$\Cbdd (H(\RR )) \cap L^2 (H(\ZZ) \backslash H(\RR ), d \mu )$ 
extends uniquely to a unitary operator on $\sH$,
also denoted $\hR$.

(i) The unitary operator $\hR$  leaves each  space $\sH_N$
invariant, including  $N=0$.  The restriction $\hR_N$  
of this operator  to  the domain $\sH_N$ is given by
\beql{eq314}
\hR_N (F)(\rra, \rrc, z ) = e^{- 2 \pi i N \rra \rrc} F(- \rrc, \rra, z ) ~,
\qquad F \in \sH_N.
\eeq

(ii) The operator $\hR$  satisfies $\hR^4 =\bone$, and has adjoint operator
\beql{eq312}
\hR^\ast (F) (\rra, \rrc, z ) =
\hR^{-1} (F)(\rra, \rrc, z ) = F( \rrc, - \rra , z+ \rra \rrc ).
\eeq

(iii) The unitary operator $\hJ := \hR^2$  is an involution, is self-adjoint, and is  given by 
\beql{eq313}
\hJ(F)(\rra, \rrc, z)   =  F(-a, -c, z ).
\eeq
\end{lemma}

\begin{proof}
In proving (i) below we will check that $\hR$ applied to  to continuous functions on $\sH_N$ maps them to
$\sH_N$.  
Assuming this, since such functions form a dense
subspace of $\sH_N$, and  the map $\hR$ uniquely extends to an isometry of $\sH_N$.
Consequently it extends uniquely to an isometry of $L^2(H(\ZZ) \backslash H(\RR))= \oplus_{N \in \ZZ} \sH_N$,
which is therefore unitary. 

(i) A continuous function $F(a, c, z) \in \sH_N$ if and only if 
$F(a, c, z) = F(a, c, 0) e^{2 \pi i N z}$, together with relations
\begin{eqnarray*} 
F(a+1, c, z ) &=& F(a, c, z)\\
F(a, c+1, z) &=& e^{-2 \pi i N a} F(a, c, z).
\end{eqnarray*}
For $F \in \sH_N$ these relations yield
$$
\hR(F) (a, c, z) = F(-c, a, z- ac) = e^{- 2 \pi i Nac} F(-c, a, z),
$$
which is \eqref{eq312}.
We now check $\hR(F)(a, c, z) \in \sH_N$. We have
$$
\hR(F)(a, c, z) = F(-c, a, z- ac) = e^{2 \pi i N z} F(-c, a, -ac) = e^{2 \pi i N z}\hR(F)(a, c,0).
$$
In addition, using these relations, we obtain
\begin{eqnarray*} 
\hR(F)(a+1, c, z) &=& F(-c, a+1, z- (a+1)c)\\
&=& e^{-2\pi i N(-c)} F(-c, a, z- ac -c)\\
&=& e^{2 \pi i N c}  e^{ 2\pi i N(-c)} F(-c, a, z-ac) = \hR(F)(a, c, z),
\end{eqnarray*}
and
\begin{eqnarray*} 
\hR(F)(a, c+1, z) &=& F(-c-1, a, z- a(c+1)) \\
&=& F(-c, a, z- ac -a)\\
&=& e^{2 \pi i N (-a)} F(-c, a, z-ac) = e^{-2\pi i Na} \hR(F)(a, c, z).
\end{eqnarray*}
Thus $\hR(F) \in \sH_N$. 
Finally  recall that  the Hilbert space norm of $F(a, c, z)$ on $L^2(H(\ZZ) \backslash H(\RR), d \mu)$  is
$$
||F||^2 = \int_{0}^1\int_{0}^1\int_{0}^{1} |F(a, c, z)|^2 da dc dz.
$$
Now for $F \in \sH_{N}$, we obtain using the relations
\begin{eqnarray*}
||\hR(F)||^2 &= & \int_{0}^1\int_{0}^1\int_{0}^1|F(-c, a, z-ac)|^2 da dc dz \\
&= & \int_{0}^1\int_{0}^1\int_{0}^1 |F(1-c, a, z)|^2 da dc dz =||F||^2.
\end{eqnarray*}
 It follows that $\hR$ is an isometry, which  is onto since $\hR$  is invertible.

(ii), (iii)  The relation $\hR^4 = \bone$ on $L^2(H(\ZZ) \backslash H(\RR), d \mu)$ is inherited from its
validity on the dense subspace $\Cbdd (H(\RR )) \cap L^2 (H(\ZZ) \backslash H(\RR ), d \mu )$.
The adjoint $\hR^{\ast}$ equals $ \hR^{-1}$ since it is unitary, and $\hR^{-1} = \hR^3$. The formulas \eqref{eq312}
and \eqref{eq313} follow by explicit calculation.
\end{proof}

\paragraph{\bf Remark.} 
The Heisenberg-Fourier  operator $\hR_1$ defined on  $\sH_1$ is related to the  Fourier transform,
in the sense that it intertwines with
the (additive) Fourier transform $\sF$  on $L^2 (\RR , dx )$
under  the Weil-Brezin map 
$\sW = \sW_{1,1} (\chi_0) : L^2 (\RR , dx ) \to \sH_1$, defined in  Section \ref{sec52}.
This explains our nomenclature. 
The operator $\hR_N$ on $\sH_N$ has a more complicated relation to
the Fourier transform, which is computed in  Section  \ref{sec8}.

%
%
%
\section{Two-Variable Hecke Operators}\label{sec4}
\setcounter{equation}{0}

We define and study two-variable Hecke operators on
the Heisenberg group. The definition extends the
two-variable Hecke operators studied in \cite{LL4}
by inserting the third Heisenberg variable $z$, which
is left unchanged under this action. 
In consequence all
the Hilbert space domains treated in that paper
carry over to the spaces $\sH_N$.

%
%
%

\subsection{Hecke operator definition} \label{sec41}

We now define Hecke operators on 
suitable spaces of piecewise continuous functions. the space of bounded
continuous functions  $\Cbdd (H(\RR ))$. 
%

\begin{defi}\label{de41}
{\em For all nonzero integers $m$  
we define the {\em two-variable Hecke operator}
\beql{eq31}
\hT_m (F)( \rra, \rrc , z) :=
\frac{1}{|m|} \sum_{j=0}^{|m|-1} 
F\left( \frac{\rra+j}{m} , m \rrc, z \right) ~,
\eeq
Here $\hT_m : \Cbdd (H(\RR )) \to \Cbdd (H(\RR ))$,
and the two variables in the name refer to  variables $(\rra, \rrc)$, 
noting that the action on the $z$-variable is trivial.
}
\end{defi} 

It is easy to compute that
\beql{eq32}
\hT_m \circ \hT_n = \hT_n \circ \hT_m = \hT_{mn}
\eeq
as operators on $\Cbddz (H(\RR ))$.
Indeed setting $G (\rra , \rrc , z ) = \hT_m F( \rra, \rrc , z )$, we have
\begin{eqnarray}\label{eq33}
\hT_n \circ \hT_m (F)(\rra, \rrc , z) & = & \frac{1}{|n|} 
\sum_{k=0}^{|n|-1} \hT_m(F)\left( \frac{\rra + k}{n}, n \rrc, z \right) 
\nonumber \\
& = & \frac{1}{|n|} \sum_{k=0}^{|n|-1} \frac{1}{|m|} \sum_{j=0}^{|m|-1} 
F\left( \frac{\frac{\rra +k}{n} +j}{m} , nm \rrc, z \right) \nonumber \\
& = & \frac{1}{|mn|} \sum_{l=0}^{|mn|-1} 
F\left( \frac{\rra +l}{mn}, mn \rrc, z \right) \nonumber \\
& = & \hT_{mn}(F)(\rra, \rrc, z ) ~.
\end{eqnarray}

%

\begin{lemma}\label{le31}
The operators $\{\hT_m : m \in \ZZ \backslash \{0\} \}$ on 
$\Cbdd (H(\RR )) \cap L^2 (H(\ZZ) \backslash H(\RR ), d \mu )$ 
extend uniquely to bounded operators 
$\{ \hT_m : m \in \ZZ \backslash \{0\} \}$ on 
$L^2 (H( \ZZ ) \backslash H( \RR ), d\mu )$.

(i) Each $\hT_m$ leaves every  space $\sH_N$
invariant, including the case $N=0$. 

(ii)  On $\sH_N$ these operators satisfy the relations
\beql{eq310}
\hT_m \circ \hT_n = \hT_n \circ \hT_m = \hT_{mn},
\eeq
for all $m,n \in \ZZ$. 

(iii) With respect to the involution $\hJ=\hR^2$
the  operators $\hT_m$ satisfy
\beql{eq313a}
\hT_{-m} = \hT_m \circ \hR^2.
\eeq
\end{lemma}

\begin{proof}
The space $\Cbdd (H( \RR )) \cap \sH$ is dense in $\sH=\oplus_{N} \sH_N$ so 
any continuous extension of $\hT_m$ to all of $\sH$ is unique.
We have $\| \hT_m \| \le 1$ on $\Cbdd (H(\RR )) \cap \sH$, 
so a continuous extension exists.

(i) A continuous function $F(a, c, z) \in \sH_N$ if
$F(a, c, z) = F(a, c, 0) e^{2 \pi i N z}$, with
\begin{eqnarray*} 
F(a+1, c, 0 ) &=& F(a, c, 0)\\
F(a, c+1, 0) &=& e^{-2 \pi i N a} F(a, c, 0).
\end{eqnarray*}
For a continuous  $F(a, c, z) \in \sH_N$ we have
\begin{eqnarray*}
\hT_m(F) (a+1, c, 0) &=& \frac{1}{|m|} \sum_{j=0}^{|m|-1} F(\frac{a+1+j}{m}, mc, 0) \\
&=& \frac{1}{|m|} \big( F(\frac{a}{m} +\frac{|m|}{m}, mc,0) + \sum_{j=1}^{|m|-1} F(\frac{a+j}{m}, mc, 0)\big)\\
&= & \frac{1}{|m|}  \sum_{j=0}^{|m|-1} F(\frac{a+j}{m}, mc, 0)\\
&=& \hT_m(F) (a, c, 0),
\end{eqnarray*}
where \eqref{38b} was used in the second to last line. 
\begin{eqnarray*}
\hT_m(F) (a, c+1, 0) &=& \frac{1}{|m|} \sum_{j=0}^{|m|-1} F(\frac{a+j}{m}, mc+m, 0) \\
&=& \frac{1}{|m|} \sum_{j=0}^{|m|-1} e^{-2\pi i mN (\frac{a+j}{m})}F(\frac{a+j}{m}, mc, 0)\\
&=& e^{-2\pi i Na} \frac{1}{|m|} \sum_{j=0}^{|m|-1} e^{-2\pi i jN}F(\frac{a+j}{m}, mc, 0)\\
&=& e^{-2\pi i Na} \hT_m(F)(a, c, 0),
\end{eqnarray*}

where \eqref{eq38c} was used in the second  line. We conclude that $\hT_m(F) \in \sH_N$.
The bounded continuous functions in $\sH_N$ are dense in $\sH_N$, so (i)  holds on all of $\sH_N$
by boundedness of the operator $\hT_m$.

(ii) The commutativity relation (\ref{eq310}) is inherited from (\ref{eq32}).

(iii) The relation \eqn{eq313a} is verified using 
\begin{eqnarray}
\hT_{-m}(F)(\rra, \rrc, z) & = & 
\frac{1}{|m|} 
\sum_{j=0}^{|m|-1} F\left( \frac{\rra+j}{-m} , -m \rrc, z \right)
\nonumber \\
& = & 
\frac{1}{|m|} 
\sum_{j=0}^{|m|-1} \hR^2(F)\left( \frac{\rra+j}{m} , m \rrc, z \right)  \nonumber \\
&=&  \hT_m \circ \hR^2 (F)(a, c, z)
\end{eqnarray}
\end{proof}

%
%
%

\subsection{Adjoint two-variable Hecke operators}\label{sec42}

Since the Hecke operators act as bounded operators on each $\sH_N$,
they have well defined adjoint Hecke operators  $\hT_m^{\ast}$.
In the next result we allow all integers $N$,  allowing $N=0$,
using  the convention that  the g.c.d. $(m, 0) =m$.

%

\begin{theorem}\label{th31}
{\em (Adjoint two-variable Hecke operators)} 

(i) The adjoint operator $\hT_m^\ast$ of $\hT_m$ on $\sH$ 
leaves each space $\sH_N$ invariant   and acts on $\sH_N$ by
\beql{eq316}
\hT_m^\ast (F) (\rra, \rrc, z) = 
\frac{1}{m} \sum_{k=0}^{m-1} e^{2 \pi i k N\rra}
F ( ma, \frac{c+k}{m}, z ) ~.
\eeq

(ii) On each $\sH_N$  each $\hT_m^{\ast} $   satisfies 
with respect  to the Heisenberg-Fourier operator $\hR$
the relation
 \beql{eq315}
\hT_m^\ast = \hR^\ast \circ \hT_m \circ \hR ~.
\eeq
where $\hR^{\ast} = \hR^3$.

(iii) For $N \in \ZZ$, 
and    $m \ge 1$, with $d= (m, N)$, then for $F \in \sH_N$  
\beql{eq317}
\hT_m^\ast \circ \hT_m (F)(\rra , \rrc, z ) = \frac{1}{m}
\sum_{\ell =0}^{d-1} F(\rra + \frac{\ell}{d}, \rrc, z)
\eeq
and
\beql{eq318}
\hT_m \circ \hT_m^\ast (F)(\rra, \rrc, z) = 
\frac{1}{m} \sum_{\ell =0}^{d-1} e^{2 \pi i (\frac{N\ell}{d}) a}F(\rra , \rrc+ \frac{\ell}{d}, z)
\eeq
The operators $\hT_m$ and $\hT_m^{\ast}$  commute on 
$\sH_N$ when $d= (m,~N) = 1$, and then  satisfy
\beql{318a}
\hT_m \circ \hT_m^{\ast} =\hT_m^\ast  \circ \hT_m= \frac{1}{m} \bone.
\eeq 
These operators do not commute on $\sH_N$ when $d >1$.
\end{theorem}

\paragraph{\bf Remarks.}(1)  In the case $N=0$, we use the convention that $\gcd(m, 0) =m$.

(2) A bounded operator $M$ on a Hilbert space is {\em} normal if $M^{\ast} M = M M^{\ast}$. 
The recult (iii)  above implies  that $\hT_m$ is not a normal operator on $\sH_N$ when 
 $d= (m, N) \ge 2$.

\begin{proof}
(i) We verify \eqn{eq316} on $\sH_N$, and later use it
to prove \eqn{eq315}. Let $\tilde{\hT}_m( F)$ denote the right side
of \eqn{eq316}, and we are to show that $\tilde{\hT}_m(F) = \hT_m^{\ast}(F)$.

We first show the adjoint $\hT_m^{\ast}$ leaves each $\sH_N$ invariant.
Suppose $F \in \sH_N$, $g \in \sH_{N'}$.
If $N \neq N'$ then
\beql{eq319}
\left(  F, \hT_{m} (G) \right) =
\int_0^1 \int_0^1 F( \rra, \rrc, 0 )
\overline{\hT_m (G) ( \rra, \rrc, 0)}
\int_0^1 e^{2 \pi i(N-N')z} dz=0\,.
\eeq
Since
$$
\langle \hT_m^\ast (F), G \rangle = \langle F, \hT_m (G) \rangle =0,
$$
it follows that  $\hT_m^\ast (F)$ is orthogonal to all $G \in \sH_{N'}$, 
$N' \neq N$.  Thus  $\hT_m^\ast (F)  \in \sH_N$, as required.

It therefore suffices to determine $\hT_m^\ast$  on $\sH_N$.
For $F,G \in \sH_N$ we integrate out the $z$-variable to obtain
$$\langle F, \hT_m (G) \rangle = \frac{1}{m}
\int_0^1 \int_0^1 F( \rra, \rrc, 0)
\left( \overline{\sum_{j=0}^{m-1} 
G \left( \frac{\rra + k}{m}, m \rrc, 0 \right)} \right) d \rra d \rrc ~.
$$
For $\frac{k}{m} \le c < \frac{k+1}{m}$ 
set $\rrc= \frac{\tilde{c} + k}{m}$, and by change of variables
$$\langle F, \hT_m (G) \rangle = \frac{1}{m^2} \int_0^1
\sum_{k=0}^{m-1} \sum_{j=0}^{m-1} 
F\left( \rra, \frac{\tilde{\rrc} +k}{m}, 0 \right)
\overline{G\left( \frac{\rra +j}{m}, \tilde{\rrc}+ k , 0 \right)} 
d \rra d \tilde{\rrc} \,.
$$
Next, for each $j$, set $\tilde{\rra} = \frac{\rra + j}{m}$, 
so $\rra = m \tilde{\rra} -j$ 
with $\frac{j}{m} < \tilde{\rra} < \frac{j+1}{m}$.
We obtain
$$
\langle F, \hT_m (G) \rangle = \frac{1}{m}
\sum_{k=0}^{m-1} \sum_{j=0}^{m-1}
\int_0^1 \int_{\frac{j}{m}}^{\frac{j+1}{m}} 
F \left( m \tilde{\rra} -j , \frac{\tilde{\rrc} +k}{m}, 0 \right)
\overline{G ( \tilde{\rra}, \tilde{\rrc} +k, 0)}
d \tilde{\rra} d \tilde{\rrc}\,.
$$
Using $F(\rra -j, \rrc, 0) = F( \rra, \rrc, 0)$ and
$G( \rra, \rrc +k , 0) = e^{-2 \pi i N k \rra} g( \rra, \rrc, 0)$ leads to
\begin{eqnarray}\label{eq320}
\langle F, \hT_m (G)\rangle & = &
\frac{1}{m}
\sum_{k=0}^{m-1} \sum_{j=0}^{m-1} \int_0^1
\int_{\frac{j}{m}}^{\frac{j+1}{m}} e^{2 \pi i kN \rra}
F\left(m \tilde{\rra}, \frac{\tilde{\rrc} + k}{m} , 0\right)
\overline{G( \tilde{\rra}, \tilde{\rrc}, 0)} 
d \tilde{\rra} d \tilde{\rrc} \nonumber \\
& = &
\int_0^1 \int_0^1
\left( \frac{1}{m} \sum_{k=0}^{m-1}
e^{2 \pi i k N\rra}
F\left( m \tilde{\rra}, \frac{\tilde{\rrc} +k}{m}, 0 \right) \right)
\overline{G ( \tilde{\rra}, \tilde{\rrc}, 0)}
d \tilde{\rra} d \tilde{\rrc} \nonumber \\
& = &
\int_0^1 \int_0^1 \int_0^1
\left( \frac{1}{m}
\sum_{k=0}^{m-1} e^{2 \pi i k N\rra} 
\left( F( m \tilde{\rra}, \frac{\tilde{\rrc} +k}{m}, z \right) \right)
\overline{G (\tilde{\rra}, \tilde{\rrc}, z)} 
d \tilde{\rra} d \tilde{\rrc} dz \nonumber \\
& = & \langle \tilde{\hT}_m (F), G \rangle~,
\end{eqnarray}
where $\tilde{T}_m$ denotes the right side of \eqn{eq316}.
Since knowing $\langle \tilde{\hT}_m (F), G \rangle$
 for all $G \in \sH_N$ uniquely determines $\tilde{\hT}_m (F) \in \sH_N$, 
we conclude that  
$\tilde{\hT}_M(F) = \hT_M^{\ast}(F)$ and \eqn{eq316} holds.\\

(ii) To prove \eqn{eq315} it suffices to show that it holds for 
$F \in \sH_N$ for all $N \in \ZZ$.
It then follows by linearity for all $F \in \sH$, 
since $\hR$ and $\hR^\ast$ leave all $\sH_N$ invariant.
We verify it by direct computation on $\sH_N$;
it suffices to check that $\hR_N \circ \hT_m^\ast = \hT_m \circ \hR_N$ and
$\hR_N^\ast \circ  \hT_m^\ast = \hT_m \circ \hR_N^\ast$ on $\sH_N$.
We compute
\begin{eqnarray*}
\hR_N \circ \hT_m^\ast (F) (\rra, \rrc, z) & = & 
\hT_m^\ast (F) (- \rrc, \rra, z - \rra \rrc ) \\
& = & \frac{1}{m}
\sum_{k=0}^{m-1} e^{-2 \pi i kN\rrc} 
F(- m \rrc , \frac{\rra +k}{m}, z- \rra \rrc )
\\
& = &\frac{1}{m}
\sum_{k=0}^{m-1} e^{-2 \pi iN \rra \rrc - 2 \pi ik N \rrc} 
F( -m \rrc , \frac{\rra +k}{m}, z ) ~,
\end{eqnarray*}
while
\begin{eqnarray*}
\hT_m \circ \hR_N (F) (\rra , \rrc, z ) & = &
\frac{1}{m} \sum_{k=0}^{m-1} \hR_N (F)( \frac{\rra +k}{m}, m \rrc , z ) \\
& = &
\frac{1}{m}
\sum_{k=0}^{m-1} F( - m \rrc, \frac{\rra +k}{m}, z -
( \frac{\rra +k}{m}) m \rrc) \\
& = &
\frac{1}{m} \sum_{k=0}^{m-1} e^{-2 \pi i N \rra \rrc} e^{-2 \pi i kN \rrc} 
F(-m \rrc , \frac{\rra +k}{m}, z) \,.
\end{eqnarray*}

(iii) We verify \eqn{eq317}. Using the formula \eqn{eq316} we have
\begin{eqnarray*}
\hT_m^{\ast} \circ \hT_m (F)(\rra , \rrc, z ) & = &
\frac{1}{m} \sum_{j=0}^{m-1} e^{2 \pi i j N \rra}
\hT_m (F)(m \rra, \frac{\rrc+j}{m}, z) \\
& = &
\frac{1}{m^2}\sum_{j=0}^{m-1}\sum_{k=0}^{m-1} e^{2 \pi i N j \rra}
F(\frac{m\rra+k}{m}, m(\frac{\rrc+j}{m}),z) \\
& = & 
\frac{1}{m^2}\sum_{j=0}^{m-1} \sum_{k=0}^{m-1} e^{2 \pi i Nj \rra}
F(a + \frac{k}{m}, \rrc+j, z) \\
& = & 
\frac{1}{m^2}\sum_{j=0}^{m-1} \sum_{k=0}^{m-1} e^{2 \pi i Nj \rra}
e^{- 2 \pi iNj(\rra+\frac{k}{m})} F(a + \frac{k}{m}, \rrc, z) \\
& = & 
\frac{1}{m} \sum_{k=0}^{m-1}  \left(\frac{1}{m} \sum_{j=0}^{m-1} 
e^{-2\pi i \frac{Njk}{m}} \right) F( \rra + \frac{k}{m}, c , z)
\end{eqnarray*}

If $N \ne 0$, non-vanishing of the inner sum requires that $m |Nk$, and setting $d=(m,N) \ge 1$
this condition becomes 
$k =\frac{N}{d} \ell$
for some integer $0\le \ell< d-1$, which yields
\eqn{eq317} in this case. 
In case $N=0$ we have greatest common divisor  $(m, 0) =m$ and \eqn{eq317} still holds.

The derivation of \eqn{eq318} is similar: 
\begin{eqnarray*}
\hT_m \circ \hT_m ^{\ast}(F)(\rra , \rrc, z ) & = &
\frac{1}{m} \sum_{k=0}^{m-1} 
\hT_m^{\ast}(F)(\frac{\rra+k}{m} ), m\rrc, z) \\
& = &
\frac{1}{m^2}\sum_{k=0}^{m-1}\sum_{j=0}^{m-1} e^{2 \pi i N j (\frac{\rra +k}{m})}
F(m(\frac{\rra+k}{m}), \frac{m\rrc+j}{m},z) \\
& = & 
\frac{1}{m^2}\sum_{k=0}^{m-1} \sum_{j=0}^{m-1}
 e^{2 \pi i N \frac{j\rra}{m}} e^{2\pi i \frac{Njk}{m}} F(a, c+ \frac{j}{m}, z)\\
& = & 
\frac{1}{m} \sum_{j=0}^{m-1}e^{2 \pi i N \frac{j\rra}{m}}  \left(\frac{1}{m} \sum_{k=0}^{m-1} 
 e^{-2\pi i \frac{Njk}{m}} \right) F( \rra , c +\frac{j}{m}, z).\\
&=& 
\frac{1}{m} \sum_{{0 \le j < m}\atop{m | jN}} e^{2 \pi i  \frac{Nj\rra}{m}}   F( \rra , c+\frac{j}{m} , z),\\
&=& 
\frac{1}{m} \sum_{\ell =0}^{d-1} e^{2\pi i (\frac{N \ell}{d}) a} F(\rra , \rrc+ \frac{\ell}{d}, z).
\end{eqnarray*}
To derive the last line, for $N \ne 0$ and $d=(m,N)$  we used the fact that  $m|jN$ makes  $j =\frac{N}{d} \ell$ for some $0 \le \ell <d$.
The case $N=0$ follows by inspection.

The relations \eqref{eq317} and \eqref{eq318} imply
 that $\hT_m$ and $\hT_m^{\ast}$ do not commute on  $\sH_N$ 
for $N \ne 0$ when  $d= (m, |N|) \ge 2$.
When  $(m,N)=1$,  only the $\ell=0$ term occurs in the sums above
and $\hT_m^{\ast}$ and $\hT_m$ commute, with 
$\hT_m^{\ast} \circ \hT_m(F)(a, c, z) =  \frac{1}{m} F(a, c, z)$
and $\hT_m \circ \hT_m^{\ast}(F)(a, c, z) = \frac{1}{m} F(a, c, z)$.
\end{proof}

\paragraph{\bf Remark.}
Theorem \ref{th31} (ii) shows that $\hT_m$ and $\hT_m^\ast$ 
are unitarily equivalent operators on each $\sH_N$ (including $N=0$),
so they have equal operator norms
\beql{eq323}
\| \hT_m \| = \| \hT_m^\ast \|, \qquad\mbox{on}\quad \sH_N.
\eeq
(Here $\| \hT_m \| := \sup_{\{\bx \in \sH:~ \| \bx \| =1\} } \|\hT_m \bx \|$.)
Theorem \ref{th31} (iii) implies that 
  the  operator $\hT_m$ on $\sH_N$  is invertible 
 when $(m,N)=1$ and has operator norm
\beql{eq325}
\| \hT_m \| = \frac{1}{\sqrt{m}} ~.
\eeq
In part II we  
 explicitly compute their action  on $\sH_0$ using  the basis 
$\{ \phi_{jk} : (j,k) \in \ZZ^2 \}$,
and show that 
for $|m| \ge 2$
these operators are not invertible and 
have infinite-dimensional kernels, 
and we show that 
\beql{eq324}
\| \hT_m \| = \| \hT_m^\ast \| = 1,
\eeq
holds for all $m \ne 0$.
%
%
%

\subsection{Heisenberg group action on  $\hT_m$ and $\hT_m^{\ast}$.}\label{sec43}

The operators $\hT_m$ and $\hT_m^{\ast}$  do not commute with
the Heisenberg action on the right, but transform in a simple way under this action.
For  $t \in \RR^{\ast}$ the maps
$ \beta(t): H(\RR) \to H(\RR)$ given by
\beql{eq325b}
\beta(t) [a, c, z] =  [\frac{1}{t}a, tc, z],
\eeq
form a one-parameter
group of  automorphisms of $H(\RR)$,
 i.e.  $\beta(t)g \cdot \beta(t)h = \beta(t)(gh)$
and 
$\beta(t) \circ \beta(t') = \beta (tt')$. \smallskip
%

\begin{theorem}~\label{th32}
For any  $h \in H(\RR)$ and $F \in L^2(H(\ZZ) \backslash H(\RR), d\mu)$
there holds
\beql{eq326}
\rho_h \circ \hT_m(F)(\rra, \rrc, z) = 
\hT_m \circ \rho_{\beta(m)h}(F)(\rra, \rrc, z).
\eeq
and
\beql{eq327}
\rho_h \circ \hT_m^{\ast}(F)(\rra, \rrc, z) = 
\hT_m^{\ast} \circ \rho_{\beta(1/m)h}(F)(\rra, \rrc, z).
\eeq
\end{theorem}

\begin{proof}
Let $h= [\rra', \rrc', z'] \in H(\RR)$, so that
$\beta(m)(h)= [ \frac{\rra'}{m}, m \rrc, z]$. Then  
\begin{eqnarray}
\rho_h \circ \hT_m(F)(\rra, \rrc, z) &= & \hT_m(F)(\rra + \rra', \rrc + \rrc',
z + z' + \rrc \rra') \nonumber \\
 &= & \frac{1}{|m|} \sum_{k=0}^{|m|-1} F( \frac{\rra + \rra'+k}{m},
m(\rrc + \rrc'), z +  z' + \rrc \rra') \nonumber
\end{eqnarray}
and
\begin{eqnarray}
\hT_m \circ \rho_{\beta(m)h}(F)(\rra, \rrc, z) & = & 
\frac{1}{|m|} \sum_{k=0}^{|m|-1} 
\pi_{\beta(m)h}(F) (\frac{\rra+k}{m}, m \rrc,z)  \nonumber \\
& = &
\frac{1}{|m|} \sum_{k=0}^{|m|-1}
F(\frac{\rra+k}{m} + \frac{\rra'}{m}, m \rrc + m \rrc',
z + z' + (m \rrc)\frac {\rra'}{m}) \nonumber
\end{eqnarray}
which yields \eqn{eq326}. 

For the adjoint Hecke operator, by linearity it suffices
to consider $F \in \sH_N$ for each $N$ separately. Then
\begin{eqnarray}
\rho_h \circ \hT_m^{\ast}(F)(\rra, \rrc, z) &= & 
\hT_m^{\ast}(F)(\rra + \rra', \rrc + \rrc',
z + z' + \rrc \rra') \nonumber \\
 &= & \frac{1}{|m|} \sum_{k=0}^{|m|-1} e^{2\pi iN(\rra + \rra')}
F( m(\rra+\rra'), \frac{\rrc + \rrc'+k}{m}, z +  z' + \rrc \rra'). \nonumber 
\end{eqnarray}
Next, using the fact that 
$F(\rra, \rrc, z + w) = e^{2\pi i N w} F(\rra, \rrc, z)$ for $f \in \sH_N$, 
we have
\begin{eqnarray}
\hT_m^{\ast} \circ \rho_{\beta(1/m)h}(F)(\rra, \rrc, z) & = & 
\frac{1}{|m|} \sum_{k=0}^{|m|-1} e^{2 \pi i N \rra}
\rho_{\beta(1/m)h}(F) (m\rra, \frac{\rrc+k}{m}, z)  \nonumber \\
& = &
\frac{1}{|m|} \sum_{k=0}^{|m|-1}e^{2 \pi i kN \rra}
F(m\rra + m\rra',\frac{\rrc+k}{m} + \frac{\rrc'}{m},
z + z' + \left( \frac{c+k}{m}\right) m\rra') \nonumber \\
& = & \frac{1}{|m|} \sum_{k=0}^{|m|-1}e^{2 \pi ik N( \rra + \rra')}
F(m(\rra + \rra'), \frac{\rrc +\rrc' +k}{m}, z + z' + \rrc \rra'),
\end{eqnarray}
which yields \eqn{eq327}.
\end{proof}

%
%
%
\section{Multiplicative Character Decomposition of $\sH_N$}\label{sec5}
\setcounter{equation}{0}

In this section we construct a  decomposition of
 $\sH_N$ for $N \neq 0$ into irreducible
$H(\RR)$ modules, which is associated to  Dirichlet characters
$\chi \in ( \ZZ / d \ZZ )^\ast$ for all $d|N$. We call this
the {\em multiplicative character decomposition} of $\sH_N$.
Note that
\beql{eq401}
\sum_{d|N} \phi (d) = N ~,
\eeq
so this decomposes $\sH_N$ into $|N|$ subspaces, labelled $\HNd (\chi )$.

%
%
%
\subsection{Schr\"{o}dinger representations of $H(\RR)$ }\label{sec51}

The Stone-von Neumann theorem asserts: that for each real $\lambda \ne 0$ there is 
(up to unitary isomorphism)  a unique
an  irreducible 
(infinite-dimensional) unitary representation $\pi_{\lambda}$ for 
the real Heisenberg group $H(\RR)$ having central character $e^{2 \pi i \lambda z}$,
which is unique up to unitary isomorphism. 
The Schr\"{o}dinger representation provides such a representation on the Hilbert space $L^2(\RR, dx)$,
constructed using the operations of modulation and translation, defined  here by 
\beql{800aaa}
\pi_{\lambda}([a,c,z])f(x) := e^{2 \pi i ax} f(x+ \lambda c)e^{2\pi i \lambda z}.
\eeq 
The {\em modulation action} is
$$
\pi_{\lambda}([a,0,0] )f(x) = e^{2 \pi i ax} f(x)
$$
and the {\em translation action} is
$$
\pi_{\lambda} ([0, c,0] )f(x) = f(x+\lambda c).
$$
The central character  multiplies by $e^{2 \pi i \lambda z}$ and
describes  a {\em phase shift action}
$$
\pi_{\lambda} ([0,0, z] ) f(x) = e^{2 \pi i \lambda z} f(x).
$$ 
All these operators  leave the Schwartz class $\sS(\RR)$ in $L^2(\RR)$ invariant.

The Schr\"{o}dinger representation $\pi_{\lambda}$  for
arbitrary $\lambda \ne 0$ can be obtained  on $L^2(\RR, dx)$ from the Schr\"{o}dinger representation $\pi_1$,
by rescaling  under an automorphism of $H(\RR)$ given by
\beql{800bbb}
\gamma_{\lambda}( [a, c, z]) = [a, \lambda c, \lambda z].
\end{equation} 
We have
$$
\pi_{\lambda} ([a, c, z]) := \pi_{1} (\gamma_{\lambda} [a, c, z]).
$$

%
%
%
\subsection{Twisted Weil-Brezin maps}\label{sec52}

We  recall first the Weil-Brezin map,  named after
Weil \cite{We64} and Brezin \cite[Sect. 4]{Br72}.
The 
Weil-Brezin map intertwines the Schr\"{o}dinger
 representation $\pi_1$ of $H(\RR)$
on $L^2(\RR, dx)$ with the Heisenberg action on $\sH_1$.

\begin{defi} \label{de51} 
{\em 
 (1) The {\em Weil-Brezin map} $\sW: L^2 ( \RR , dx ) \to \sH_1$ is defined 
for Schwartz functions $f \in \sS (\RR )$ by
\beql{eq402}
\sW (f) (\rra, \rrc , z ) := e^{2 \pi i z} 
\left( \sum_{n \in \ZZ} f(n+c) e^{2 \pi ina} \right)~.
\eeq
Under Hilbert space completion this map extends to an isometry of  these Hilbert spaces.

(2) The {\em inverse Weil-Brezin map} is
\beql{eq403}
\sW^{-1} (g) (x) = \int_0^1 g(\rra, x-n, 0) e^{-2 \pi in \rra} d \rra
~~\mbox{for}~~
n < x < n +1 \,.
\eeq
}
\end{defi}

This map was independently discovered in  time-frequency signal analysis,
with $f(x)$ being the time-domain  signal, by  Zak \cite{Zak67}, \cite{Zak68},
where it is now called the {\em Zak transform},
cf. Janssen  \cite{Ja82}, \cite{Ja88} and Gr\"{o}chenig \cite[Chap. 8]{Groch01}.

As an example, the  image under the Weil-Brezin map of the Gaussian $\phi (x) = e^{- \pi tx^2}$ is
\begin{eqnarray}\label{eq404}
\sW (f) (a,c,z) & = & e^{2 \pi iz} 
\sum_{n \in \ZZ} e^{- \pi t (n+ \rrc )^2} e^{2 \pi in \rra} \nonumber \\
& = & e^{2 \pi iz} e^{- \pi t \rrc^2} \vartheta_3 (it, \rra + i \rrc t )
\end{eqnarray}
where $\vartheta_3 (\tau, z)$ is the Jacobi theta function
\beql{eq405}
\vartheta_3 (\tau,z ) := \sum_{n \in \ZZ} e^{\pi i n^2 \tau} e^{2 \pi inz},
\eeq
where $\tau \in \HH_{\CC}$ and $z \in \CC$.
L. Auslander \cite{Au77} introduced  a class of $C^{\infty}$ functions on $H(\RR)$
which he called  {\em nil-theta functions}, 
generalizing \eqn{eq404}.

We now generalize this map  to $N \ne 0$, inserting a multiplicative  character.
Given a (primitive or imprimitive) Dirichlet character $\chi \,(\bmod~d )$ with $d$ dividing $|N|$,
we define  a notion of twisted Weil-Brezin map, and 
introduce a  Hilbert space $\HNd(\chi)$  as  its image.
This map will intertwine a copy  of the Schr\"{o}dinger representation $\pi_N$ on $L^{2}(\RR, dx)$
with the Heisenberg action on its image space.

\begin{defi} \label{de52} 
{\em Given a (primitive or imprimitive) Dirichlet character $\chi(\bmod \, d)$,
and an integer $N$ with $d | N$,  
the {\em twisted Weil-Brezin map}
$\WNd (\chi ) : L^2 (\RR , dx ) \to \sH_N$
is
 defined for Schwartz functions $f \in \sS (\RR )$ by
\beql{eq406}
\WNd (\chi ) (f) (\rra, \rrc, z ) :=
\sqrt{\CNd} \, e^{2 \pi iN z} \sum_{n \in \ZZ}
\chi \left( \frac{nd}{N} \right) f(n+ Nc) e^{2 \pi in \rra},
\eeq
in  which 
\beql{eq407}
\CNd := \frac{N}{\phi (d)}
\eeq
is a normalizing factor, using the convention that $\chi (r) :=0$ if $r \not\in \ZZ$.
  (Note   that $\chi (r) =0$
for those $r \in \ZZ$ having   $(r,d) > 1$.)}
\end{defi}

The Weil-Brezin map $\sW$ is the special case $\sW= \sW_{1,1}$ of this definition.

%

\begin{lemma}\label{le41}
For $N \ne 0$ each twisted Weil-Brezin map $\WNd (\chi ) : \sS (\RR ) \to C_\infty (\sH_N )$ 
extends to a Hilbert space isometry
\beql{eq408}
\WNd (\chi ) : L^2 ( \RR , dx ) \longrightarrow \HNd(\chi) \subseteq \sH_{N}
\eeq
whose range $\HNd (\chi )$ is a closed subspace of $\sH_N$.
The Hilbert space $\HNd (\chi )$ is invariant under the action of $H(\RR)$,
and the map $\WNd (\chi )$ intertwines the Schr\"{o}dinger representation $\pi_N$ on $L^2(\RR, dx)$
with this action. 
\end{lemma}

\begin{proof}
We check that for $f,g \in \sS (\RR )$,
\beql{eq409}
\langle f, g \rangle_{L^2 (\RR , dx )} = 
\left( \WNd (\chi ) (f),~ \WNd (\chi ) (g) \right)_{\sH_N} ~.
\eeq
Since $\WNd (\chi ) (f)$, $\WNd (\chi )(g )\in \sH_N$ 
we have 
\begin{eqnarray*}
&&\langle \WNd (\chi ) (f),~ \WNd (\chi ) (g) \rangle \,\, =  \int_0^1\int_0^1 \WNd (\chi ) (f) (\rra, \rrc , 0)
\overline{\WNd (\chi ) (g) ( \rra , \rrc, 0)} d \rra d \rrc \\
& =&  \CNd \int_0^1 \int_0^1
\sum_{n_1 \in \ZZ} \sum_{n_2 \in \ZZ} \chi
\left( \frac{n_1 d}{N} \right) \overline{\chi \left( \frac{n_2 d}{N}\right)}
f(n_1 + Nc ) \overline{g(n_2 + Nc)}
e^{2 \pi i (n_1 - n_2 ) \rra} d \rra d \rrc \\
& = & \CNd \int_0^1 \sum_{n_1 \in \ZZ} \chi \left( \frac{n_1 d}{N} \right)
\overline{\chi \left(\frac{n_1 d}{N}\right)}f(n_1 + N \rrc )\overline{g (n_1 + N \rrc )} d \rrc \,.
\end{eqnarray*}
Now introduce the new summation variable  $n =  \frac{n_1 d}{N} \in \ZZ$. 
Since  
$\chi(\cdot)$ is a character $(\bmod~ d)$ we have
 $\chi(n) \bar{\chi} (n) =1$ if $(n,d ) =1$ and
is $0$ otherwise, and $\CNd= N/\phi(d)$ we obtain 
\begin{eqnarray*}
\langle \WNd (\chi ) (f), ~\WNd (\chi ) (g) \rangle &=&
\frac{N}{\phi(d)} \int_0^1
\sum_{n \in \ZZ \atop (n,d) =1} f ( \frac{nN}{d} + Nc ) 
\overline{g ( \frac{nN}{d} + Nc)} d \rrc \\
& = & N \int_{-\infty}^\infty f(N \rrc ) \overline{g (N \rrc )} d \rrc \\
& =& \langle f,g \rangle_{L^2 (\RR, dx )} ~.
\end{eqnarray*}
Thus $\WNd (\chi): \sS(\RR) \to \sH_N$  is an isometry onto its range. 
We now use the general fact that an isometry $f: \sD \to V$ defined 
on a dense subspace $\sD$ of a separable Hilbert space $\sD \subseteq \sK_1$ 
into a subspace $V \subseteq \sK_2$ extends to an isometry 
$f: \sK_1 \longrightarrow \sK_2$ whose range is a closed subspace 
of $\sK_2$.    We obtain a map $\WNd(\chi)$ from $L^2(\RR, dx)$ to its closed range,
which we name $\HNd(\chi)$. 

It remains to  show that $\sH_{N,d} (\chi)$ is invariant under the 
action  of $H(\RR )$ acting on $\sH_N$. 
Given  $h = (\rra ' , \rrc ' , z ' ) \in H(\RR )$ and 
$F \in \sH_{N,d} (\chi )$, we have 
$$
\rho_h (F) ((\rra, \rrc, z )) = 
F(( \rra, \rrc, z ) \circ (\rra ' , \rrc ',  z' )) =
F(\rra + \rra ' , \rrc + \rrc' , z+z' + \rrc \rra ').
$$
Now $F(\rra, \rrc, z) = \WNd (\chi ) (f) (\rra, \rrc, z )$
and one checks that
\begin{eqnarray*}
\rho_h(F)(\rra, \rrc, z) & = &
e^{2 \pi iN(z+z' + ca')} \sum_{n \in \ZZ} 
\chi (\frac{nd}{N}) f(n+N (\rrc + \rrc')) e^{2 \pi in (\rra+ \rra')} \\
&= & \sW_{N,d}(\chi)(\tilde{f}) (\rra, \rrc, z ),
\end{eqnarray*}
where
\beql{eq415}
\tilde{f}(x) :=
 e^{2 \pi ix \rra '} f(x + N \rrc' ) e^{2 \pi i N z'} \in L^2 (\RR, d x ) ~.
\eeq
Thus $\rho_h (F) \in \sH_{N,d} (\chi )$.
(Strictly speaking this is verified on the dense set of Schwartz functions
on $L^2 (\RR, d x )$ and extended to the whole space by completion.)
Now  $\tilde{f}(x)  = (\pi_N(h)f)(x)$ 
gives the  Schr\"{o}dinger representation action of $[a, c, z]$ 
of $H(\RR )$ with central character $e^{2 \pi iNz}$ on $L^2(\RR, dx)$. 
 This fact establishes the intertwining, and shows
that the $H(\RR)$-action on $\HNd(\chi)$ is irreducible with central character $e^{2 \pi iNz}$.
\end{proof}

Note that every function $F(a,c,z) \in \sH_N$ 
has a unique Fourier expansion
$$
F(a,c,z) = e^{2 \pi iNz}\sum_{m \in \ZZ} h_m(c) e^{2 \pi i m a}
$$
and that those $F \in \sH_{N,d}(\chi)$ have
Fourier coefficients $h_m(c)$ supported
on those $m$ with  $(m, N) = \frac{N}{d}$, as is evident
from \eqn{eq406}.

The next result shows that the twisted Weil-Brezin map $\WNd (\chi )$ 
for $\chi$ a Dirichlet character $(\bmod~d )$ is 
 a rescaling of the twisted Weil-Brezin map  $\Wdd (\chi )$.\smallskip
%

\begin{lemma}\label{le42}
Given a Dirichlet character $\chi$ $(\bmod~d)$, 
for any $f \in L^2 (\RR, dx )$ there holds
\beql{eq410n}
\WNd (\chi )(f) ( \rra, \rrc, z) = \Wdd (\chi )
(\hU(\frac{N}{d}) f)
\left( \frac{N \rra}{d}, d \rrc , \frac{N}{d} z \right) \,.
\eeq
in which the dilation 
$\hU (t) : L^2 (\RR , dx ) \to L^2 (\RR , dx )$ 
is the unitary transformation
\beql{eq411n}
(\hU(t)f)(x) = |t|^{1/2} f(tx ) ~.
\eeq
\end{lemma}

\begin{proof}
Let $f \in \sS (\RR )$.
Starting from \eqn{eq406} only terms 
$n \equiv 0$ $(\bmod~\frac{N}{d} )$ give nonzero contributions 
on the right side and, setting $n=\frac{N}{d} \tilde{\ell}$, we have
\begin{eqnarray*}
\WNd (\chi ) (f) ( \rra, \rrc, z) &=& \sqrt{\CNd} e^{2 \pi i N z}
\left( \sum_{\tilde{l}\in \ZZ} \chi (\tilde{l} ) f\left( \frac{N}{d} \tilde{l} + Nc \right) e^{2 \pi i \frac{N \tilde{l} \rra}{d}} \right) \\
&=& \left( \frac{\CNd}{\frac{N}{d}} \right)^{1/2}
e^{2 \pi i N z} \left(\sum_{\tilde{l} \in \ZZ} \chi ( \tilde{l} ) ({\bU}(\frac{N}{d})f)) 
(\tilde{l} + dc ) e^{2 \pi i\tilde{l} (\frac{N \rra}{d} )}\right) \\
&=&\left( \frac{\CNd}{C_{d,d} \cdot \frac{N}{d}} \right)^{1/2}\Wdd (\chi ) ({\bU}(\frac{N}{d})f)\left( \frac{N \rra}{d}, d \rrc, \frac{N}{d} z \right) ~.
\end{eqnarray*}
Since
$$
\left( \frac{\CNd}{C_{d, d} \cdot \frac{N}{d}}\right)^{1/2}
= \left( \frac{\frac{N}{\phi (d)}}{\frac{d}{\phi(d)} 
\cdot \frac{N}{d}} \right)=1,
$$
the result follows.
\end{proof}

%
%
%
\subsection{Multiplicative character decomposition}\label{sec54}

%

\begin{theorem}\label{th41}
{\rm (Multiplicative Decomposition)}
For $N \neq 0$, the Hilbert space $\sH_N$ has an 
orthogonal direct sum decomposition
\beql{eq410}
\sH_N = \bigoplus_{d| |N|} \left( \bigoplus_{\chi \in (\ZZ / d \ZZ )^\ast}
\HNd (\chi ) \right) \,.
\eeq
In the inner sum  $\chi$ runs over all Dirichlet characters, primitive and imprimitive $(\bmod\, d)$.
Each $\HNd (\chi )$ is invariant under the $H(\RR )$-action 
on $\sH_N$ and is an irreducible representation of $H(\RR )$ with 
central character $e^{2\pi iNz}$.
\end{theorem}

\begin{proof}
Assume  it is known that 
the spaces $\HNd (\chi )$ are 
pairwise orthogonal.
Lemma~\ref{le41} 
shows that for $N \ne 0$  each $\HNd (\chi )$ carries a representation of $H(\RR )$ 
with central character $e^{2\pi i N z}$ of multiplicity  one.
Since we have $|N| = \sum_{d \,| \,|N|} \phi (d)$ summands on the right side, 
and the decomposition is orthogonal we conclude the right side  of \eqref{eq410} carries a
direct sum of   $|N|$ copies of this representation. 
Since it  is known that  $\sH_N$ 
equals a representation of multiplicity $|N|$ of this kind  and  is
completely reducible (\cite{CG90}), the equality \eqref{eq410} follows. 

It remains to  verify pairwise orthogonality of the $\HNd (\chi )$. 
It suffices to verify orthogonality  on the dense subspace of functions
\beql{eq411}
 \SNd (\chi ) := \{ \WNd (\chi ) (f) : f \in \sS (\RR) \} ~.
\eeq
Given $f_1$, $f_2 \in \sS (\RR )$ we have
\begin{eqnarray}\label{eq411b} 
&&\frac{1}{\sqrt{ C_{N,d_1} C_{N, d_2}}}\, \lefteqn{\left( \sW_{N, d_1} (\chi_1 ) (f_1 ) ,~
\sW_{N, d_2} (\chi_2 ) (f_2 )~ \right)_{\sH_N}}  \nonumber\\
&& =
\int_0^1 \int_0^1
\sum_{n_1, n_2 \in \ZZ}
\chi_1 \left( \frac{n_1 d_1}{N} \right)\overline{ \chi_2 \left( \frac{n_2 d_2}{N} \right)}
f_1 (n_1 + N \rrc ) \overline{f_2 (n_2 + Nc )} e^{2 \pi i (n_{1}-n_{2} )\rra}
d \rra d \rrc \nonumber \\
&& =  \int_0^1
\sum_{n_1 \in \ZZ} \chi_1 \left( \frac{n_1 d_1}{N} \right)
\overline{\chi_2 \left( \frac{n_1 d_2}{N} \right)} f_1 (n_1 + N \rrc )
\overline{f_2 (n_1 + N \rrc )} d \rrc\,,
\end{eqnarray} 
where  the final equality  integrated out the $a$-variable, noting that all  terms with $n_1 \ne n_2$ vanish.

Suppose $d_1 \neq d_2$. Then each term in the sum 
\beql{eq412}
\chi_1 \left( \frac{nd_1}{N} \right) \overline{\chi_2 \left( \frac{nd_2}{N} \right)} =0,
 \quad\mbox{for all}\quad n \in \ZZ \,,
\eeq
i.e. they have disjoint support. For the product to not vanish the first character's entry  $\frac{nd_1}{N}$
must be an integer relatively prime to $d_1$,
while the second character's entry  $ \frac{nd_2}{N}$ must be an integer relatively prime to $d_2$, which is not possible. 
Thus
$$
\langle~ \sW_{N, d_1} (\chi_1 ) (f_1),~\sW_{N, d_2} (\chi_2 ) (f_2) ~\rangle =0 ~, \qquad d_1 \neq d_2 ~.
$$

Now  suppose $d_1 = d_2$ and $\chi_1 \ne \chi_2$.  Set $\chi= \chi_1 \overline{\chi_2}$,
which will be  a nontrivial Dirichlet character $(\bmod \, d_1)$. 
We  integrate the right side of \eqref{eq411b} over $c$,  summing $n_1$ in arithmetic progressions $(\bmod \, N)$,
 The character can be nonzero only when  $ \frac{n_1 d_1}{N}$ is an   integer, which  requires $n_1= \frac{N}{d_1} \ell $
 for $\ell \in \ZZ$.  We group the integers $n_1$ in  $d_1$ arithmetic progressions  $(\bmod \, N)$ as
 $n_1=  \frac{N}{d_1} k + N \ZZ,  (0 \le k \le d_1-1)  $
and  obtain
\begin{eqnarray}\label{eq413a}
\langle~ \sW_{N, d_1} (\chi_1 ) (f_1),~&&
\sW_{N, d_1} (\chi_2 ) (f_2)~ \rangle_{\sH_N} \nonumber\\
&= &
\int_0^1
\sum_{n_1 \in \ZZ} \chi_1 \left( \frac{n_1 d_1}{N} \right)
\overline{\chi_2 \left( \frac{n_1 d_1}{N} \right)} f_1 (n_1 + N \rrc )
\overline{f_2 (n_1 + N \rrc )} d \rrc\,, \nonumber \\
&=&  
 \int_0^1
\sum_{k=0}^{d_1-1}    \sum_{n \in \ZZ}    \chi (k +d_1n) f_1 \left( \frac{Nk}{d_1}+N(n+\rrc) \right) 
\overline{f_2 \left( \frac{Nk}{d_1} + N(n+ c)\right)} d \rrc \nonumber \\
&=& 
 \left( \sum_{k=0}^{d_1 -1} \chi (k) \right)
\int_{-\infty}^\infty f_1 (N \rrc ) \overline{f_2 (N \rrc )} d \rrc 
= 0 ~,
\end{eqnarray}
where $\sum_{k=0}^{d_1 -1} \chi(k) =0$ since $\chi$ 
is a nontrivial Dirichlet character $(\bmod~d_1 )$. We have verified pairwise orthogonality.
\end{proof}

%
%
%
\subsection{Hecke operator action on $\HNd (\chi)$}\label{sec55}

We consider the action of the Hecke operators on 
the spaces $\HNd (\chi )$.
For $m \in \ZZ_{> 0}$ and a Dirichlet character $\chi$, define 
${\hT}_m^{\chi} : L^2 (\RR , dx ) \to L^2 (\RR, dx )$ by
\beql{eq416}
{\hT}_m^{\chi} (f)(x) := \chi(m) f(mx ) ~.
\eeq
If $\chi$ is a character $(\bmod~d)$ and $d | d'$ then 
we let $\chi |_{d'}$ denote the (imprimitive)
character $(mod~d')$ which equals 
$\chi(n)$ if $(n, d') = 1$ and $0$ otherwise.\\

%

\begin{theorem}\label{th42}
{\em (Hecke-invariant  Hilbert Subspaces)} 
Let $N \ne 0$ and
suppose that  $\chi$ is a Dirichlet character $(\bmod~d)$ with
$d|N$. Given $m$, 
and letting  $d'= (m, \frac{N}{d} )$, the two-variable
Hecke operator $\hT_m$ leaves invariant
the  Hilbert space
\beql{418}
\sH_{N,d} (\chi; d' ) := 
\bigoplus_{e | d'} \sH_{N,de} (\chi |_{de} ) ~,
\eeq
For  $f \in \sH_{N,d} (\chi )$,
\beql{419}
\hT_m \circ \WNd (\chi ) = \sum_{e|d'} 
\sqrt{\frac{\phi(de)}{\phi(d)}}
\chi \left( \frac{m}{e} \right)
\sW_{N,de} (\chi |_{de} )  \circ {\hT}_m^{\chi_0} ~
\eeq
where $\chi_0$  is the principal character $(\bmod \, 1)$,
and $(\hT_m^{\chi_0} (f))(x) = f(mx)$.

In particular, 
if $(m, \frac{N}{d} ) =1$ then 
$\hT_m$ leaves 
$\HNd (\chi )$ invariant, and
\beql{417}
\hT_m \circ \WNd (\chi ) = \WNd (\chi ) \circ {\hT}_m^{\chi}
=\chi (m)~ \WNd (\chi ) \circ {\hT}_m^{\chi_0} ~.
\eeq
\end{theorem}

\begin{proof}
Given $\WNd (\chi )(f) \in \sH_N$ with $f \in \sS (\RR )$, we have
\begin{eqnarray*}
\hT_m \circ \sW_{N,d} (\chi ) (f) (\rra, \rrc, z) & = &
\frac{1}{m} \sum_{k=0}^{m-1} \sW_{N,d} (\chi ) (f) 
\left( \frac{\rra + k}{m}, m \rrc, z \right) \\
&= & \sqrt{C_{N,d}}e^{2 \pi i Nz}
\sum_{n \in \ZZ} \chi \left(\frac{nd}{N}\right)
f(n+Nm \rrc ) \left( \frac{1}{m} e^{2 \pi in \frac{\rra}{m}}
\sum_{k=0}^{m-1} e^{2 \pi i \frac{nk}{m}} \right) \\
& = &  \sqrt{C_{N,d}}   e^{2 \pi i Nz}
\sum_{n = lm \in m \ZZ } \chi ( \frac{lmd}{N})
 f(lm + Nm \rrc ) e^{2 \pi ila} ~.
\end{eqnarray*}

Writing $m=d'm' $ with $d' = (m, \frac{N}{d} )$, we obtain
\begin{equation}~\label{420}
\hT_m \circ \WNd (\chi ) (f) ( \rra, \rrc, z) =
\sqrt{C_{N,d}}e^{2 \pi i Nz}  \sum_{\ell \in \ZZ} 
\chi\left( \frac{\ell dd'm'}{N} \right)
{\hT}_m^{\chi_0} (f)
(\ell+N \rrc ) e^{2 \pi i\ell a} \,.
\end{equation}
We subdivide $\ell \in \ZZ$ into arithmetic progressions $(\bmod~N)$.

Now $(m', \frac{N}{dd'})= (\frac{m}{d'}, \frac{N}{dd'}) =1$ (by definition of $d'$),
hence $\chi (\frac{\ell dd' m'}{N} ) =0$ unless $\frac{N}{dd'} \Big| \ell$.
Thus we may write $\ell= \frac{N}{dd'}k$, and we also have 
$\chi (\frac{\ell dd' m'}{N} ) =\chi(km')=0$ if $(k, d) >1$. Therefore $(k,d)=1$
so $k$ can only cancel factors in $d'$ and 
\beql{420b} (\ell, N) = \frac{N}{de} ~,
\eeq
for some  $e | d'$.
For fixed $e$ there are $\phi (de )$  
arithmetic progressions $(\bmod~N )$ satisfying \eqn{420b}.
On such a progression, writing $\ell = \frac{N}{de} \cdot \ell'$,
one has $(\ell', de ) =1$, so that 
\begin{equation}~\label{421}
\chi ( \ell \frac{dd'm'}{N}) =  \chi (\ell' \frac{d'm'}{e}) 
= \chi ( \frac{d'm'}{e}) \chi (\ell') 
=  \chi ( \frac{d'm'}{e}) \chi \Bigl|_{de} (\ell '),
\end{equation}
the rightmost equality holding because $(\ell', de ) =1$.

Now \eqn{420} yields for $f \in \sS (\RR )$ that
\begin{eqnarray}\label{eq422}
\hT_m \circ \WNd (\chi ) (f) (\rra, \rrc, z) & = & 
\sqrt{C_{N,d}} e^{2 \pi i N z} \sum_{e\ell d'} \chi \left( \frac{d'm'}{e} \right)
\left( \sum_{{(l,N)= \frac{N}{de}} \atop {l' = \frac{lde}{N}}}
\chi \Bigl|_{de} (\ell') {\hT}_m^{\chi_0} (f) (\ell+ N \rrc ) e^{2 \pi il\ell a} \right) 
\nonumber \\
& = & 
\sqrt{C_{N,d}} e^{2 \pi i N z}\sum_{e | d'} \chi \left( \frac{d'm'}{e} \right)
\sum_{l \in \ZZ} \chi |_{de}(\frac{\ell de}{N}) T_m^{\chi_0}(f)(\ell + Nc)
e^{2 \pi ila} 
\nonumber \\
& = &  \sum_{e | d'} \chi \left( \frac{m}{e} \right)
\sqrt{ \frac{C_{N,d}}{C_{N, de}}  }
\sW_{N,de} \left( \chi \Bigl|_{de} \right)
({\hT}_m^{\chi_0}( f)) (\rra, \rrc, z ) \,.
\end{eqnarray}
This relation holds for all $f \in \HNd (\chi )$ by Hilbert space completion,
and  \eqn{418} follows since
$$\frac{C_{N,d}}{C_{N, de}}= \frac{\phi(de)}{\phi(d)}.$$ 

The relation \eqn{418} shows that
\beql{423}
\hT_m (\HNd (\chi )) \subseteq \bigoplus_{e|d'} \sH_{N, de} 
\left( \chi \Bigl|_{de} \right) ~.
\eeq
Applying this to all $\sH_{N, de} (\chi |_{de} )$ on the left establishes that
$\oplus_{e | N/d} \sH_{N,de} (\chi |_{de} )$
is an invariant subspace for $\hT_m$ in $\sH_N$.
\end{proof}

%
%
%
\subsection{Coarse multiplicative character decomposition}\label{sec56}

Theorem~\ref{th42} yields 
a orthogonal decomposition 
of $\sH_N$,
labeled by  {\em primitive} Dirichlet characters
$\chi ~(\bmod~ \fe)$ with $\fe~|~N$, 
 coarser than the multiplicative character decomposition in Theorem \ref{th41},
 whose 
summands are left invariant by  all two-variable Hecke operators.
 We call the resulting decomposition \eqn{425} below
the {\em coarse multiplicative decomposition} of $\sH_N$.\smallskip
%

\begin{theorem}~\label{th43}
{\rm (Coarse Multiplicative Decomposition)}
Let $N \ne 0$  and 
to each  primitive character $\chi ~(\bmod~\fe)$
with $\fe | N$ assign the Hilbert space
\beql{424} 
\sH_{N}(\chi; \fe) := \bigoplus_{{d} \atop{\fe|~d~|N}} 
\sH_{N,d}(\chi |_d).
\eeq
Then the Hilbert space $\sH_N$ has
the orthogonal direct sum decomposition
\beql{425}
\sH_N = \bigoplus_{\chi,~\fe} \sH_N(\chi; \fe),
\eeq 
in which $\chi$ runs over all primitive characters $(\bmod~\fe)$
for all $\fe | N$. Each Hilbert space $\sH_N(\chi; \fe)$ is  
invariant under all
two-variable Hecke operators $\{\hT_m: ~m \in \ZZ \backslash \{0\} \}$.
\end{theorem}

\begin{proof}
The orthogonality of the decomposition follows from
Theorem~\ref{th41}.
The invariance under all the two-variable Hecke operators
$\hT_m$ follows from the Hecke operator
action given in Theorem~\ref{th42}.
\end{proof}

The formulas in Theorem~\ref{th42} indicate that
the mutually commuting action of the 
two-variable Hecke operators on $\sH_N(\chi, \fe)$
cannot be semisimple, except in cases where $\fe=N$.
Even in the case $\fe=N$,  Theorem ~\ref{th31} shows that $\hT_m$ is 
not a normal operator when  $(m,N)>1$,   in that case one can only hope for simultaneous
diagonalization of those $\hT_m$ having $(m,N) =1$. There remains the possibility of a
simultaneous triangularization of all the Hecke operators acting on the coarse decomposition
spaces.

%
%

\section{Additive character decomposition of $\sH_N$}\label{sec6}
\setcounter{equation}{0}

In this section, for comparative purposes,
we describe an orthogonal decomposition of $\sH_N$ 
for $N \neq 0$ into irreducible  $H(\RR)$-submodules,
due to Brezin \cite{Br70}, which we term the {\em additive character decomposition.} 
This decomposition
 was generalized in 
Auslander and Brezin \cite{AB73} to other nilpotent Lie groups,
including the $(2n+1)$-dimensional Heisenberg group $H_{n}(\RR)$.
In order to state results in parallel with the multiplicative
decomposition given in Section \ref{sec5}, our notation and formulation differs slightly from
that in  Auslander and Brezin \cite{AB73} and 
Auslander\cite{Au77}. We then determine the
action of the two-variable Hecke operators on the additive
character decomposition modules.

%
%
%
\subsection{Distinguished subgroups and additive Brezin maps}\label{sec61}

These decompositions  are associated
 to additive characters of certain 
``distinguished subgroups'' of a given discrete subgroup
of $H_{n}(\RR)$. We define for $N \ge 1$ the discrete subgroups
\beql{501}
\Gamma^U \left( \frac{1}{N} \right) := \left[
\begin{array}{ccc}
1 & \frac{1}{N} \ZZ & \frac{1}{N} \ZZ \\
0 & 1 & \ZZ \\
0 & 0 & 1
\end{array}
\right]
\eeq
of $H(\RR )$. 
Note that  
$H(\ZZ ) = \Gamma^U (1) \subseteq \Gamma^U \left( \frac{1}{N} \right)$ 
and $\Gamma^U (1)$ is a normal subgroup of 
$\Gamma^U \left( \frac{1}{N}\right)$.

%
\begin{defi}\label{def:61}
For each additive character $\psi_k$ on $\ZZ / N \ZZ$, with 
\beql{502}
\psi_k (m)  := e^{\frac{2 \pi i  k m}{N}}, \quad
0 \le k \le N-1 ~,
\eeq
we define the  {\em additive Brezin map}
$\sW_N (\psi_k ) : L^2 (\RR, dx ) \to \sH_N$ on 
Schwartz function $f \in \sS(\RR )$ by
\beql{503}
\sW_N (\psi_k )(f) (\rra, \rrc, z) :=
e^{2 \pi iNz}
\sum_{n \in \ZZ} e^{\frac{2 \pi ikn}{N}} f(n+Nc) e^{2 \pi ina}\,. 
\eeq 
That is, $\sW_N (\psi_k )(f) (\rra, \rrc, z) = \sW(f)(a+ \frac{k}{n}, Nc, Nz).$ 
\end{defi} 
\smallskip

%

\begin{lemma}\label{le51}
(1) The additive Brezin map $\sW_N (\psi ) : \sS (\RR) \mapsto C_\infty (\sH_N )$ 
extends to a Hilbert space isometry
$$\sW_N (\psi_k ) : L^2 (\RR , dx ) \mapsto \sH_N$$
whose image $\sH_N (\psi_k )$ is a closed subspace of $\sH_N$.

(2) The space $\sH_N (\psi_k )$ consists of those elements 
$F \in \sH_N$ that satisfy for $0 \le j \le N-1$ the relations
\beql{504}
F \left( \rra + \frac{j}{N} , \rrc , z \right) =
\psi_k (j) F ( \rra, \rrc, z ) ~,
\eeq
almost everywhere.
\end{lemma}

\begin{proof}
(1) As in Lemma \ref{le41}, one computes that for all $f,g \in \sS (\RR )$,
$$
\langle f, g \rangle_{L^2 (\RR, dx )} =
\langle~ \sW_N (\psi_k )(f) , \sW_N (\psi_k ) (g) ~\rangle_{\sH_N} ~.
$$
The map therefore extends under Hilbert space completion to an isometry.
onto its range, which is closed.

(2) If $F = \sW_N (\psi) (f) \in \sH_N (\psi_k )$
then  the relation 
\eqn{504} holds by inspection of terms in the formula \eqn{503}.
It can be checked that, conversely, the
the functional equations \eqn{504} for 
$0 \le j \le N-1$ suffice to determine the Fourier coefficients 
on the right side of \eqn{503}.
\end{proof}

The group $\Gamma_U(\frac{1}{N})$ is associated to
these Weil-Brezin maps in that 
the transformation \eqn{504} involves 
a translation in the group  $\Gamma_U(\frac{1}{N})$.

The next result shows that
the additive Weil-Brezin maps $\sW_N (\psi_k )$ arise from the original
Weil-Brezin map by a rescaling of variables.\smallskip
%

\begin{lemma}\label{le52}
For $f \in L^2 (\RR, dx )$ and the additive character 
$\psi_k (n) = e^{\frac{2 \pi ikn}{N}}$ on $\ZZ /N \ZZ$, the
additive Brezin map is given by
\beql{505a}
\sW_N (\psi_k ) (f) ( \rra, \rrc, z ) =
\sW (f) ( \rra+ \frac{k}{N} , N \rrc, Nz)
\eeq
where $\sW (f) = \sW_1(\psi_0)(f)= \sW_{1,1} (\chi_0)(f)$ 
is the Weil-Brezin map.
\end{lemma}

\begin{proof}
This formula holds for Schwartz functions by \eqn{503} and carries 
over to all $f \in L^2 (\RR, dx )$ under
Hilbert space completion.
Note that the right side
of \eqn{505a} is $\Phi_N^k \circ \sW (f)$ where
\beql{506a}
\Phi_N^k (F) (\rra, \rrc , z) :=
F( \rra+ \frac{k}{N} , N \rrc, N z ) ~.
\eeq
In particular \eqn{505a} shows that 
$\Phi_N^k: \sH_1 \to \sH_N (\psi_k )$ is a Hilbert space
isometry mapping $\sH_1$ onto $\sH_N (\psi_k )$.
\end{proof}

%
%
%
\subsection{Additive Character Decomposition}\label{sec62}

The following result is a special case of a result of 
Auslander and Brezin \cite{AB73}.\\
%
\begin{theorem}\label{th51} {\rm (Auslander-Brezin)}
For $N \neq 0$ the Hilbert space $\sH_N$ has an orthogonal direct sum
decomposition
\beql{506}
\sH_N = \bigoplus_{k=0}^{N-1} \sH_N (\psi_k ) ~.
\eeq
Each $\sH_N (\psi_k )$ is invariant under the $H(\RR )$-action on $\sH_N$
and is an irreducible representation of $H(\RR )$ with central 
character $e^{2 \pi i N z}$.
\end{theorem}

\begin{proof}
This is Theorem~2 (iii) of Auslander and Brezin \cite{AB73}.
It can also be proved by similar orthogonality calculations 
to those in Theorem \ref{th41}.
\end{proof}

The Hecke operators $\hT_m$ act on the additive character 
decomposition \eqn{502} of $\sH_N$ by mapping each  subspace
$\sH(\psi)$ into another subspace, and
when $(m, N)=1$  this action is a permutation.

%

\begin{theorem}\label{th52} 
For each $N \neq 0$ and each $m \ge 1$ the 
Hecke operator $\hT_m : \sH_N \to \sH_N$ restricts to a  map
\beql{507}
\hT_m : \sH_N (\bpsi_k ) \to \sH_N (\bpsi_{km} )
\eeq
for $0 \le k \le N-1$.
In consequence
\begin{enumerate}
\item
$\sH_N (\psi_k )$ is invariant under 
$\hT_m$ if and only if 
$m\equiv 1$ $(\bmod~N/(k,N ))$.
In particular  $\sH_N (\psi_0)$ is invariant under all $\hT_m$.
\item
 If $(m, N)=1$ then the  action of $\hT_m$ permutes
the $\sH_N(\psi_k)$.
\end{enumerate} 
\end{theorem}

\begin{proof}
Let $\sW_N (\psi_k ) (f) \in \sH_N (\bpsi_k )$, and let 
$\tilde{\hT}_m^{\chi_0} : L^2 (\RR, dx ) \to L^2 (\RR, dx )$ by
\beql{509}
\tilde{\hT}_m ^{\chi_0} (f) (x) := f(mx) ~.
\eeq
Then
\begin{eqnarray}\label{508}
\hT_m \circ \sW_N (\psi_k) (f)(\rra, \rrc,z ) 
& = & \frac{1}{m} \sum_{j=0}^{m-1} \sW_N (\psi_k ) (f)
\left( \frac{a+ j}{m}, m \rrc, z \right) \nonumber \\
&=&  \frac{1}{m} \sum_{j=0}^{m-1}
\sW (f)(\frac{a+j}{m} +\frac{k}{N}, Nmc, Nz) \nonumber\\
& = &\frac{1}{m} \sum_{j=0}^{m-1} \sum_{n \in \ZZ} e^{\frac{2 \pi ikn}{N}}
f(n+Nm \rrc ) e^{2 \pi in ( \frac{\rra +j}{m})} e^{2\pi i Nz} \nonumber \\
& = &\sum_{n \in \ZZ} e^{\frac{2 \pi kn}{N}} f(n+Nm \rrc ) 
e^{\frac{2 \pi in \rra}{m}}
\left( \frac{1}{m} \sum_{j=0}^{m-1} e^{\frac{2 \pi inj}{m}} \right) e^{2 \pi iNz}\nonumber \\
& = & \sum_{\ell \in \ZZ \atop n = m\ell} e^{\frac{2 \pi kml}{N}}
f(m(l+N \rrc )) e^{2 \pi ila} e^{2\pi i N z} \nonumber \\
& = & \sW_N (\psi_{km} ) \circ \tilde{\hT}_m^{\chi_0} (f) (a, c, z).
\end{eqnarray}
Now \eqn{508} shows that 
$\hT_m \circ \sW (\psi_k ) (f) \in \sH_N (\psi_{km} )$, as asserted.

We have  $\psi_{km} \equiv \psi_k$ if and only if either
$k=0$  or $k \not\equiv 0$ $(\bmod~N)$ 
and $m\equiv 1$ $(\bmod~N/(k,N))$. If $(m, N)=1$ the map
on additive characters $\psi_k$ to $\psi_{km}$ is bijective. 
\end{proof}

\paragraph{\bf Remarks.}
(1) There are similar orthogonal direct sum decompositions
associated to additive characters of other discrete subgroups.
For example, there is an orthogonal direct sum decomposition 
of $\sH_N$ associated to the characters of
\beql{509a}
\Gamma_L \left( \frac{1}{N} \right) :=
\left[
\begin{array}{ccc}
1 & \ZZ & \frac{1}{N} \ZZ \\
0 & 1 & \frac{1}{N} \ZZ \\
0 & 0 & 1
\end{array}
\right] ~.
\eeq
that leave the normal subgroup $\Gamma_L(1)$ fixed.

(2) For
$N=M^2$ there is also  an orthogonal direct sum decomposition
\beql{510}
\sH_N = \bigoplus_{\psi_1, \psi_2 \in (\ZZ / M \ZZ)} \sH (\psi_1, \psi_2 )
\eeq
associated to the discrete subgoup 
\beql{511}
\Gamma_L^U \left( \frac{1}{N^2} \right) =
\left[ \begin{array}{ccc}
1 & \frac{1}{N} \ZZ & \frac{1}{N^2} \ZZ \\
0 & 1 & \frac{1}{N} \ZZ \\
0 & 0 & 1
\end{array}
\right] ~.
\eeq
of $H(\RR)$.

(3) This additive character decomposition is associated to 
the finite Fourier transform is studied by
 Auslander and Tolimieri \cite{AT82}. 

%
%
\section{Dilation  Action on $\sH_N$ and the  sub-Jacobi group}\label{sec7}
\setcounter{equation}{0}

The group $ GL(1, \RR):= \RR^{\ast}  $ has a unitary action on
$L^2(\RR, dx)$ by dilations
\beql{601}
\bU(t) (f)(x) := |t|^{1/2} f(tx),~~~~ t \neq 0.
\eeq
Using the Weil-Brezin maps $\sW_{N,d}(\chi)$, we show  this action carries over
to all Heisenberg modules $\sH_N$, for $N \neq 0$,
and that it behaves nicely with respect to all the two-variable 
Hecke operators.
%
%
\subsection{Dilations and Two-Variable Hecke Operators}\label{sec71}

The dilation  operators give an action of $\RR^{\ast}$ on each
Heisenberg module $\sH_{N,d}(\chi)$ using 
the intertwining map  $\WNd (\chi)$.
Set $\hV(t): \HNd (\chi) \to \HNd (\chi)$ by
\beql{602}
\hV(t) := \WNd (\chi) \circ \bU(t) \circ \WNd (\chi)^{-1}
~~\mbox{for}~~ t \ne 0,
\eeq
Writing $F= \sW_{N,d}(\chi)(f)  \in \sH_{N,d}(\chi)$ as \eqn{eq406},
we have 
\beql{602b}
\hV(t)(F)(a,c,z) =|t|^{1/2} \sqrt{C_{N,d}}e^{2\pi iNz}\sum_{n \in \ZZ}
\chi(\frac{nd}{N}) f(t(n+Nc))e^{2\pi ina}.
\eeq
Now $\hV(t)$ is a unitary operator on $\WNd (\chi)$,
and taking direct sums defines the (diagonal)  dilation
action $\hV(t): \sH_N \to \sH_N$, for  all $N \neq 0$.\\

%

\begin{theorem}~\label{th61}
For each $N \neq 0$, the $\RR^{\ast}$-action 
$\{ \hV(t)~:~ t \in \RR^{\ast} \}$ on $\sH_N$ commutes
with all two-variable Hecke operators. That is, for
each $m \ne 0$ and $t \in \RR^{\ast}$,
\beql{603}
\hV(t) \circ \hT_m = \hT_m \circ \hV(t).
\eeq
\end{theorem}

\begin{proof}
This is a consequence of  Theorem~\ref{th42}. 
It suffices to check it on each $\sH_{N,d}(\chi)$ separately.
For $m \ge 1$,
on setting 
$d' = \mbox{ gcd} (m, \frac{N}{d})$ it yields
\beql{603b}
\hT_m \circ \WNd (\chi)(f) = \frac{1}{\sqrt{m}} \sum_{e~|~d'}
\chi(\frac{m}{e})
\sqrt{\frac{\phi(de)}{\phi(d)}}
 \sW_{N, de}(\chi|_{de})({\bU}(m) (f)).
\eeq
Then
\begin{eqnarray*}
\bV(t) \circ \hT_m \left( \WNd(\chi)(f) \right) & = & 
\frac{1}{\sqrt{m}} \sum_{e~|~d'}
\chi(\frac{m}{e}) \sqrt{\frac{\phi(de)}{\phi(d)}}
\bV(t) \circ \sW_{N, de}(\chi|_{de})(\bU(m) (f)) \\
& = & \frac{1}{\sqrt{m}}\sum_{e~|~d'}\chi(\frac{m}{e})\sW_{N, de}(\chi|_{de})
\sqrt{\frac{\phi(de)}{\phi(d)}} 
(\bU(t) \circ  \bU(m) (f)) \\
& = &  \frac{1}{\sqrt{m}}\sum_{e~|~d'}\chi(\frac{m}{e})
\sqrt{\frac{\phi(de)}{\phi(d)}}
\left(\sW_{N, de}(\chi|_{de}) \circ   \bU(m) \right) ( \bU(t)(f)) \\
& = & \hT_m \circ  \WNd (\chi)({\bU}(t) (f))  \\
& = & \hT_m \circ \hV(t) \left( \WNd (\chi)(f)\right),
\end{eqnarray*}
as asserted.

Finally  $\hT_{-m} = \hR^2 \circ \hT_m$ and $\hR^2 = \hV(-1)$
on $\sH_{N,d}(\chi)$ so the result follows for negative $m$.
\end{proof}


Note that \eqn{603b} gives for $(m,N)=1$ that
on $\sH_{N,d}(\chi)$ there holds
$\hT_m = |m|^{-1/2} \chi(m) \hV(m)$.

%
%
\subsection{Sub-Jacobi Group Action on $\sH_N$}\label{sec72}

The $\RR^{\ast}$-action on each $\HNd(\chi)$ $(N \neq 0)$  combines with the
Heisenberg $H(\RR)$-action to give an irreducible unitary representation
of a certain four-dimensional real Lie group $H^J$, defined below in \eqref{609}, 
 which we call the {\em sub-Jacobi group}.

Let $N \neq 0$ and the Dirichlet character $\chi ~(\bmod~d)$ be
given.  Recall that the  right $H(\RR)$-action on $F(\rra, \rrc, z) \in \HNd(\chi)$ is
\beql{604}
\rho([\rra', \rrc', z'])(F)(\rra, \rrc, z) := 
F((\rra, \rrc, z) \circ (\rra', \rrc', z')) =
F(\rra + \rra', \rrc + \rrc', z + z' + \rrc \rra').
\eeq

\
%

\begin{defi}\label{defi73}
{\em 
For $t \in \RR$ and $[a', c', z'] \in H(\RR)$  the 
{\em (unitary)  operators} $\rho_J([t, \rra', \rrc', z'])$ 
act on $\HNd(\chi)$ by
$$
\rho_J([t, \rra', \rrc', z'])(F)(\rra, \rrc, z) :=
\hV(t) \circ \rho([\rra', \rrc', z'])(F)(\rra, \rrc, z).
$$
Here $\hV(t)$ acts on $F= \sW_{N,d}(\chi)(f)$
for suitable $f \in L^2(\RR, dx)$ via \eqref{602b}.
}
\end{defi}

We obtain the following result.\medskip

%

\begin{theorem}~\label{th62}
{\rm (Sub-Jacobi group action)} 

For $N \neq 0$, each positive $d | N$
and each Dirichlet character
$\chi~(\bmod~ d)$ the Hilbert space $\HNd (\chi)$ carries an irreducible
unitary representation of a four-dimensional solvable real
Lie group $H^J$ with central character $e^{2 \pi i N z}$.
Two such representations are unitarily equivalent
$H^J$-modules if and only if they have the same value of $N$.
\end{theorem}

\begin{proof}
To show that the set of all  operators $\rho_J([t, a', c', z'])$ forms a 
(unitary) representation on $\HNd(\chi)$ of a  four-dimensional
real Lie group, we 
pull the action  back to $L^2(\RR, dx)$ using the inverse
twisted Weil-Brezin map $\WNd(\chi)^{-1}$. 
The $H(\RR)$- representation pulls back as follows. If
$f(x) = \WNd(\chi)^{-1}(F)$ then we have
\beql{605}
\tilde{\rho}_N ([\rra', \rrc', z'])(f)(x) = 
e^{2 \pi i N z'} f(x + \rrc') e^{2 \pi i N \rra'},
\eeq
which depends only on $N \neq 0 $ and is independent
of $d$ and the character $\chi$. Note here that for  $N=0$ the  formula
\eqn{605} makes sense and
defines  an $H(\RR)$-action with trivial central character 
viewed with image in  $L^2(\RR, dx)$,
but there is no corresponding Weil-Brezin map.

The $\RR^{\ast}$-action $\hV(t)$ pulls back to $\bU(t)$, and
we set
\beql{611}
\tilde{\rho}_N ( [t, \rra, \rrc, z])(f)(x) := \bU(t) \circ
\tilde{\rho}_N (\rra, \rrc, z) (f) (x)
\eeq
Using \eqn{605} we compute 
\beql{606}
\bU(t) \circ \tilde{\rho}_N ([\rra', \rrc', z'])(f)(x) =
\sqrt{|t|} e^{2 \pi i Nz'} f(tx + \rrc') e^{2 \pi i N \rra'},
\eeq
\beql{607}
\tilde{\rho}_N ([\rra', \rrc', z']) \circ \bU(t) (f)(x) = \sqrt{|t|}
e^{2 \pi i N z'} f(t(x+ \rrc'))e^{2 \pi i N \rra'}.
\eeq
These formulas combine to give 
\beql{608}
\bU(t) \circ \tilde{\rho}_N ([\rra, \rrc, z])  \equiv
\tilde{\rho}_N ( [\rra, \frac{1}{t} \rrc, z]) \circ \bU(t),
\eeq
whose functional form is independent of $N \in \ZZ$. 
Thus we have  a representation of
a four-dimensional real  Lie group $H^J$
whose general element is
$$ [t, \rra, \rrc, z] \in \RR^{\ast} \times \RR^3, $$
having  multiplication law
\beql{609}
 [t, \rra, \rrc, z] \circ [ t', \rra', \rrc', z'] =
 [ t t', \frac{1}{t'}\rra + \rra', t'\rrc + \rrc', z + z' + t'\rrc \rra'].
\eeq
The group $H^J$ has a faithful $4 \times 4$ matrix representation 
\beql{610}
\rho( [ t, \rra, \rrc, z])=  \left[ \begin{array}{llll}
1 & \rrc & \rra        &   z     \\
0 &  t   &  0          & t \rra  \\
0 &  0   & \frac{1}{t} &   0     \\
0 &  0   &  0          &   1
\end{array} 
\right].
\eeq
It is a  semi-direct product   $\RR^{\ast} \ltimes H(\RR)$,
and is a solvable Lie group.
 
The action $\tilde{\rho}_N ( [t, \rra, \rrc, z])$ 
in \eqn{611} 
defines a unitary representation of $H^J$ on $L^2(\RR, dx)$, 
where unitarity follows from \eqn{608}.
For each  $N \ne 0$ this $H^J$-representation
 is irreducible because it is already
irreducible as an $H(\RR)$-representation. This property then
holds for  the representation $\rho_J$ acting on $\HNd(\chi)$ using the intertwining map
$\WNd(\chi)$. 

Two such representations viewed  as $H(\RR)$-representations
are unitarily equivalent if and only if they have the
same value of $N$. Viewed as $H^J$-representations, they
are unitarily equivalent if they have the same value
of $N$, by inspection of the action \eqn{606} on 
$L^2(\RR, dx)$.
\end{proof}

We term the group $H^J$ the {\em sub-Jacobi group} because it
can be identified as a subgroup of the
{\em Jacobi group} $\mbox{Aut}(H(\RR)) \ltimes H(\RR)$,
see  Appendix A. This group $H^J$
has been called the {\em extended (1+1)-dimensional
Poincar\'{e} group} in the physics literature, see de Mello and Rivelles \cite{MR02}. 

The action of $H^J$ on each $\sH_{N,d}(\chi)$ extends to
an action of $H^J$ on $\sH_N$ for each $N \neq 0$, by taking the direct sum.
It is an interesting question whether this action has
a natural definition directly on $\sH_N$ without
invoking the Weil-Brezin maps. A crucial feature of
this action is that it commutes with the two-variable
Hecke operators as given in Theorem~\ref{th61}.

%
%
\subsection{Compatiblity of Additive and Multiplicative character $\RR^{\ast}$-actions}\label{sec73}

One  can  define in a similar fashion 
an  (possibly different)
$H^J$-action on $\sH_N$ for each $N \ne 0$ by using the
additive character decomposition of $\sH_N$  given in Section \ref{sec6}.
That is, one pushes  forward the action \eqn{606} to  each $\sH_N(\psi)$ 
by the appropriate Weil-Brezin map $\sW_N(\psi)$, and then takes a direct
sum. We show that this action $\thV(t)$ coincides with the
$H^J$-action above.\smallskip
%

\begin{theorem}~\label{th63}
For each $N \ne 0$ the  $\RR^{\ast}$-action $\hV(t)$ on $\sH_N$
induced  from the multiplicative character decomposition coincides with
the $\RR^{\ast}$-action $\thV(t)$ induced from the additive
character decomposition of $\sH_N$.
\end{theorem}

\begin{proof}
The action $\thV(t)$ for a function 
$F = \sH_N(\psi)(f) \in \sW_N(\psi)$ is given by
$$
\thV(t)(F)(\rra, \rrc, z) = |t|^{1/2} e^{2\pi i Nz}\sum_{n \in \ZZ}
\psi(n) f(t (n + N\rrc))e^{2 \pi i n\rra}.
$$

By linearity it suffices to check the equivalence for
any  function $G= \sW_{N,d}(\chi)(g) \in \sH_{N,d}(\chi).$
We use the fact that the function $ h(n) = \chi(\frac{nd}{N})$
is periodic with period $N$ and therefore can be expressed
as a linear combination
$h(n) = \sum_{k=0}^{N-1} a(k) \psi_{k}(n),$ for certain 
(complex) coefficients $a(k)$. It follows that
for all $g(x) \in L^2(\RR, dx)$ there holds
$$\sW_{N, d}(\chi)(g) = \sqrt{C_{N,d}} \sum_{k=0}^{N-1} a(k)\sW_N(\psi_k)(g).$$
We now obtain
\begin{eqnarray}
\hV(t)(G)(a, c, z) & = & |t|^{1/2} \sqrt{C_{N,d}}e^{2\pi iNz}\sum_{n \in \ZZ}
\chi(\frac{nd}{N}) g(t(n+N\rrc))e^{2\pi i n \rra} \nonumber \\
& = & |t|^{1/2} \sqrt{C_{N,d}}e^{2\pi iNz} \sum_{k=0}^{N-1} a(k)
\sum_{n \in \ZZ} \psi_k(n) g(t(n+N\rrc))e^{2\pi i n \rra}  \nonumber \\
& = & \sqrt{C_{N,d}} \sum_{k=0}^{N-1} a(k) 
\thV(t)(\sW_N(\psi_k)(g)) \nonumber \\
& = & \thV(t)(\sW_{N,d}(\chi)(g))(\rra, \rrc, z)= \thV(t)(G)(\rra, \rrc, z),
\end{eqnarray}
as asserted.
\end{proof}

%
%

\subsection{Lerch $L$-Functions as Mellin Transforms}\label{sec74}

The two functions $L^{\pm}( s, a, c)$  studied in \cite{LL1} 
and given in \eqref{100} can 
be interpreted as arising from a multiplicative
Fourier transform (Mellin transform) associated to
the $\RR^{\ast}$-action $\{ \hV(t):~ t \in \RR^{\ast}\}$.
These two  functions 
can  be identified
with the value $z=0$ of
Lerch $L$-functions  $ L_{1,1}^{\pm}(\chi_0,s,  \rra, \rrc ,z)$
with character $\chi_0$ on the Heisenberg group
 defined in Section \ref{sec91} following,
 in which $\chi_0$ is the principal character $(\bmod \, 1)$.

Recall that
the two-sided Mellin transforms $\sM_k(f)$ for $k = 0,1$ 
are defined  by
\beql{620}
\sM_k(f)(s) := \int_{-\infty}^{\infty} f(x) (sgn(x))^k 
|x|^{s} \frac{dx}{|x|}.
\eeq
The {\em one-sided Mellin transform} is
\beql{621}
\sM(f)(s) := \int_{0}^\infty f(x) x^s \frac{dx}{x},
\eeq
which satisfies, formally,
$$
\sM(f)(s) = \frac{1}{2} \left( \sM_0(f)(s) +  \sM_1(f)(s) \right).
$$
The function $f(x)$ must have some growth restrictions as $ x \to 0^+$
and $x \to \infty$ in order for these integrals to converge for
some $s \in \CC$. The multiplicative averaging operator
$\hA^{\rra, \rrc}[f](t)$ introduced  in \cite{LL1} 
is given, for $f(x) \in \sS(\RR)$ and
$t \in \RR^{\ast}$, by
\beql{622}
\hA^{\rra, \rrc}[f](t) = \sum_{n \in \ZZ} f( (n+\rrc)t) e^{2 \pi i n \rra} 
= \sW( \bU(t)(f))(\rra, \rrc, 0)= [\hV(t) \circ \sW(f)](a, c, 0),
\eeq
where $\sW= \sW_{1,1}(\chi_0)$ is the Weil-Brezin map.
We obtain a Mellin transform

\begin{prop}\label{thm75}
{\rm (Mellin Integral Representation)}
For  any test function $f(x) \in \sS(\RR)$, on  the half-plane
$\Re(s) > 1 $ there holds
\begin{eqnarray}\label{625}
&&\frac{1}{2} \sM_k(f)(s)
L_{1,1}^{\pm}(\chi_0, s, \rra, \rrc)  \nonumber\\
&=&\int_{0}^{\infty} \left[\sW_{1,1}( \bU(t)(f))(\rra, \rrc, 0) +
(-1)^k \sW_{1,1}( \bU(t)(f)(-\rra, -\rrc, 0)\right] t^s \frac{dt}{t}.
\end{eqnarray}
\end{prop}

\paragraph{\bf Remark.}
The formula \eqn{625} exhibits the $\RR^{\ast}$-action $\bU(t)$ inside  the
Mellin transform. The combination of two terms on the
right side of \eqn{625} was needed to get a 
Mellin transform  having
a nonempty  half-plane of absolute convergence. This combination of terms
creates a function invariant under the reflection
automorphism 
$$
\alpha^2(a, c, z) := (-a, -c, z)
$$ 
of $H(\RR)$. Alternatively,
since $\bV(t) (\sW_{1,1}(f))= \sW_{1,1}(\bU(t)f)$ the right side of \eqref{625} equals
$$
\int_{0}^{\infty} \left[\bV(t) \sW_{1,1}(f)(\rra, \rrc, 0) +
(-1)^k \bV(t) \sW_{1,1}(f)(-\rra, -\rrc, 0)\right] t^s \frac{dt}{t}.
$$

\begin{proof}
In \cite[Lemma 2.1]{LL1} it was shown that for $k = 0, 1$ the 
the function given by  the operator
\beql{623}
\hB_k^{\rra, \rrc}[f](t): = \hA^{\rra, \rrc}[f](t) + 
(-1)^k e^{- 2 \pi i \rra} \hA^{1-\rra, 1 - \rrc}[f](t),
\eeq
with $0 < \rra, \rrc < 1$ 
applied to a Schwartz function $f(s)$,
has  one-sided Mellin transforms $\sM(\hB _k^{\rra, \rrc}[f])(s)$
which are  absolutely convergent 
in the half-plane $\Re(s) > 1$, with
\beql{624}
\sM(\hB _k^{\rra, \rrc}[f])(s) = \frac{1}{2} \sM_k(f)(s)
L_{1,1}^{\pm}(\chi_0, s, \rra, \rrc)~~\mbox{with}~~  (-1)^k = \pm.
\eeq
The  formula \eqref{625} is derived using \eqn{623} together with the identity
$$
e^{- 2 \pi i \rra} \sW_{1,1}( \bU(t) (f))(1- \rra, 1 - \rrc, 0) =
\sW_{1,1}( \bU(t) (f))( \rra, \rrc, 0).
$$
\end{proof}

Given a primitive Dirichlet character $ \chi~(\bmod~N)$, 
one may derive   formulas similar to \eqn{625} which replace the
Weil-Brezin map with an appropriate modified Weil-Brezin
map $\sW_{N, N}(\chi)$. In Section \ref{sec91} we extend the definition
of {\em Lerch $L$-functions}  to
all $\WNd(\chi)$ and to all (primitive or imprimitive) Dirichlet characters $\chi$,
specifying 
$L_{N,d}^{\pm}( \chi, s, \rra, \rrc, z)$.
In Section \ref{sec94} we  show that these $L$-functions 
satisfy  suitable functional equations. The integral formulas above specialize
to the value $z=0$.

%
%
\section{$\hR$-Operator Action and  Additive Fourier Transform}\label{sec8}
\setcounter{equation}{0}

This section determines  the action of the  Heisenberg-Fourier operator
$\hR(F)(a, c, z) = F(-c, a, z - ac)$  on the spaces $\sH_{N,d}(\chi)$,
where $d$ divides $N$ . Recall from Section \ref{sec32} that on the invariant
subspace $\sH_N$ we
denote this operator by $\hR_N$, and it is given by
$$
\hR_N (F)(\rra, \rrc, z ) = e^{- 2 \pi i N \rra \rrc} F(- \rrc, \rra, z ) ~,
\qquad F \in \sH_N.
$$
A. Weil \cite{We64} observed in 1964 that 
$\hR_1(F)(a, c, z)$ 
intertwines with the additive Fourier transform $\sF$
under the Weil-Brezin map on $\sH_1= \sH_{1,1}(\chi_0)$. 
That is,
\beql{700}
\hR_1( \sW_{1,1}(\chi_0)(f)) = \sW_{1, 1}(\chi_0)(\sF(f)),
\eeq
where the (normalized) Fourier transform is given by 
\beql{700a}
\sF(f) (y)=
\int_{-\infty}^{\infty} f(x) e^{ 2 \pi i x y} dx.
\eeq
We show that for general $N \ne 0$ the action of $\hR_N(F)(a, c, z)$ 
has  a more  complicated intertwining with  the additive Fourier transform. 
 It  is necessary to 
distinguish  between primitive and imprimitive Dirichlet
characters, and the dilation operator  $\hU(N)$ is needed to describe
the action.

%
%

\subsection{Gauss sums for imprimitive characters}\label{sec81}

We need to make use of  known formulae for Gauss sums of imprimitive characters.
Let $\chi$ be a primitive Dirichlet character $(\bmod~\fe)$
and if $\fe~|~d$  let $\chi |_d$ be the (imprimitive) Dirichlet
character $(\bmod~d)$ defined by 
$\chi |_d(m) = \chi(m)$ if $(m, d) = 1$, and $0$ otherwise. For
any integer $m$ the  {\em Gauss sum} $G(m, \chi |_d)$ is given by
$$
G(m, \chi |_d):= \sum_{k (\bmod~d)} \chi |_d (k) e^{2 \pi i \frac{km}{d}}
$$
The {\em standard Gauss sum} $\tau(\chi)$ of a primitive
character $\chi (\bmod\,  \fe)$  is
$\tau(\chi) := G(1, \chi)$,
and it satisfies $|\tau(\chi)|^2 = \fe$. \\
%

\begin{prop}~\label{pr71}
Let $\chi$ be a primitive Dirichlet character $(\bmod~\fe)$ and
suppose $\fe | d$. For any integer $m$ set 
$ m' = \frac{m}{(m,d)} ~~\mbox{and}~~ d' = \frac{d}{(m,d)}.$
Then

(i) If $\fe \nmid d'$, then
\beql{e701}
G(m, \chi |_d) = 0.
\eeq

(ii) If $\fe | d'$, then
\beql{e702}
G(m, \chi |_d) = \left( \frac{\phi(d)}{\phi(d')} 
\mu (\frac{d'}{\fe}) \chi(\frac{d'}{\fe})\bar{\chi}(m') \right) \tau(\chi).
\eeq
\end{prop}

\begin{proof}
This is  derived in Hasse ~\cite[pp. 444--450]{Ha64}
and in Joris ~\cite[Theorem A]{Jo77}. A generalization
appears in Nemchenok \cite{Ne93}.
\end{proof}

More general formulas for Gauss sums on  finite (commutative or
noncommutative) rings were derived by
Lamprecht \cite{Lm53}, \cite{Lm59}. We note
that in all cases the modulus squared of
such a Gauss sum is an integer.

For primitive or imprimitive
characters, if   $(m, d) = 1$ then one has 
$G(m, \chi |_d) = \bar{\chi} |_d (m) G(1, \chi)$.
However only for 
primitive characters is it true 
that  $G(m, \chi) = 0$ if and only if $(m,d)>1$.

%
%

\subsection{Fourier transform intertwinings with $\hR$-operator under  Weil-Brezin maps }\label{sec82}

The  nonvanishing
of $G(m, \chi |_d)$ for some
$(m, d) > 1$ for imprimitive
characters leads to complications in
the  formulas for the $\hR$-operator action.
%

\begin{theorem}~\label{th71}
Let $N \ne 0$, and let $d$ be a positive divisor of $|N|$.
Suppose that
 $\chi$ is a primitive Dirichlet character $(\bmod~\fe)$  
with  $\fe | d$. 
Then for  $f(x) \in L^2(\RR, dx)$,
there holds
\beql{e703}
\hR_N(\sW_{N,d}(\chi |_d)(f)) = 
\frac{ \tau(\chi)} 
{|N|^{\frac{1}{2}}}
\sum_{\tilde{d}~|~|N|}
C_{N,d}(\tilde{d}, \chi)~ \sW_{N, \tilde{d}}(\bar{\chi}|_{\tilde{d}})
(\sF \circ \hU(N)(f)),
\eeq
for certain  coefficients $C_{N, d} (\tilde{d}, \chi)$.
These coefficients  are given in terms of 
$d' := \frac{d}{(|N|/\tilde{d}, d)}$ as follows.
If  $\fe | d'$ then 
\beql{e704}
C_{N, d} (\tilde{d}, \chi) := \sqrt{ \frac{\phi(\tilde{d})}{\phi(d)} }
\left( \frac{\phi(d)}{\phi(d')} \mu ( \frac{d'}{\fe})
\chi(\frac{d'}{\fe}) \bar{\chi}(\frac{Nd'}{\tilde{d}d}) \right),
\eeq
and if $\fe \nmid d'$ then $C_{N, d} (\tilde{d}, \chi) = 0$.
\end{theorem}

\paragraph{\bf Remark.} 
The proof shows that
$\fe \nmid \tilde{d}$ implies  $\fe \nmid d'$, so that
the coefficients $C_{N, d} (\tilde{d}, \chi) = 0$ whenever 
$\fe \nmid \tilde{d}$.
However it can happen that  $C_{N, d} (\tilde{d}, \chi) \ne 0$
for some $\tilde{d}$ a strict divisor of $d$, e.g. 
when $N=d = \fe^3$
and $\tilde{d}= \fe$, where $d'=\fe$.\\

\begin{proof}
Set 
$$
F(\rra, \rrc, z) = \sW_{N,d}(\chi |_{d})(f)(\rra, \rrc, z)
= \sqrt{C_{N,d}} e^{2 \pi i N z} \sum_{N \in \ZZ} \chi |_d(\frac{nd}{N}) 
f(n + N\rrc) e^{2\pi i n\rra},
$$
in which $C_{N,d} = \frac{N}{\varphi(d)} $. Then,
taking $\tilde{n} = \frac{nd}{N}$, we have
\begin{eqnarray}~\label{e705}
\hR(F)(\rra, \rrc, z) & = & \sqrt{C_{N,d}} e^{2 \pi i N (z- \rra \rrc)} 
\sum_{n \in \ZZ}
\chi |_d(\frac{nd}{N}) f(n + N\rra) e^{-2\pi i n\rrc} \nonumber \\
& = & 
\sqrt{C_{N,d}} e^{2 \pi i N z}\sum_{\tilde{n} \in \ZZ}\chi |_d(\tilde{n})
f(N(\rra + \frac{\tilde{n}}{d}))
e^{-2 \pi i N\rra \rrc}  e^{-2\pi i\frac{N}{d}\tilde{n}\rrc}.
\end{eqnarray}
We compute the partial  Fourier expansion
\beql{e706}
\hR(F)(\rra, \rrc, z)= e^{2 \pi i N z} 
\sum_{m \in \ZZ} h_m(\rrc) e^{2 \pi i m \rra},
\eeq
with Fourier coefficients
$$
h_m(\rrc) := \int_{0}^{1} \hR(F)(\rra, \rrc, 0) e^{- 2 \pi i m \rra} d \rra.
$$
We obtain 
\begin{eqnarray}~\label{e707}
h_m(\rrc) & = &\sqrt{C_{N,d}}  \sum_{\tilde{n} \in \ZZ} \chi |_d (\tilde{n})
e^{-2 \pi i \frac{N}{d} \tilde{n} \rrc} \int_{0}^{1}
f(N( \rra + \frac{\tilde{n}}{d}))e^{- 2 \pi iN \rra \rrc}
e^{-2\pi i m \rra} d\rra  \nonumber \\
& = & \sqrt{C_{N,d}} \sum_{\tilde{n} \in \ZZ} \chi |_d (\tilde{n}) e^{ 2 \pi i \frac{\tilde{n} m}{d}}
\int_{0}^{1}f(N( \rra + \frac{\tilde{n}}{d}))
e^{-2 \pi iN (\rra + \frac{\tilde{n}}{d}) (\rrc + \frac{m}{N})} d \rra.
\end{eqnarray}
Splitting  the sum on the right into  residue classes $ \tilde{n} \equiv k (\bmod~d)$ gives
\begin{eqnarray} 
h_m (\rrc) & =  & \sqrt{C_{N,d}}\sum_{k=1}^{d} \chi |_d (k) e^{2 \pi i \frac{k m}{d}}
\int_{- \infty}^{\infty} f(N(a + \frac{k}{d}))
e^{-2 \pi i N (a + \frac{k}{d})(c + \frac{m}{N})} d \rra \nonumber \\
& = & \sqrt{C_{N,d}}\sum_{k=1}^d \chi |_d (k) e^{2 \pi i \frac{k m}{d}}
\int_{- \infty}^{\infty} f(N\tilde{\rra}) 
e^{ - 2 \pi i N\tilde{\rra}(c + \frac{m}{N})} \frac{d (|N|\tilde{\rra})}{|N|} 
\nonumber \\
& = & \sqrt{C_{N,d}} G(m, \chi |_d)\cdot 
\frac{1}{|N|}\sF(f)(\rrc + \frac{m}{N})  \nonumber \\
& = & \sqrt{C_{N,d}} G(m, \chi |_d)\cdot 
\frac{1}{|N|}\sF(f)(\rrc + \frac{m}{N})  \nonumber .
\end{eqnarray} 
Now 
$\sF(f)(\rrc + \frac{m}{N})= |N|^{1/2} \hU(\frac{1}{N})(\sF(f))(\rrc + m)$
and, using the fact that for all $t \in \RR^{\ast}$,
$\hU(\frac{1}{t}) \circ \sF = \sF \circ \hU(t)$
on $L^2(\RR, dx)$,
we obtain
\begin{eqnarray} ~\label{e708}
h_m(\rrc) &=& \sqrt{C_{N,d}} \frac{G(m, \chi |_d)}{|N|^{\frac{1}{2}}} 
\left(\hU(\frac{1}{N}) \circ \sF \right)(f)( m + N\rrc) 
\nonumber \\
& = & \sqrt{C_{N,d}}\frac{G(m, \chi |_d)}{|N|^{\frac{1}{2}}}  
\left( \sF \circ \hU(N) \right) (f)( m + N\rrc).
\end{eqnarray}

Now we  group  terms in the Fourier expansion \eqn{e705}
according to the  value of $\tilde{d}= \frac {N}{(m, N)},$ noting that 
$(m, N) = \frac{N}{\tilde{d}}$. 
We will apply 
Proposition~\ref{pr71}
to evaluate the Gaussian sums, and for this we set 
$$
d' = \frac{d}{(m,d)} = \frac{d}{(|N|/\tilde{d}, d)} ~~\mbox{and}~~ 
m' = \frac{m}{(m,d)} = \frac{m}{(|N|/\tilde{d}, d)},
$$
Proposition ~\ref{pr71}  states, if $\fe | d'$, that
\beql{e711}
G(m, \chi |_d)=\left( \frac{\phi(d)}{\phi(d')} \mu ( \frac{d'}{\fe})
\chi(\frac{d'}{\fe}) \cdot \bar{\chi} ( m'),\right) \tau(\chi), 
\eeq
while $G(m, \chi |_d)=0$ if $\fe \nmid d'$.

We will show that, when $\fe | d'$, that 
\beql{e709}
G(m, \chi |_d) = C_{N,d}^{\ast}(\tilde{d}, \chi) 
 \cdot \bar{\chi} |_{\tilde{d}}( \frac{m \tilde{d}}{N} ).
\eeq
with 
\beql{e710}
C_{N,d}^{\ast}(\tilde{d}, \chi) : = 
\frac{\phi(d)}{\phi(d')} \mu ( \frac{d'}{\fe})
\chi(\frac{d'}{\fe}) \bar{\chi}(\frac{Nd'}{d \tilde{d}}). 
\eeq

We will need to relate $\tilde{d}$ and $d'$, 
and now show  that $\fe | d'$ implies that $\fe | \tilde{d}$, or
equivalently, that $\fe \nmid \tilde{d}$ implies  $\fe \nmid  d'$,
as remarked after the theorem statement. So suppose $\fe | d'$,
and let a prime
 $p | \fe$ , and set $p^{\fe_p} || \fe, p^{f_p} || f, p^{h_p} || N, $
so that $1 \le \fe_p \le f_p \le h_p$. 
Now $\fe | d'$ gives $ord_p( (\frac{N}{\tilde{d}}, d)) \le f_p - \fe_p < f_p$
so that $ord_p(\frac{N}{\tilde{d}}) =   ord_p((\frac{N}{\tilde{d}}, d))$,
which yields 
$ord_p(\tilde{d}) \ge h_p - (f_p - \fe_p) = \fe_p + (h_p - f_p) \ge \fe_p$.
Since this holds for all primes dividing $\fe$, we have  $\fe | \tilde{d}$.

Next, define $\tilde{g} := \frac{Nd'}{d \tilde{d}}$ 
and note that the definition of $d'$ yields 
$ \frac{N}{\tilde{d} \tilde{g}} := (\frac{|N|}{\tilde{d}}, d),$
which shows  that $\tilde{g}$ is an integer. 
We have 
$m' =  \frac{m \tilde{d}}{N} \cdot \tilde {g},$
and since  $\frac{m \tilde{d}}{N}$ is an integer, 
$$
\bar{\chi}(m') = \bar{\chi}(\frac{m \tilde{d}}{N}) \bar{\chi}( \tilde {g}).
$$
Whenever   $\fe | \tilde{d}$, the relation
$( \frac{m \tilde{d}}{N}, \tilde{d}) = 1$ yields
\beql{e712}
\bar{\chi}(m') = 
\bar{\chi} |_{\tilde{d}}(\frac{m \tilde{d}}{N}) \bar{\chi}( \tilde{g})~.
\eeq
We showed above that  $\fe | d'$ 
 implies $\fe | \tilde{d}$, so we can combine  this with \eqn{e711},
to deduce \eqn{e709}.

We next note that
\beql{e713}
\sW_{N, \tilde{d}}(\bar{\chi} |_{\tilde{d}})
\left( \sF \circ \hU(N) (f) \right) =
\sqrt{ C_{N, \tilde{d}} }
\sum_{m \in \ZZ} \bar{\chi} |_{\tilde{d}}( \frac{m \tilde{d}}{N})
(\sF \circ \hU(N))(f)(m+N\rrc) e^{2 \pi i m \rra}.
\eeq
The nonzero terms in the sum on the right all
have $(m, N)= \frac{|N|}{\tilde{d}}.$ 
(Indeed, we
 must have $\frac{|N|}{\tilde{d}} | (n, N)$ for $\frac{m \tilde{d}}{N}$
to be an integer, and if any further factor of $\tilde{d}$ divides
$(m, N)$ then the imprimitive character $\chi |_{\tilde{d}}$ vanishes.)
Comparing this formula with \eqn{e708}
on those terms with $(m, N)= \frac{|N|}{\tilde{d}}$
we find that they agree
up to a multiplicative scale factor, given by multiplying  
\eqn{e713} by
$$
C_{N, d} (\tilde{d}, \chi) := \sqrt{ \frac{ C_{N,d}} {C_{N, \tilde{d}}} }
 C_{N, d}^{\ast} (\tilde{d}, \chi) =
\sqrt{\frac{\phi(\tilde{d})}{\phi (d)}} 
C_{N, d}^{\ast} (\tilde{d}, \chi).
$$
Combining this with   \eqn{e710}  yields \eqn{e704} and completes the proof.
\end{proof}

%
%

\subsection{Action of $\hR$-operator on coarse multiplicative decomposition}\label{sec83}
The simplest case of Theorem~\ref{th71} is the case 
of a primitive character $\chi~(\bmod~|N|)$ 
in which case 
$\hR( \sH_{N, |N|} (\chi)) = \sH_{N,|N|} (\bar{\chi}),$
and \eqn{e703}
simplifies to
\beql{817a}
\hR \left( \sW_{N,|N|}(\chi)(f) \right) = \epsilon(\chi) 
 \sW_{N,|N|}(\bar{\chi})\left( \sF \circ \hU(N)(f) \right), 
\eeq
where $\epsilon(\chi)$ is given by  
$$
\epsilon(\chi) :=  \frac{\tau(\chi)}{|N|^{\frac{1}{2}}},
$$
and satisfies $|\epsilon(\chi)| =1$ and 
$\epsilon(\chi) \epsilon(\bar{\chi}) = 1$.
More generally, we obtain  the following result.
\smallskip
%

\begin{theorem}~\label{th72}
For $N \ne 0$, the Heisenberg-Fourier operator $\hR$
restricted to the invariant subspace $\sH_N$ is a 
unitary operator $\hR_N$ which acts to permute the
Hilbert spaces $\sH_N(\chi; \fe)$ given by the
coarse multiplicative decomposition of $\sH_N$.
It satisfies 
\beql{721}
 \hR_N( \sH_N(\chi; \fe)) = \sH_N(\bar{\chi}; \fe).
 \eeq
\end{theorem}

\begin{proof}
The coarse multiplicative decomposition of $\sH_N$ was
given in Theorem~\ref{th43}.
The map $\hR_N$ is the restriction of $\hR$ to $\sH_N$.
and map $\hR$ is a unitary transformation of $\sH_N$ into itself,
which is onto since $\hR^4 = \bone$.
Theorem~\ref{th71} and the remark
following its statement
show that when $d|\fe$ the image of $\sH_{N,d} (\chi |_d)$
falls in $\sH_{N}(\bar{\chi}; \fe)$ so we conclude that
$\hR\left(\sH_{N}(\chi, \fe)\right) \subseteq \sH_{N}(\bar{\chi}; \fe)$. 
Since the
image $\hR\left(\sH_{N}(\chi'; \fe')\right)$ falls in
$\sH_{N}(\bar{\chi'}; \fe')$, which    is orthogonal 
to $\sH_{N}(\bar{\chi}; \fe)$ inside $\sH_N$,
we conclude that the map $\hR_N$ on $\sH_{N}(\chi, \fe)$ must be an isometry
onto $\sH_{N}(\bar{\chi}; \fe)$.
\end{proof}

%
%
%

\section{Lerch $L$-Functions viewed as Eisenstein Series}\label{sec9}
\setcounter{equation}{0}

 
The classical non-holomorphic Eisenstein series 
$E(z, s)$
associated to  $SL(2, \ZZ) \backslash SL(2, \RR)$ has three characteristic
properties as an automorphic object.  First, for each fixed $s \in \CC$
it is a  (generalized) eigenfunction
in the $z= x+iy$ variable 
of the non-Euclidean  Laplacian 
$\Delta_{\HH}= y^2 \left( \frac{\partial^2}{\partial x^2} + \frac{\partial^2}{\partial y^2} \right)$
(with eigenvalue $s(s-1)$), and on the 
critical line  $s= \frac{1}{2} + it$
these eigenfunctions 
comprise the continuous spectrum of  $\Delta_{\HH}$
on $SL(2, \ZZ) \backslash SL(2, \RR)$.
Second, for each  fixed $s \in \CC$
it is a simultaneous  eigenfunction of a
commutative algebra of Hecke operators 
acting on  the $z$-variable.  
Third, it has a functional equation in the spectral variable $s$,
relating $E(z, s)$ and $E(z, 1-s)$. In this section we  define
Lerch $L$-functions $L_{N,d}^{\pm}(\chi, s, a, c, z)$ and
their tempered distribution analogues, and show they
possess analogues of all 
three properties. 
%
%

\subsection{Lerch $L$-functions}\label{sec91}

Let $d$ divide $|N|$ and suppose that $\chi$ is a (primitive or imprimitive)
Dirchlet character $(\bmod~d)$.
We apply the intertwining map $\sW_{N,d}(\chi)$ to
functions and operators on $L^2(\RR, dx)$, carrying them
to functions and operators associated to $\sH_{N,d}(\chi)$.
In Appendix B we
 give the resulting correspondence for
$\sH_1 = \sH_{1,1}(\chi_0),$ for the operators discussed in Section \ref{sec8},
plus the two variable Hecke operators and their adjoints.
The dilation group $\hU(t)$ is carried to a group
$\hV(t) : = \sW_{N,d}(\chi) \circ \hU(t) \circ \sW_{N,d}(\chi)^{-1}$ 
of unitary operators on $\sH_{N,d}(\chi)$;  we therefore call
operators on  $\sH_{N,d}(\chi)$  {\em dilation-invariant}  if they
commute with all $\hV(t)$. The action of
the additive Fourier transform $\sF$
on $\sH_1$,  given in Weil~\cite{We64},
was derived in Section \ref{sec6}.

    Now apply  the  intertwining operator $\sW_{N,d}(\chi)$
 to the generalized eigenfunctions\\
  $(sgn(x))^k |x|^{-1/2 + it}$ 
of $x \frac{d}{dx} + \frac{1}{2}$ to obtain, formally,
$$
\sW_{N,d}(\chi)((sgn(x)^k) |x|^{-1/2 + it})(a, c, z) = 
\sqrt{C_{N,d}} e^{2 \pi i N z} \sum_{n \in \ZZ} \chi (\frac{nd}{N}) 
( sgn (n + Nc))^k
e^{2 \pi i n a} |n + Nc|^{-\frac{1}{2} + i \tau}.
$$
in which $\pm = (-1)^k \chi(-1),$ and $C_{N,d} = N/\phi(d)$.
This series converges conditionally in the critical strip $0 < \Re(s) <1$,
for non-integer values of $\rra$
and $\rrc$. (One may split the sum into $n \ge 0$ and $n <0$
and get absolute convergence in the half-plane $\Re(s) > 1$, resp.
$\Re(s) < 0.$)

We use it to make the following definitions of
 Lerch $L$-function
attached to $\sH_{N,d}(\chi)$ in the critical strip, as analytic functions 
in the $s$-variable.\\

%

\begin{defi}~\label{de91}
{ \em
For fixed $s$ in $0 < \Re(s) < 1$ 
and non-integer $\rra$ and $\rrc$ the
{\em Lerch $L$-function} $L_{N,d}^{\pm}(\chi, s, \rra, \rrc, z)$
 is given by 
\beql{901c}
L_{N,d}^{\pm}(\chi, s, \rra, \rrc, z) := e^{2\pi iNz} L_{N,d}^{\pm}( \chi, s, a, c),
\eeq
in which  the last term is given by the conditionally convergent series,
\beql{901d}
L_{N,d}^{\pm}(\chi, s, a, c) := 
 \sum_{n \in \ZZ} \chi(\frac{nd}{N})(sgn(n+N\rrc))^{k}e^{2 \pi i n \rra}~|n+N\rrc |^{-s}.
\eeq
Here  $\pm = (-1)^k$, with $k=0,1$ We also call $L_{N,d}^{\pm}(\chi, s, a, c)$
a {\em Lerch $L$-function}; it corresponds to setting $z=0$ in \eqn{901c}.
}
\end{defi} 

Now consider the special case $N=1$, where necessarily $d=1$ and
$\chi=\chi_0$ is the trivial character. 
We have 
\beql{901e}
L_{1,1}^{\pm}(\chi_0, s, \rra, \rrc, z) = 
e^{2 \pi i z} \left( \sum_{n \in \ZZ}  (\sgn(n+\rrc ))^k e^{2 \pi i n \rra} |n+\rrc |^{-s} \right)
 = e^{2 \pi i z} L^{\pm}(s, \rra,\rrc),
\eeq
where $L^{\pm}(s, \rra,\rrc)$ are the Lerch functions studied in Lagarias and Li \cite[Theorem 2.2]{LL1},
which for $\Re(s) >0$ and $(a, c) \in \RR \times \RR$ are given by
$$
L^{\pm}(s, a,c) = \sum_{n \in \ZZ} (\sgn(n+c))^k e^{2 \pi i m a} | n+c|^{-s},~~~ \mbox{with}~~ (-1)^k= \pm.
$$

All Lerch $L$-functions can be expressed in terms of the functions
$L^{\pm}(s,a,c)$ treated in Lagarias and Li \cite{LL1}, as follows.\\
%

\begin{lemma}~\label{le91a}
Let $\chi$ be a (primitive or imprimitive) Dirichlet character $(\bmod ~d)$, with
$d$ dividing $N$. Then 
for  fixed $(a,c) \in \RR \times \RR$,  and $\Re(s)>1$ 
there holds 
\beql{902a}
L_{N,d}^{\pm}(\chi, s, \rra, \rrc, z) = e^{2\pi i N z} N^{-s}  \left(
 \sum_{m=0}^{d-1} \chi(m) e^{2 \pi i (\frac{N}{d}) m\rra} L^{\pm}(s, N\rra, \rrc + \frac{m}{d})\right).
\eeq
\end{lemma}

\begin{proof}
In what follows we use the convention that $\chi(r)=0$ if $r$ is not an integer, and $\sgn(0)=0$.
We have, for $\Re(s)>1$, writing $n = \frac{Nj}{d}, $ with $j \in \ZZ$
and $j= ld+m$, with $l \in \ZZ$ and $0\le m \le d-1$, that
\begin{eqnarray*}
L_{N,d}^{\pm}(\chi, s, \rra, \rrc) & := & 
\sum_{n \in \ZZ} \chi(\frac{nd}{N})(sgn(n+N\rrc))^{k}e^{2 \pi i n \rra}~|n+N\rrc |^{-s} \\
&=& 
 N^{-s} \left( \sum_{j \in \ZZ} \chi(j) (\sgn(\frac{Nj}{d} + Nc))^k
 e^{2 \pi i \frac{Nj}{d} a} |\frac{j}{d} + c|^{-s} \right)\\
&=& N^{-s} \left( \sum_{m=0}^{d-1} \chi(m)  \left( \sum_{l \in \ZZ}
(\sgn(l+(\frac{m}{d} + c))^k e^{ 2\pi i ( \frac{m}{d} + l)Na} |l+ (\frac{m}{d} + c)|^{-s}\right) \right)\\
&=& 
N^{-s} \left( \sum_{m=0}^{d-1} \chi(m) e^{2 \pi i \frac{m}{d} (Na)} L^{\pm}(s, Na, c+ \frac{m}{d})
\right),
\end{eqnarray*}
giving the result.
\end{proof}

From this lemma we deduce an analytic continuation of $L_{N,d}^{\pm}(\chi, s, a,c)$
in the $s$-variable, and that these  functions satisfy the``twisted periodicity"
conditions needed to belong to $\sH_N$.\\
%

\begin{theorem}~\label{th91a}
Let $N \ne 0$ with $d|N$ and let $\chi~(\bmod ~d)$ be a Dirichlet character.
Then for fixed  $(a, c) \in \RR \times \RR$ the function
$L_{N,d}^{\pm}(\chi, s, a, c)$  analytically continues
to a meromorphic function of $s$, whose only singularities are a possible
simple pole at $s=0$, which may only occur if $a$ is an integer,
and a possible simple pole at $s=1$, which may only occur if $c$ is an integer.
It  satisfies the ``twisted periodicity" conditions
\begin{eqnarray}
L_{N,d}^{\pm}(\chi, s, \rra + \frac{d}{N}, ~~\rrc,z) &=& ~~~~~~~~~L_{N,d}^{\pm}(\chi, s,\rra,\rrc,z) \label{903a} \\
L_{N,d}^{\pm}(\chi, s, ~~\rra~,~\rrc+1,z) &=& e^{-2\pi i N \rra} L_{N,d}^{\pm}(\chi, s, \rra, \rrc,z). \label{903b}
\end{eqnarray}
The first of these implies that
\beql{903c}
L_{N,d}^{\pm}(\chi, s, \rra + 1, \rrc, z) = L_{N,d}^{\pm}(\chi, (s,\rra,\rrc,z).
\eeq
\end{theorem} 

\paragraph{Remark.} These twisted periodicity relations imply 
$L_{N,d}^{\pm}(\chi, s, \rra +1, ~\rrc,z)= L_{N,d}^{\pm}(\chi, s, \rra, ~\rrc,z)$,
 so are sufficient to make $L_{N,d}^{\pm}(\chi, s, \rra, ~~\rrc,z)$
a well-defined function on the Heisenberg nilmanifold $\sN_3 = H(\ZZ) \backslash H(\RR)$.

\begin{proof}
The $z$-variable plays no role, so it suffices to prove the result for
$L_{N, d}^{\pm}(\chi, s, \rra, \rrc).$
The meromorphic continuation follows from the right side of \eqn{902a} in Lemma~\ref{le91a}, using
Lagarias and Li \cite[Theorem 2.2]{LL1}.  In fact a stronger result holds: the meromorphic
continuation as stated
holds for the completed functions
$$
\hat{L}_{N, d}^{\pm} (\chi,s ,\rra ,\rrc) := \gamma^{\pm}(s) L_{N,d}^{\pm}(\chi, s, \rra,\rrc),
$$
in which $\gamma^{\pm}(s)$ is the Tate gamma function, which is meromorphic
and has no zeros. This implies that $L_{N,d}^{\pm}(\chi, s,\rra,\rrc) $ must have ``trivial zeros" at the appropriate set of negative integers.

The ``twisted periodicity conditions" \eqn{903a} and \eqn{903b} now follow from the
``twisted periodicity conditions" for $L^{\pm}(s,a,c)$ in \cite[Theorem 2.2]{LL1},
\begin{eqnarray*}
L^{\pm}(s, \rra+1, \rrc) &=& L^{\pm}(s, a, c) \\
L^{\pm}(s, \rra, \rrc+1) &=& e^{-2\pi i  \rra} L^{\pm}(s, \rra,\rrc)
\end{eqnarray*}
substituted into \eqn{902a}.
\end{proof}

For fixed non-integer $s \in \CC$ the
Lerch $L$-functions $L_{N,d}^{\pm}(\chi, s, \rra, \rrc, z)$
are continuous and real-analytic on $H(\RR)$ except possibly at 
values $(\rra, \rrc, z)$ 
where at least one of  $\rra= \frac{l}{N}$ or $\rrc=\frac{m}{N}$ holds for integer $l,m$.
This follows from the right side of \eqn{902a}, using the fact tha
$L^{\pm}(s, \rra,\rrc)$ has discontinuities only at integer values of $\rra$ and $\rrc$,
see \cite[Theorem 2.3]{LL1}.
In addition, for $0 < \Re(s) < 1$ the
Lerch $L$-functions $L_{N,d}^{\pm}(\chi, s, \rra, \rrc, z)$
 are locally  $L^{1}$-functions on $H(\RR)$,
as a consequence of the same result for $L^{\pm}(s, a,c)$, in  \cite[Theorem 2.4]{LL1}.
However for fixed $s$ outside  of $0 \le \Re(s) \le 1$ this function 
is not locally $L^{1}$ on $H(\RR)$.
For this paper it is sufficient to know these functions are 
analytic in the (open) critical strip $0 < \Re(s) < 1$,
which is derivable from the Dirichlet series representation above.  

Note that each  Lerch $L$-function possess a reflection symmetry
\beql{901f}
L_{N,d}^{\pm}( \chi, s, -\rra, -\rrc, z) = 
\pm \chi(-1) L_{N,d}^{\pm}( \chi, s, \rra, \rrc, z).
\eeq
This follows directly from the definition \eqn{901c}, using a  change of 
the summation variable from $n$ to $-n$.

There are also simple rescaling relations relating
Lerch  $L$-functions at different levels
having the same Dirichlet character. 
The following result  relates functions associated to 
the Heisenberg modules $\sH_N$ and $\sH_{N'}$, 
when $|N^{'}|~ \mid N $.\\

%

\begin{theorem}~\label{th91}
Let $N \ne 0$ with $d \mid  N$.
Suppose that $N'$ is a (positive or negative) integer
satisfying  $d \mid N' $ and $|N^{'}| \mid |N|$,
Then for any Dirichlet character $\chi (\bmod~ d)$, there holds
\beql{n900}
L_{N, d}^{\pm} (\chi, s, \rra, \rrc, z) = (\sgn (\frac{N}{N'})^{k}
|\frac{N}{N'}|^{-s} 
L_{N', d}^{\pm}(\chi, s, \frac{N}{N'} \rra, \rrc, \frac{N}{N'}z),
\eeq
where $\pm = (-1)^k.$ In particular,
\beql{N900b}
L_{-N, d}^{\pm} (\chi, s, \rra, \rrc, z) = 
\pm L_{N, d}^{\pm} (\chi, s, -\rra, \rrc, -z).
\eeq
\end{theorem}

\begin{proof}
This result parallels Lemma~\ref{le42}.
For $0 < \Re(s) < 1$ the definition of Lerch $L$-function,
taking $\tilde{n}= \frac{nd}{n}$, yields
$$
L_{N,d}^{\pm}(\chi, s, \rra, \rrc, z) =  |N|^{-s} e^{2\pi iNz}
\sum_{\tilde{n} \in \ZZ} \chi(\tilde{n}) 
(sgn(N))^{k}(sgn(\frac{\tilde{n}}{d}+c))^{k}
e^{2 \pi i \tilde{n}\frac{N}{d}  a}~|\frac{\tilde{n}}{d} +c|^{-s}. 
$$
On the other hand,
\begin{eqnarray}
L_{N',d}^{\pm}(\chi, s, \frac{N}{N'}\rra, \rrc,\frac{N}{N'} z) &= &
e^{2\pi i N' (\frac{N}{N'}z)}
\sum_{n \in \ZZ} \chi(\frac{nd}{N'})(sgn(n+N'c))^{k}
e^{2 \pi i n ( \frac{N}{N'}\rra)}~|n+N'c|^{-s} \nonumber \\
& = & |N'|^{-s} e^{2\pi i Nz}  \sum_{\tilde{n} \in \ZZ} 
\chi(\tilde{n})((sgn(N'))^{k}(sgn(\frac{\tilde{n}}{d}+c))^{k}
e^{2 \pi i \tilde{n} (\frac{N}{d} a)}~|\frac{\tilde{n}}{d}+c|^{-s}.
\nonumber
\end{eqnarray}
The theorem follows for $0 < \Re(s) < 1$ on comparing  these
two formulae.
It then  holds for all complex $s$ by analytic continuation.
\end{proof}

%
%

\subsection{Generalized Eigenfunctions of $\Delta_L$}\label{sec92}

The ``Laplacian operator'' $\Delta_L$  is
a left-invariant differential operator
on the Heisenberg group, given in \eqn{800aa} below.
It acts on all Heisenberg modules $\sH_N$, including
$\sH_0$. On $\sH_{N, d}(\chi)$
it can be identified with
the image of a (properly scaled)  dilation-invariant operator 
$D = x \frac{d}{dx} + \frac {1}{2} \bone$
under the  twisted Weil-Brezin map $\sW_{N,d}(\chi)$,
and this  permits the results of Section \ref{sec5} to be applied.\smallskip

For  $\sH_{N,d}(\chi)$
the  dilation-invariant  differential operator 
$N(x \frac{d}{dx} + \frac{1}{2}) $
is carried to the differential operator
\begin{eqnarray}~\label{800a}
\Delta_L &:= & \sW_{N,d}(\chi) \circ N(x \frac{d}{dx} + \frac{1}{2}) \circ
\sW_{N,d}(\chi)^{-1} \nonumber \\
& = & \frac{1}{2 \pi i} \frac{\partial}{\partial a} 
\frac{\partial}{\partial c } + Nc \frac{\partial}{\partial c } + \frac{N}{2}. 
\end{eqnarray}
This operator coincides with 
\begin{equation}~\label{800aa}
\Delta_L = \frac{1}{4 \pi i}(\hX  \hY + \hY  \hX)
\end{equation}
in which
$$
\hX := \frac{\partial}{\partial a} + 
  c  \frac{\partial}{\partial z},~~~ \hY := \frac{\partial}{\partial c},
$$
are left-invariant differential operators on the 
(non-symmetric) Heisenberg group. This definition of $\Delta_L$ is
intrinsic on the Heisenberg group $H(\RR)$, and it also makes sense on
the ``degenerate'' Heisenberg module $\sH_0$.\smallskip

From the definition \eqn{901c} one checks that 
for any fixed value
of $s$ with $0 < \Re(s) < 1$
the  Lerch $L$-functions 
$L_{N,d}^{\pm}(\chi, s, \rra, \rrc,z)$ 
formally are generalized eigenfunctions of $\Delta_L$
acting on $H(\RR)$, 
with eigenvalue $-(s- \frac{1}{2})$. 
The eigenfunction property obviously 
 holds for each term separately in the expansion
\eqn{901c}.
The  generalized eigenfunction property can be extended to 
all $s \in \CC$ if one uses tempered distributions,
using the notion of Lerch $L$-distribution defined in Section \ref{sec93}.
(Here we note that the analytically continued version of the Lerch
$L$-function for a fixed $s \in \CC$ is differentiable
in $a$ and $c$ away from the singular set where $a$
or $c$ are integers, and there
satisfies the eigenfunction equation with
eigenvalue $-(s- \frac{1}{2})$. One can  establish  the generalized eigenfunction property
at the singular set using tempered distributions.\smallskip

The  linear partial differential
operator $\Delta_L$ has several features of  a  classical Laplacian operator,
but also has some differences.
A {\em  classical
Laplacian} is  a differential operator  in the center
of the universal enveloping algebra of some real Lie algebra. Here the
operator $\Delta_L$ 
a left-invariant operator in the Heisenberg Lie algebra, 
but  is not right-invariant, so it is not in the center of
the universal enveloping algebra.  




In the classical $SL(2)$ case Eisenstein series on the critical line are generalized eigenfunctions
of the non-Euclidean Laplacian operator. The following result interprets
Lerch $L$-functions as an ``Eisenstein series" in a similar  sense.
%

\begin{theorem}~\label{th95} {\em (Eisenstein Series Interpretation of Lerch $L$-functions)}\\
Let $N \ne 0$ and $d \ge 1$ with $d \mid N$.

$~~~~~$ (1) Consider  the unbounded
 operator $\Delta_L= \frac{1}{2\pi i} \frac{\partial}{\partial \rra} 
\frac{\partial}{\partial \rrc} + N \rrc  \frac{\partial}{\partial \rrc}
+ \frac{N}{2}$
 on the dense domain 
 \beql{932aa}
 \sD_{N,d}(\chi) := \sW_{N,d}(\chi) (\sD) \,  
  \eeq
 in the Hilbert space $\sH_{N,d}(\chi)$, in which $\sD$ denotes  the maximal domain for 
$D= x \frac{\partial}{\partial x}+ \frac{1}{2} \bone$
on $L^2(\RR, dx)$. 
The operator $(\Delta_L, \sD_{N,d}(\chi))$
commutes with all elements of the unitary group 
$\{ \hV(t): ~ t \in \RR^{\ast} \}$.

(2) The operator
$(\Delta_L, \sD_{N,d}(\chi))$ 
is skew-adjoint on $\sH_{N,d}(\chi)$,
and its associated spectral multiplier
function  on $L^2(\ \ZZ/ 2\ZZ \oplus \RR, d\tau)$
is $a_0(\tau)= -i \tau$ and  $a_1(\tau_1) = -i \tau_1$.

(3) The two families of Lerch $L$-functions 
$L_{N,d}^{\pm}( \chi, \frac{1}{2} + i\tau, a, c,z)$,
parameterize the (pure) continuous spectrum 
of  $(\Delta_L, \sD_{N,d}(\chi))$ on $\sH_{N,d}(\chi)$, 
giving   a complete set of generalized eigenfunctions,
as $\tau$ varies  over $\RR$.
All functions  $F(a, c, z)$ in the dense subspace  $\sS(\sH_{N,d}(\chi))$ have
a convergent spectral representation 
\begin{eqnarray*}
F(a, c, z)  &=&  \frac{1}{4\pi} \int_{-\infty}^\infty 
\hat{F}^{+}(\frac{1}{2}+ i\tau)L_{N,d}^{+}( \chi, \frac{1}{2} - i\tau, a, c.z)
d \tau  \\
&&\quad\quad\quad+
\frac{1}{4\pi} \int_{-\infty}^\infty\hat{F}^{-}(\frac{1}{2}+ i\tau)
L_{N,d}^{-}( \chi, \frac{1}{2} - i\tau, a, c.z) d\tau, 
\end{eqnarray*}
in which 
$\hat{F}^{+}(s) = \sM_0( \WNd(\chi)^{-1}(F))(s)$
and $\hat{F}^{-}(s) = \sM_1( \WNd(\chi)^{-1}(F))(s)$.
\end{theorem}

\begin{proof}
The  main idea of the proof is that the nice properties of the operators
follow  from the property of commuting with the group of unitary dilations
on a suitable dense domain in the Hilbert space which is invariant
under all the operators involved. 
It transforms the problem to that of a dilation-invariant
 operator on $L^2(\RR, dx)$, viewed  as a rigged Hilbert space,
with  Schwartz functions $\sS(\RR)$ as the smallest  class in the triple.

(1) Using the Weil-Brezin transform 
we pull  back the problem
from each Hilbert space $\sH_{n, D}(\chi)$ to a problem on $L^2(\RR, dx)$,
in which  the operator $\Delta_L$ is transformed to $x \frac{d}{dx} + \frac{1}{2}$.
In $L^2(\RR, dx)$ we take 
the Schwartz space $\sS(\RR)$ as our  dense domain, since
 the domain  $\sS(\sH_{N,d}(\chi))$ is the
push-forward under the Weil-Brezin map $\sW_{N,d}(\chi)$ of $\sS(\RR)$.
 The group  of  dilation operators $\hV(t)$  is pulled back 
to the group of dilations $\hU(t)$ on $L^2(\RR, dx)$.
The Schwartz space $\sS(\RR)$ is invariant under dilations,  
 the  operator $D= x \frac{d}{dx} + \frac{1}{2}$ is dilation-invariant
and preserves this  this space. 

(2), (3) 
We  apply results of Burnol \cite[Theorems 2.2 and 2.4]{Bu01} concerning 
  dilation invariant
operators on $L^2 [ \RR_{\ge 0}, \frac{dx}{x}]$ (the group $G = \RR_{>0}^{*}$),
which can be  transferred to  $L^2(\RR, dx)$
by an  inverse Mellin transform identity. 
Applied to the operator $iD$ it gives a unique self-adjoint extension domain
and an absolutely continuous  spectral measure
as specified. These issues are discussed in Appendix B, where
Proposition \ref{prop83} supplies  an answer to both (2) and (3)
for $x \frac{d}{dx} + \frac{1}{2}$.
\end{proof}

%
%

\subsection{Hecke Operator Eigenfunctions}\label{sec93}

We now  characterize  Lerch $L$-functions
for fixed $s \in \CC$ as 
simultaneous generalized eigenfunctions of the set of
two-variable Hecke operators $\hT_m, m \ge 1.$ 
This parallels the second  property of classical Eisenstein
series mentioned above.
We give complete details for $L^{\pm}(\chi_0, s, \rra, \rrc, z)$
and then sketch the result for general Lerch $L$-functions.

We obtain the result using a form of the inverse
Weil-Brezin map to pull back the question from the
Heisenberg group to the real line, and formalize
our results in terms of tempered distributions.
The basic observation is that $L^{\pm}(\chi_0, s, \rra, \rrc, z)$
is the formal image under the Weil-Brezin m
of the quasicharacter $\sgn(x)^k |x|^{-s}$; we
justify this for tempered distributions
when $0 < \Re(s) < 1$. Let $\sS'(\RR)$ denote the
space of tempered distributions, the dual space
to Schwarz functions $\sS(\RR)$, characterized as
continuous linear functionals $F: \sS(\RR) \to \CC$.
We write $\left<F,  \varphi(x)\right>$ for the value of the
linear functional on $\varphi(x) \in \sS(\RR)$. 
(Note that this scalar product is not a Hilbert
space inner product; it is linear in both arguments
for complex-valued $\varphi(x)$, and in particular is 
not conjugate-linear in the second argument.)
If the distribution corresponds to an $L^1$ function then we have
$$
\langle F, \varphi(x) \rangle := 
\int_{- \infty}^{\infty} F(x) \varphi(x) dx.
$$
For $t \in \RR^{\ast}$ we define a dilation action
$\hU(t): \sS'(\RR) \to \sS'(\RR)$ by defining for a
distribution $F$ its image $U(t)(F)$ by
\beql{801}
\langle  \hU(t)(F), \varphi(x) \rangle := 
\langle F, \hU(\frac{1}{t})\varphi(x)\rangle = 
\langle F, |t|^{-1/2} \varphi ( \frac{x}{t}) \rangle,
\eeq 
using the $\hU(t)$-action on functions $\varphi(x)$ given in \eqref{601}.
These operators satisfy $\hU(t_1 t_2) = \hU(t_1) \circ \hU(t_2)$
and $U(0) = \bone.$ 

A multiplicative quasicharacter on $\RR^\ast$ has the form
$\chi(t) = (\sgn(t))^k |t|^s$ with $s \in \CC$, $k = 0, 1.$
We write $\chi_{+}(t)  := 1$ for the trivial character and
$\chi_{-}(t) = \sgn(t)$ for the sign character.
A tempered distribution $\Delta$ is said to be {\em homogeneous}
\footnote{Homegeneous distributions were
treated by Gel'fand and Sapiro \cite{GS55} and
G{\aa}rding \cite{Ga61}. Burnol~\cite[p. 16]{Bu98} states that a tempered distribution
$\Delta$ has homogeneity $\chi$ if $\Delta(xt) = \chi(x) |x|^{-1}\Delta(t)$,
i.e. if for all Schwartz functions 
$\int_{-\infty}^\infty \Delta(t)\phi(t)dt = \chi(x) \int_{-\infty}^\infty
\Delta(t) \phi(xt)dt.$
}
with quasicharacter $\chi$ if
\beql{802}
U(t)(\Delta) = \chi(t) |t|^{- 1/2} \Delta ~~~\mbox{for all}~~ t \in \RR^\ast.
\eeq 
A  tempered distribution $\Delta$ is {\em even} if
$U(-1)(\Delta) = \Delta$ and is {\em odd} if $U(-1)\Delta = - \Delta.$\\
%

\begin{prop}~\label{pr91}
{\em (Weil)}
For each quasicharacter $\chi(t) = (\sgn(t))^k |t|^s$ on $\RR^{\ast}$
there exists a tempered distribution $\Delta(x)$ on $\RR$
having homogeneity $\chi$.
This distribution on $\RR$ is unique up to multiplication
by a nonzero constant.
\end{prop}

\paragraph{Proof.} This appears in  Weil~\cite{We66}.
Weil's  paper treats more general situations, including
homogeneous distributions on
all local fields, and on adeles. ~~$\bsq$\\

The homogeneous distributions, excluding a countable number of
values of $s$, can be grouped into two families $\Delta_s^{+}$
and $\Delta_s^{-}$ associated to
$|t|^s$ and $\sgn(t)|t|^s$, respectively,
which are meromorphic in the parameter
$s \in \CC$, in the sense that for each test function $\varphi \in \sS(\RR)$,
the function
\beql{803}
f_{\varphi}^{\pm}(s) := \langle \Delta_s^{\pm}, \varphi(x) \rangle
\eeq
is a meromorphic function of $s$. 
Furthermore the functions 
$f_{\varphi}^{\pm}(s)$
have at most simple poles, which can occur only at the 
quasicharacters where the associated 
homogeneous distribution 
is {\em local}, which means supported at the point $\{0\}$. 
Some cases of local distributions are that of 
homogeneity $\chi_0$, where the distribution
 is the Dirac delta function $\delta_0$ at $x =0$,
and those of homogeneity
$|t|^{-2n}$ for $n \ge 1$, where the distribution
 is the $n$-th  derivative of the
Dirac delta function $\delta_0^{(n)}$ at $x=0$.\\

For $\Re(s) > 0$ the function $x \mapsto |x|^{s-1}$ is locally integrable,
and we take
\beql{804}
\Delta_s^{\pm} := (\sgn (x))^k |x|^{s - 1} ~~~\mbox{with}~~  (-1)^k= \pm,
\eeq
where the distribution is defined by
\beql{805}
\langle \Delta_s^{\pm}, \varphi(x)\rangle := \int_{-\infty}^{\infty}
\varphi(x) (\sgn (x))^k |x|^{s - 1} dx.
\eeq
The distributions $\Delta_s^{\pm}$
form an  analytic family in the region $\Re(s) > 1,$ and 
Weil showed they analytically 
continue to a  meromorphic family on $s \in \CC$. 
The polar divisors are exactly the values of $s$
where the homogeneous distribution is local, and are
a subset of $\ZZ$.\\

The Hecke operator $\hT_m^{\chi_0}(f)(x) := f(mx)$
when regarded as a linear operator  acting 
on locally integrable functions of moderate growth at $\pm \infty$
has a unique extension to  a (continuous)  operator,
also denoted  $\hT_m^{\chi_0}$, acting 
on tempered distributions  $\sS'(\RR)$,
defined by 
\beql{806a}
\langle  \hT_m^{\chi_0}(F), \varphi(x) \rangle := 
\langle F, (\hT_m^{\chi_0})^{\ast}(\varphi)(x)\rangle,~~~ \varphi(x) \in \sS(\RR),
\eeq
in which
\beql{806b}
(\hT_m^{\chi_0})^{\ast}(\varphi)(x) := 
\frac{1}{m} \sum_{j=0}^{m-1} \varphi(\frac{x+j}{m}). 
\eeq
Using Theorem~\ref{th42} one deduces that
$\hT_m^{\chi_0} = |m|^{-1/2} \hU(m)$, where $\hU(m)$ is the
$\RR^\ast$-action on tempered distributions given above. \\
%

\begin{theorem}~\label{th97}
Let $d \ge 1$ be an integer, and for  each $s \in \CC$ fixed,
let $\sE_s$ be 
the vector space consisting  of those tempered distributions $\Delta$ 
on $\RR$ such that
\beql{809}
{\hT}_m^{\chi_0}(\Delta) = m^{-s} \Delta, ~~~\mbox{for all}~~ 
 m \equiv 1~(\bmod~ d)~~ \mbox{with}~~ m \ge 1.
\eeq
Then $\sE_s$ is two-dimensional, and is independent of $d$. It 
is spanned by an even homogeneous tempered distribution of homogeneity
$|t|^{1-s}$ and an odd homogeneous tempered distribution of
homogeneity $\sgn(t)|t|^{1-s}.$ For $s \not\in \ZZ$,
these distributions can be taken to be $\Delta_{1 - s}^{+}$ 
and $\Delta_{1 - s}^{-}$, respectively. 
\end{theorem}

\begin{proof} The condition \eqn{809} is
satisfied by homogeneous distributions of the two types,
as is evident from the relation of $\hT_m^{\chi_0}$ to $U(m)$.
By Proposition \ref{pr91}
these generate a two-dimensional
vector space $\tilde{\sE}_s$ of tempered distributions,
independent of $d$, with $\tilde{\sE}_s \subseteq \sE_s$.
It remains to  show that $\tilde{\sE}_s =  \sE_s$.

Given $\Delta \in \sE_s$, we must show that 
$\Delta \in   \tilde{\sE}_s.$
We first show that \eqn{809} implies that, for all $t > 0$, 
\beql{810}
\hU(t) (\Delta) = |t|^{-(s - \frac{1}{2})} \Delta.
\eeq
Since $\hT_m^{\chi_0} = |m|^{-1/2}\hU(m),$ 
\eqn{809} gives 
$$
\hU(m) (\Delta) = |m|^{-(s - \frac{1}{2})} \Delta
$$
for $m \equiv 1~ (\bmod~d)$, $m > 0$. Letting $\hU(\frac{1}{m})$
operate on both sides of this equation yields
$$
\hU(\frac{1}{m}) (\Delta) = m^{s - \frac{1}{2}} \Delta.
$$
Now
$$
\hU(m_1) \circ \hU(\frac{1}{m_2})(\Delta) = \hU(m_1) 
(|m_2|^{s -  \frac{1}{2}}\Delta) = 
(\frac{m_1}{m_2})^{-(s - \frac{1}{2})} \Delta.
$$
Thus \eqn{810} holds for the values
$\{ t= \frac{m_1}{m_2}: m_1 \equiv m_2 \equiv 1~(\bmod~ d) \}$
which are dense in $\RR_{>0}$. The relation \eqn{810} now holds by
a limiting process. Take a seqence $t_k = \frac { m_{1,k} }{ m_{2,k} }$
with $t_k  \to t$
and one has
$$
\left< \hU(t_k)(\Delta), \varphi(x) \right> = 
\left< |t_k|^{-(s- \frac{1}{2})}\Delta, \varphi(x)\right> \to 
\left< t^{-(s- \frac{1}{2})} \Delta, \varphi(x) \right>, 
$$
while also
\begin{eqnarray*}
\langle \hU(t_k)(\Delta), \varphi(x) \rangle &= &
\langle \Delta, \hU(\frac{1}{t_k}) \varphi(x) \rangle
 =  \langle \Delta, |t_k|^{ - \frac{1}{2}} \varphi(\frac{x}{t_k})\rangle \\
& {\longrightarrow}\atop{k \to \infty} & \langle \Delta, |t|^{ - \frac{1}{2}} \varphi(\frac{x}{t}) \rangle \\
& = & \langle \Delta, \hU(\frac{1}{t})\varphi)(x) \rangle =
\langle \hU(t)(\Delta), \varphi(x) \rangle.
\end{eqnarray*}
This verifies \eqn{810}.

To complete the proof we use $\hU(-1)$, which commutes with $\hU(t) (t >0)$
so that
$$
\hU(t)\circ \hU(-1)\Delta = |t|^{-(s - \frac{1}{2})}\hU(-1)\Delta.
$$
It follows that the tempered distribution
$$
\Delta^{+} := \frac{1}{2}\left( \Delta + \hU(-1) \Delta \right)
$$
satisfies $\hU(-1)(\Delta^{+}) = \Delta^{+}$
and, using \eqn{810}
$$
\hU(t) \Delta^{+} = |t|^{-(s - \frac{1}{2})}\Delta^{+},~~\mbox{all}~~ t >0.
$$
Thus $\Delta^{+}$ is a homogeneous tempered distribution with
homogeneity $\chi(x) = |x|^{1 - s}$, hence it lies in $\tilde{\sE}_s.$
Similarly we find that the tempered distribution
$$
\Delta^{-} := \frac{1}{2}\left( \Delta - \hU(-1) \Delta \right)
$$
is homogeneous with homogeneity $\sgn(x)|x|^{1 - s}$, so also
lies in $\tilde{\sE}_s.$ We conclude that 
$\Delta = \frac{1}{2}(\Delta^{+} + \Delta^{-}) \in \tilde{\sE}_s.$\\

Finally we note that $\Delta_{1-s}^{\pm}$ are even (odd) homogeneous
distributions in the appropriate spaces, as required, as long
as $s \not\in \ZZ$, or more generally whenever $s$ is not a polar value
of the meromorphic family.
\end{proof}

We transport the notion of tempered distribution to $\sH_{N, d}(\chi)$
using the  twisted Weil-Brezin map $\sW_{N, d}(\chi)$.
We let $\sS(\sH_{N, d}(\chi))$ denote the image of the Schwartz space
under $\sW_{N, d}(\chi)$. This map is one-to-one 
and we put the induced topology on
$\sS(\sH_{N, d}(\chi))$. We call the 
topological dual space  $\sS^{'}(\sH_{N, d}(\chi))$
the space of {\em tempered distributions} on $\sH_{N, d}(\chi)$.
We obtain  a twisted Weil-Brezin map on tempered distributions
$ \sW_{N, d}(\chi) : \sS^{'}(\RR) \to     \sS^{'}(\sH_{N, d}(\chi))$
which associates to a tempered distribution $\Delta \in \sS(\RR)$
a distribution $\sW_{N, d}(\chi)^{\sharp}(\Delta) \in \sS^{'}(\sH_{N, d}(\chi))$ by
the linear form
$$
\left<
\sW_{N, d}(\chi)^{\sharp}(\Delta), \sW_{N, d}(\chi)(\phi)
 \right> : = 
\left< \Delta, \phi \right>, ~~\mbox{for~all}~~\varphi \in \sS(\RR).
$$
This map is an isomorphism onto.  
We use the ``sharp'' notation 
to distinguish the distribution 
 $\sW_{N, d}(\chi)^{\sharp}(\varphi)$ associated to 
the function $\varphi$
from the function $\sW_{N, d}(\chi)(\varphi)$, because for
$f \in L^1(\RR, dx)$ one computes the linear form as  
$$
\left< \sW_{N, d}(\chi)^{\sharp}(f), \sW_{N, d}(\chi)(\phi)
 \right> : = \int_{0}^1\int_{0}^1\int_{0}^1
\sW_{N, d}(\chi)^{\sharp}(f) \cdot \overline{\sW_{N, d}(\chi)(\bar{\varphi})}
da dc dz,
$$
Two complex conjugations appear in the last term in this
integral because $\sW_{N,d}(\chi)$ respects the Hilbert
space inner product, which  is conjugate-linear in its second
argument.
We define the space $L^1(\sH_{N, d}(\chi))$ to be the space of
locally $L^1$-functions on $H(\RR)$ which transform according to
the relations in $\sH_{N, d}(\chi)$. To each function
$F \in L^1(\sH_{N, d}(\chi))$ there is naturally associated
a tempered distribution in $\sS^{'}(\sH_{N, d}(\chi))$.\\

%

\begin{defi}~\label{def92}
{\em
For arbitrary
$s \in \CC$
the  {\em Lerch $L$-distribution} $L_{N,d}^{\pm}(\chi, s, a, c, z)$ in $\sH_{N, d}(\chi)$
is the tempered distribution $\sW_{N,d}(\chi)( \Delta_{1-s}^{\pm})$
in $\sS'(\sH_{N,d}(\chi))$.
}
\end{defi}

For fixed $s$ in $0 < \Re(s) < 1$ the
Lerch $L$-distribution can
be identified with the
Lerch $L$-function  $L_{N,d}^{\pm}(\chi, s, \rra, \rrc, z)$,
viewed as a locally $L^{1}$-function on $H(\RR)$.
We therefore use the same notation for it, in an
abuse of language. Note that outside $0 \le \Re(s) \le 1$
one cannot necessarily identify the Lerch $L$-distribution with
the (analytically continued) Lerch $L$-function, since 
this function is not locally $L^{1}$ on $H(\RR)$.\\
%

\begin{theorem}~\label{th99}
{\rm ( Hecke operator  termpered distribution eigenspace)}
Let $N$ be a nonzero integer, 
and $\chi$ be a Dirichlet character $(mod~d)$, with
$d | N$. 

(1) For  each fixed $s \in \CC$,
let $\sE_s(\sH_{N, d}(\chi))$ be 
the vector space  of tempered distributions 
$\Delta \in \sS'(\sH_{N, d}(\chi))$  such that
\beql{811}
\hT_m(\Delta) = \chi(m)m^{-s} \Delta, ~~~\mbox{for~all}~~ 
 m \ge 1~~ \mbox{with}~~(m, N) = 1.
\eeq
Then $\sE_s(\sH_{N, d}(\chi))  $ is a two-dimensional
vector space, and
is spanned by an even homogeneous tempered distribution of homogeneity
$|t|^{1-s}$ and an odd homogeneous tempered distribution of
homogeneity $\sgn(t)|t|^{1-s}.$ 

(2) For all non-integer $s \in \CC$ the two Lerch
$L$-distributions $L_{N,d}^{\pm}({\chi}, s, \rra, \rrc, z)$
are nonzero  even and odd homogeneous distributions spanning
 $\sE_s(\sH_{N, d}(\chi))$, respectively. 
For $0 < \Re(s) < 1$ these two distributions
are induced by the Lerch $L$-functions
$L_{N,d}^{\pm}({\chi}, s, a, c, z)$, which  
both lie in $L^1( \sH_{N, d}(\chi))$.
\end{theorem} 

\begin{proof}
We pull the condition \eqn{811} back to 
$L^2(\RR, dx)$ using
the map $\sW_{N,d}({\chi})^{-1}$ acting on
tempered distributions. Any tempered distribution 
$\Delta = \sW_{N, d}(\chi)(\Delta')$ for a (unique)
tempered distribution $\Delta'$,  and the
 analogue of Theorem~\ref{th42} on tempered distributions asserts that
\beql{811b}
\hT_m \circ  \sW_{N, d}(\chi)(\Delta') = 
\chi (m) \sW_{N, d}(\chi)(\hT_m (\Delta')),
\eeq
provided $(m, \frac{N}{d}) = 1$.
This  implies that all elements  of
$\sW_{N,d}({\chi})(\sE_s)$ satisfy \eqn{811}.
In the converse direction, if $\Delta \in \sE_s(\sH_{N, d}(\chi))$
then 
applying $ \sW_{N, d}(\chi)^{-1}$ to \eqn{811} using \eqn{811b} yields
$$
 \chi(m) \hT_m(\Delta') = \chi(m)m^{-s} \Delta'.
$$
when $(m, \frac{N}{d}) = 1$.
For $m \equiv 1~(\bmod N)$ this gives
$$
\hT_m(\Delta') = \chi(m)m^{-s} \Delta',
$$
and Theorem~\ref{th97} applied with $d=N$ 
establishes that $\Delta' \in \sE_{s}$.
We conclude that  
$\sE_s(\sH_{N, d}(\chi)) =  \sW_{N,d}({\chi})(\sE_s)$.
It follows from 
Theorem~\ref{th97} that
 $\sE_s(\sH_{N, d}(\chi))$ has dimension $2$,
and the rest of (1) follows because $\sW_{N,d}(\chi)$
preserves evenness and oddness of distributions.

This argument  also establishes (2), aside from the 
identification of the tempered distribution
with the corresponding Lerch $L$-function when $0 < \Re(s) <1$.
The identification follows from  the property  that these functions are  locally $L^{1}$ on
the Heisenberg group. This local $L^{1}$ property can be established  by 
extending the proof of the local $L^{1}$-property
in \cite[Theorem 2.4]{LL1} to general Lerch $L$-functions.
\end{proof}

\paragraph{\bf Remarks.} 
(1) One can prove a similar result for the adjoint
Hecke operators $(\hT_m)^{\ast}$.
The vector space $\sE_s^{\ast}(\sH_{N, d}(\chi))$ 
 of tempered distributions in
$\sW_{N,d}(\chi)$ such that 
$$
(\hT_m)^{\ast}(\Delta) = \bar{\chi}(m)m^{1-s} \Delta, ~~~\mbox{for~all}~~ 
 m \ge 1~~ \mbox{with}~~(m, N) = 1,
$$
is two-dimensional, with 
$\sE_s^{\ast}(\sH_{N, d}(\chi))= \sE_s(\sH_{N, d}(\chi)).$
To prove this one can use the formula 
$(\hT_m)^{\ast}= \hR^{\ast}\circ \hT_m \circ \hR$
together with Theorem~\ref{th71}. We use the fact that
the action of
$\hT_m$ is constant on  $\sH_N(\bar{\chi}; e)$
as long as $(m, N) = 1.$  \\

(2) Theorem~\ref{th99} is analogous to a result
of Milnor~\cite[Theorem 1]{Mi83}, see also 
Lagarias and Li \cite[Theorem 5.5]{LL4}. 
 In fact Milnor's result can
be interpreted as describing simultaneous
continuous eigenfunctions  of a two-variable
Hecke operator on a certain vector subspace of the
Hilbert space $\sH_0$, as shown in part II.
%
%

\subsection{Generalized Lerch Functional Equations}\label{sec94}

The original Lerch zeta function satisfies
two  symmetrized four-term functional equations relating $s$ to $1-s$,
given in Weil  \cite[p. 57]{We76}   and in Lagarias and Li \cite{LL1}.
Each functional equation encodes
the action of the additive Fourier transform. It can be
derived from that of the homogeneous distributions
$\Delta_s^{\pm}$ on the real line, using the Weil-Brezin
maps $\sW_{N,d}(\chi)$ on tempered distributions.

The additive Fourier transform acts on tempered distributions
by 
$\langle \sF(\Delta), \varphi \rangle = \langle \Delta, \sF(\varphi)\rangle.$
Weil~\cite{We66}
observed that
the  additive Fourier transform takes homogeneous distributions to
homogeneous distributions, with
\beql{806}
\sF(\Delta_{s}^{\pm}) = \gamma^{\pm}(s) \Delta_{1-s}^{\pm}.
\eeq
 in which $\gamma^{\pm} (s)$ is a certain meromorphic function of
$s$, with the data $(\chi_{\pm}, s)$ specifies the homogeneity
type of the homogeneous distribution on the left side
of the equation. 
This  follows for $0 < \Re(s) < 1$  from the identity
valid for Schwartz functions $\varphi \in \sS(\RR)$,
\beql{807}
\int_{- \infty}^\infty \sF(\varphi)(x) (\sgn(x))^k |x|^{s-1} dx =
\gamma^{\pm} (s) 
\int_{- \infty}^\infty \varphi(y) (\sgn(y))^k |y|^{-s} dy.
\eeq
and extends by analytic continuation to $s \in \CC$.
We call the functions $\gamma^{\pm}(s)$ 
{\em Tate-Gel'fand-Graev gamma functions}, following
the terminology of Burnol~\cite{Bu98}, \cite{Bu99}, who named them 
after Tate ~\cite{Ta50} and Gelfand and Graev, cf. ~\cite{GG69}. Recall that they are 
\beql{eq927}
\gamma^{+}(s) = \frac{\pi^{-s/2} \Gamma(\frac{s}{2})}
{ \pi^{-(1-s)/2} \Gamma(\frac{1-s}{2})},~~~\,\,
\gamma^{-}(s) = -i~\frac{\pi^{-(s+1)/2} \Gamma(\frac{s+1}{2})}
{\pi^{-(2-s)/2} \Gamma(\frac{2-s}{2})}.
\eeq
These functions provide the necessary correction factor in a 
nonsymmetric form of the {\em local}
functional equation at the real place. Recall that 
the local Euler factor at the real place is 
$\Gamma_{\RR}(s):= \pi^{-s/2} \Gamma( \frac{s}{2})$,
and the local functional equation can be written 
$$
\Gamma_{\RR}(s) = \gamma^{+}(s) \Gamma_{\RR}(1-s). 
$$
Note that $\gamma^{\pm}(s) \gamma^{\pm}(1-s) = \pm 1$.

We next state  functional equations for the Lerch $L$-functions
expressed using the $\hR$-operator, to relate  functions at value
$s$ with functions at value $1-s$. These functional equations are obtained
by pushing the relation \eqn{806} forward
through the Weil-Brezin map $\sW_{N,d}(\chi)$; recall that 
$C_{N, d}(\tilde{d}, \chi)$ are defined by \eqref{e704} in Theorem~\ref{th71}.
%

\begin{theorem}~\label{th910} {\em (Generalized Lerch  Functional Equations)}
Suppose that $N \ne 0$.
Let $\chi$ be a primitive character $(\bmod~\fe)$ and suppose
that $\fe | d$ and $d|N$, and let $ \chi |_d$ denote the (generally imprimitive)
character $(\bmod~d)$ co-trained with $\chi$.
Then for  $0 < \Re(s) < 1$ the two
 Lerch $L$-functions $L_{N,d}^{\pm}(\chi |_d, s, a, c, z)$ 
associated to  $\sH_{N,d}(\chi |_d)$ satisfy
the functional equations
\beql{888}
\hR (L_{N,d}^{\pm})(\chi |_d, 1-s, a, c, z) =  
\chi(-1) \tau(\chi) |N|^{s-1} \gamma^{\pm}(s)  
\left(\sum_{ \tilde{d} | N} C_{N, d}(\tilde{d}, \chi)
L_{N,\tilde{d}}^{\pm}(\bar{\chi}|_{\tilde{d}}, s, a, c, z) \right),
\eeq
in which $\hR f(a, c, z) = f(-c, a, z - ac)$
and $\gamma^{\pm}(s)$ are Tate-Gelfand-Graev gamma
functions, and the coefficients $C_{N, d}(\tilde{d}, \chi)$
vanish whenever $\fe \nmid \tilde{d}$.
\end{theorem}

\paragraph{Remarks.}
(1) The restriction to $0 < \Re(s) < 1$ 
allows both sides  of the functional equation 
to  be viewed as functions of four variables
that are locally $L^{1}$ functions on the
Heisenberg group in the $(a, c, z)$ variables, 
which are continuous off 
the set where $a$ or $c$ take integer values.
Assuming the values of $a$ and $c$ are fixed, 
and are not integers, the associated function of
the variable  $s$ can be 
meromorphically continued from $0 < \Re(s) < 1$
to $s \in \CC$, with polar divisor set contained
in $\ZZ$. This can be proved in a standard fashion,
similar to that in  \cite{LL1}.

(2) The left side of \eqn{888} satisfies the identity
\begin{eqnarray}~\label{889} 
\hR (L_{N,d}^{\pm})(\chi |_d, 1-s, a, c, z) & = &
L_{N,d}^{\pm}(\chi |_d, 1-s, -c, a, z-ac) \nonumber \\
& = & e^{-  \pi i N ac}L_{N,d}^{\pm}(\chi |_d, 1-s, 1-c, a, z).
\end{eqnarray}
This identity facilitates comparison with the functional equations
given in \cite{LL1}, for $N=1, \chi= \chi_{0}.$

\begin{proof}
We assume $0 < \Re(s) <1$ to that we can identify the
homogeneous distribution $\Delta_s^{\pm} = (sgn (x))^k |x|^{s-1}$
with the corresponding locally $L^{1}$-function.
The calculations of the functional equation will be at
the level of locally $L^{1}$-functions, thus avoiding
the question of identifying distributions on different
Hilbert spaces.

We push forward the local functional equation \eqn{806} for
the homogeneous distribution $\Delta_s^{\pm}$
through the Weil-Brezin
map $\sW_{N,d}(\chi |_d)$ acting on tempered distributions
$\sS(\sW_{N,d}(\chi |_d))$. We have, formally,
that the tempered distribution $\sW_{N,d}(\chi |_d)(\Delta_s^{\pm})$
is given by
\begin{eqnarray}
\sW_{N,d}(\chi |_d)(\Delta_s^{\pm}) &= & e^{2 \pi i Nz} 
\left( \sum_{n \in \ZZ} \chi |_d(\frac{nd}{N}) (sgn(n + N\rrc))^k    
|n + N\rrc|^{-(1-s)} e^{2 \pi i n \rra} \right) \nonumber \\
& = & L_{N,d}^{\pm}(\chi |_d, 1-s, \rra, \rrc, z). \nonumber
\end{eqnarray}
The sum on the right
is conditionally convergent for $0 < \Re(s) < 1$ but gives
absolutely convergent sums when evaluated against any test function.
Now the  formula of  Theorem \ref{th71} asserts 
$$
\hR(\sW_{N,d}(\chi |_d)(f)) = \chi(-1) \frac{ \tau(\chi)} 
{|N|^{\frac{1}{2}}}\sum_{\tilde{d}~|~|N|}
C_{N,d}(\tilde{d}, \chi)~ \sW_{N, \tilde{d}}(\bar{\chi}|_{\tilde{d}})
(\sF \circ \hU(N)(f)),
$$
which we can apply for any Schwartz function $f$. This yields
\beql{890}
\hR( L_{N,d}^{\pm})(\chi |_d, 1-s, \rra, \rrc, z) =
\chi(-1) \frac{ \tau(\chi)} 
{|N|^{\frac{1}{2}}}\sum_{\tilde{d}~|~|N|}
C_{N,d}(\tilde{d}, \chi)~ \sW_{N, \tilde{d}}(\bar{\chi}|_{\tilde{d}})
(\sF \circ \hU(N)(\Delta_s^{\pm}).)
\eeq
Next, using \eqn{801} we obtain
$ U(N)(\Delta_s^{\pm}) = |N|^{s-\frac{1}{2}} \Delta_s^{\pm}.$
We deduce the equality of tempered distributions
\begin{eqnarray}
\sW_{N, \tilde{d}}(\bar{\chi} |_{\tilde{d}})
\left( \sF \circ \hU(N) (\Delta_s^{\pm}) \right)  & = &
\sW_{N, \tilde{d}}(\bar{\chi} |_{\tilde{d}}) 
\left( |N|^{s-\frac{1}{2}} \gamma^{\pm}(s) 
\Delta_{1-s}^{\pm} \right) \nonumber \\
& = &
|N|^{s- \frac{1}{2}} \gamma^{\pm}(s) e^{2 \pi i Nz}
\left( \sum_{n \in \ZZ} \bar{\chi} |_{\tilde{d}}(\frac{n \tilde{d}}{N})
(sgn(n+N\rrc))^k ~|n + N\rrc|^{-s} e^{2 \pi i n \rra} \right) \nonumber \\
& = & |N|^{s-\frac{1}{2}} \gamma^{\pm}(s) 
L_{N, \tilde{d}}^{\pm} ( \bar{\chi} |_{\tilde{d}}, s, \rra, \rrc, z). 
\nonumber
\end{eqnarray}
Substituting this in the right side of \eqn{890}
yields the functional equation \eqn{888}. 
\end{proof}

  The functional equations given in
Theorem~\ref{th910} typically involve several different
Lerch $L$-functions on their right side. 
They can however be
reformulated as a vector-valued functional equations 
for each Hilbert space $\sH_N(\chi; \fe)$ in Theorem~\ref{th72}.
One  uses a vector of 
the Lerch $L$-functions $L_{N,d}(\chi |_d, s, \rra, \rrc, z)$   with
a fixed sign $\pm$ associated to characters $\chi |_d$
indexed by the set 
$$
\Sigma (\fe, N) := \{ d \ge 1:   ~~\fe | d ~~\mbox{and}~~ d|N \},
$$
and relates it to Lerch $L$-functions associated to 
$\sH_N(\bar{\chi}; \fe)$.
The functional equation 
involves a matrix $M(\chi)$ whose entries are the
$C_{N, d}(\tilde{d}, \chi)$ with rows and columns indexed by $d$ 
(resp. $\tilde{d}$), both drawn from $\Sigma (\fe, N)$.\\

\paragraph{\bf Remark.} The  Lerch $L$-function
(with a proper choice of sign $\pm$) 
recovers the corresponding Dirichlet $L$-function
by taking a limit as $(\rra, \rrc) \longrightarrow (0, 0).$
The functional equation for the Dirichlet $L$-function
with a primitive character $\chi (\bmod~N)$ is recoverable from the
functional equation of the Lerch $L$-function 
$L_{N,N}^{\epsilon}(\chi, s, \rra, \rrc, z)$ with $\epsilon = \chi(-1)$
under this limiting process, taking $z=0$. 

  
%
%
\section{Concluding Remarks}\label{sec10}
\setcounter{equation}{0}

%
%
\subsection{Is the Lerch zeta function a global or a local object?}\label{sec101}
The Lerch zeta function has some unusual features.
One may view it as a  ``global'' zeta 
function attached  to the rational field $\QQ$, in 
the sense that,
when specialized the corners of the unit square $\Box$,
it yields  formally the Riemann zeta function.
The Lerch $L$-functions in this paper,  correspondingly specialize
 at the corners 
to Dirichlet $L$-functions,  and sometimes specialize
to be identically zero. 
When these are nonzero, these specializations
are global $L$-functions.

On the other hand,  the Lerch $L$-functions
$L_{N,d}^{\pm}(\chi, s, a, c, z)$
appear to behave  like a  kind of 
{\em local} $L$-function at the real place, in that
Theorem~\ref{th95} exhibits them as 
the image under a Weil-Brezin map of a local
homogeneous distribution at the {\em real place}.
This  interpretation leads to the  question
whether  there  exist analogous
constructions of ``Lerch $L$-functions''  at 
other local places. The thesis  of Ngo \cite{Ngo14}
gives an adelic construction, in terms of certain zeta integrals, 
for local fields and globally for  number fields
and function fields.
In the  globalization  the archimedean places
play a special role, 
leading  to the global functions not  having an Euler product. 
The Lerch $L$-functions 
in \eqref{100} appear as global zeta integrals for particular adelic
test functions. 

%
%
\subsection{Heisenberg modules invariant under all Hecke operators and $\hR$-operator}\label{sec102}
One may consider the  algebra $\sA_N$  of operators acting on $\sH_N$ having $\CC$-coefficients
 generated by the set of  all two-variable Hecke operators
$\{\hT_m: m \ge 1\}$ together with  the $\hR$-operator. This algebra  is a $*$-algebra (by Theorem \ref{th31}),  which has an
interesting  non-commutative  structure, particularly for $|N| \ge 2$.
To formulate a decomposition into subspaces of $\sH_N$ 
that are invariant under the action of the whole algebra $\sA_N$
we find (on combining Theorem \ref{th43}
with Theorem \ref{th72}
) 
that one  must use  a coarser decomposition than any so far, indexed by a
primitive Dirichlet character $\chi (\bmod~ \fe)$ together with 
its contragredient $\bar{\chi}$, which is given by
$$
\sH_{N}^{prim}(\chi, \bar{\chi}) := \bigoplus_{{d}\atop{\fe \mid d \mid N}} \left(\sH_{N,d}(\chi |_d)
\oplus \sH_{N,d}(\bar{\chi} |_d) \right).
$$
It might prove worthwhile to study 
 the action of these operators on the full Hilbert space associated
to a single primitive character $\chi (\bmod~ \fe)$ and its contragredient,
where the level $N$ may vary, as 
$$
\sH^{prim}(\chi, \bar{\chi}) := 
\bigoplus_{N \in \ZZ \smallsetminus \{0\}} \Big(\bigoplus_{{N}\atop{ \fe | |N| }}\sH_{N}^{prim}(\chi, \bar{\chi})\Big)
$$ 
Here one  allows positive and negative $N$ in the sum. Note that
$\sH_N$ for positive $N$ are associated to holomorphic functions
(theta functions) while those for negative $N$ are associated
to {\em anti-holomorphic} functions.

%
%
\subsection{$xp$ operator and Riemann hypothesis}\label{sec103}
This paper showed  that
the  Lerch $L$-functions on the critical line
are  generalized eigenfunctions
for a spectral decomposition associated to the action of the
dilation group $\hV(t)$. In  particular they 
are generalized eigenfunctions for a differential operator having the form
``$xp$'' noted by Berry and Keating ~\cite{BK99a}, explicitly 
exhibited in \eqref{800aa}. 

At certain limiting values of
their domain variables, such as $(\rra, \rrc, z) = (1, 1, 0)$,
they yield Dirichlet $L$-functions, 
generalizing the case of the Riemann zeta function treated  in  \cite{LL1}.
What seems noteworthy  is that 
 Dirichlet $L$-functions  are 
associated with a { multiplicative} structure,
while Lerch $L$-functions embody 
an { additive structure}, coming from the group law on 
the Heisenberg group. 
The  Heisenberg group structure brings together both  the additive and multiplicative
structures  via a limit process. 

This structure might conceivably be
 relevant to understanding the Riemann hypothesis.
Berry and Keating (\cite{BK99a}, \cite{BK99b})
have suggested that suitable operators of  ``$xp$'' form might be
involved in a spectral interpretation of the Riemann hypothesis.
They studied $\frac{1}{2}(xp ~+ px)$ which is  a
(formally) self-adjoint operator. The operator $\Delta_L$ is of an
analogous form, although  in our framework 
it is a (formally)  skew-adjoint operator, cf Theorem~\ref{th95}.

%
%
\subsection{Further work}\label{sec104}

A   sequel  (\cite{LH2}), will address further points in understanding
the two-variable Hecke operators acting on
the  Heisenberg module $L^2(H(\ZZ)\backslash H(\RR))$. 
First  it shows that all the $GL(1)$ $L$-functions associated to $\QQ$
(Dirichlet $L$- functions) do appear directly in their traditional
multiplicative guise as {\em generalized eigenvalues}
rather than as generalized eigenfunctions. 
Their values on the
critical line then give the {\em spectral multiplier function}
for pure continuous spectra of
a ``zeta operator''
\beql{901}
\hZ : = \frac{1}{2} \left( \sum_{m \in \ZZ \backslash \{0\}} \hT_m \right),
\eeq 
acting as an unbounded operator on suitable subspaces
 of $L^2(H(\ZZ)\backslash H(\RR))$, namely on $\sH_{N,N}(\chi)$. 
Second, it treats the Hecke operator  action on the ``degenerate'' module $\sH_0$,
which completes
 the study of their action on  $\sH_N~ (N \ne 0)$ given here. 
Third, it observes that the results of Milnor~\cite{Mi83}
and of Bost and Connes~\cite{BC95} have an interpretation in
terms of this action restricted to certain subspaces of $\sH_0$.

The results of this paper motivate
further study of automorphic representations 
and automorphic forms on the sub-Jacobi group $H^J$
and  related groups. Automorphic
representations over adelic nilpotent groups, including
the Heisenberg group, were worked out in 1965
by C. C. Moore \cite{Mo65}. 
There has been extensive study
of automorphic forms on the full Jacobi group, see
Berndt and Schmidt ~\cite{BS98}. However the structures
considered here have not  been studied  in 
the context of the Jacobi group. Some other points relevant to 
 an adelic treatment for the Heisenberg group
are given in  Haran \cite[Chap. 12]{Har01}.\\

%
%
\section{Appendix A.  Heisenberg and sub-Jacobi Groups}\label{secA1}
\setcounter{equation}{0}

We describe various matrix realizations of
the (real) Heisenberg and sub-Jacobi groups.

We consider the one-parameter family of real Lie groups $G_\lambda$
specified by the real parameter $\lambda$,  
whose underlying elements $(x,y,z) \in \RR^3$ with group law
given by 
\beql{eq81}
[x_1, y_1, z_1 ]_\lambda \circ [ x_2, y_2, z_2 ]_{\lambda} =
[x_1 + x_2 , y_1 + y_2 , z_1 + z_2 + \lambda x_1 y_2 + 
(1- \lambda ) y_1 x_2  ]_\lambda \,.
\eeq
The group law $\lambda = 0$ is the {\em nonsymmetric Heisenberg group} 
considered in this paper, and the case $\lambda= \frac{1}{2}$ is
the {\em symmetric Heisenberg group}. 
The groups $G_{\lambda}$ are isomorphic as real Lie groups.
An explicit isomorphism $\alpha_\lambda : G_1 \to G_\lambda$ is given by
\beql{eq82}
\alpha_\lambda ([x,y,z]_1 ) = [x,y,z - (1- \lambda ) xy ]_\lambda
\eeq
with
\beql{eq83}
\alpha_\lambda^{-1}
([ \tilde{x}, \tilde{y} , \tilde{z} ]_\lambda ) =
[\tilde{x}, \tilde{y}, \tilde{z} + (1- \lambda ) \tilde{x} \tilde{y} ]_1\,.
\eeq

These groups have finite-dimensional matrix representations.
In terms of the parameter $\lambda$, the  case 
$\lambda = 0$ has the three-dimensional (real) matrix representation
$$
[x,y, z]_0 =
\left[ \begin{array}{ccc}
1 & y & z \\
0 & 1 & x \\
0 & 0 & 1 \\
\end{array}
\right].
$$\
and the case $\lambda = 1$ has the three-dimensional (real) matrix representation
$$
[x,y, z]_1 =
\left[ \begin{array}{ccc}
1 & x & z \\
0 & 1 & y\\
0 & 0 & 1 \\
\end{array}
\right].
$$

For general $\lambda \in \RR$ the groups $G_\lambda$ have a $4 \times 4$ (real) 
linear representation
\beql{eq84}
[x,y,z]_\lambda =
\left[ \begin{array}{cccc}
1 & x & y & z \\
0 & 1 & 0 & \lambda y \\
0 & 0 & 1 & - (1- \lambda ) x \\
0 & 0 & 0 & 1
\end{array}
\right] \,.
\eeq
The symmetric Heisenberg  group $G_{1/2}$ has an alternative  $4 \times 4$ 
matrix representation in
terms of variables $(p, q, z)$, as
\beql{eq85}
[p,q,z]_{1/2} =:
\left[
\begin{array}{cccc}
1 & p & q & 2z \\
0 & 1 & 0 & q \\
0 & 0 & 1 & -p \\
0 & 0 & 0 & 1
\end{array}
\right] \,.
\eeq
Certain properties of the Heisenberg group are more easily 
visible using the $4 \times 4$ matrix representation.
Setting $G_1 = H(\RR )$, $G_{1/2} = H_{sym} (\RR )$ and define
\beql{eq86}
\Gamma_{sym} (1) := \alpha_{1/2} (H(\ZZ )) ~.
\eeq
Then
\begin{eqnarray}\label{eq87}
\lefteqn{\Gamma_{sym} (1):= \Bigl\{ 
(p,q,z ): p,q \in \ZZ ~\mbox{and}~ z \in \frac{1}{2} \ZZ 
\setminus \ZZ} \nonumber \\ [+.05in]
&& \qquad \quad \mbox{if}~ p \equiv q \equiv 1 ~ (\bmod~2) ~\mbox{and}~
z \in \ZZ ~\mbox{otherwise} \Bigr\} \,.
\end{eqnarray}
Let $F(p,q,z) : H_{sym} (\RR) \to \CC$ be a function 
in $L^2 ( \Gamma_{sym} (1) \setminus H_{sym} (\RR ))$.
We have the Hilbert space decomposition
\beql{eq88}
L^2 \left( \Gamma_{sym} (1) \setminus H_{sym} (\RR ) \right) =
\bigoplus_{N \in \ZZ} \sH_N^{sym} ~.
\eeq
The smooth functions $F \in \sH_N^{sym}$ satisfy
\begin{eqnarray}\label{eq89}
F(p,q,z) & = & e^{2 \pi i Nz} F(p,q,0) \nonumber \\
F(p+1, q, 0) & = & e^{\pi iq} F(p,q,0) \nonumber \\
F(p,q+1,0) & = & e^{-\pi i p} F(p,q,0) \,.
\end{eqnarray}
The Hilbert space inner product on $\sH_N^{sym}$ is
\beql{eq810}
\langle F, G \rangle = \int_0^1 \int_0^1 \int_0^1
F(p,q,z) \overline{G(p,q,z)} dp dq dz \,.
\eeq
There is a Hilbert space isomorphism 
$\alpha_{1/2}^\ast : \sH_N \to \sH_N^{sym}$ which sends 
a smooth function $F( \rra ,\rrc ,w)$ in $\sH_N$ to the smooth function
\beql{eq811}
\tilde{F} (p,q,z) = F \left( p,q, z- \frac{1}{2} pq \right) = e^{\pi i Npq}
F(p,q,z) \,.
\eeq
in $\sH_N^{sym}$.

We now consider matrix versions
the sub-Jacobi group  $H^J(\RR)$   treated in Section \ref{sec7}
This is an  extension $H^J(\RR)$ of the Heisenberg group 
$G_1 ( \RR )$ by $\RR_{>0}^\ast$. It is a four-dimensional
exponential solvable Lie group, and is unimodular.
The Haar measure is $d\mu = \frac{dt}{t} \, da \, dc \, dz.$
In Section \ref{sec7} we observed that it  has a  
$4 \times 4$ matrix representation, 
given by , for $t >0$, 
\beql{eq813}
[\rrc, \rra, z, t ] =
\left[ \begin{array}{cccc}
1 & \rrc  & \rra  & z \\
0 & t& 0 & t \rra \\
0 & 0 & \frac{1}{t} & 0 \\
0 & 0 & 0 & 1
\end{array}
\right] \,
\eeq
A  symmetrized form of this group action is 
\beql{eq814}
[\rrc, \rra, z, t ] =
\left[ \begin{array}{cccc}
1 & \rrc & \rra  & z \\
0 & t& 0 & t\rra \\
0 & 0 & \frac{1}{t}& -\frac{1}{t}\rrc \\
0 & 0 & 0 & 1
\end{array}
\right] \,.
\eeq
The latter group multiplication does not coincide
with \eqn{eq813}, but the resulting groups are isomorphic  as 
Lie groups. Using the change of variable $t= e^u$, we can view this
solvable group as being homeomorphic to $\RR^4$, with variables
$(u, a, c, z)$ and Haar measure $d\mu' := du \, da\, dc\, dz$.

The universal enveloping algebra of the sub-Jacobi
group has a two-dimensional center. One generator is
the vector field $\frac{\partial}{\partial z}$ associated
with the center of the Heisenberg Lie algebra. The other
generator is a second order differential operator
which is a ``lift'' of the Heisenberg operator
$$
\Delta_L 
:= \frac{1}{2\pi i} \frac{\partial }{\partial a} \frac{\partial }{\partial c} + 
c \frac{\partial }{\partial c}\frac{\partial }{\partial z}
+ \frac{1}{2}\frac{\partial }{\partial z}. 
$$

The  sub-Jacobi group  
is a subgroup of the Jacobi group, 
viewed in its  $4 \times 4$
matrix representation, contained inside $SL(4, \RR)$.\\


\section{Appendix B.  The  Dilation-Invariant Operator $x \frac{d}{dx} + \frac{1}{2}$ }\label{sec120}
\setcounter{equation}{0}

This appendix gives results on (unbounded) operators on a Hilbert space
that commute with the action of a locally compact
Lie group $G$ due to  Burnol  \cite{Bu01}, \cite{Bu02},. We apply it to the group of dilations $G= \RR^{\ast}$ 
and the operator $x \frac{d}{dx} + \frac{1}{2}$ acting inside $L^2 ( \RR, dx)$ .
Burnol's work was originally done to understand invariances of the ``explicit formula" in 
prime number theory (\cite{Bu98}, \cite{Bu99}, \cite{Bu00}), which includes this case.

%
%
\subsection{ Operators Commuting with a Locally Compact Group Action}\label{secB1}

Let $G$ be a locally compact abelian group and $\hat{G}$
be its dual group of unitary characters. There is
a two-sided Haar measure $dg$ on $G$ unique up to a multiplicative
constant, and there is a dual Haar measure $d\chi$ on $\hat{G}$
such that the $G$-Fourier transform
$$
\sF_G(\varphi)(\chi) := \int_{G} \varphi(g)\chi(g) dg
$$
is an isometry from $L^2(G, dg)$ to $L^2(\hat{G}, d \chi)$.
The group action
$\hU(g) f(x) = f(xg)$ on $L^2(G, dg)$
consists of unitary operators. We consider 
(possibly unbounded) operators that respect the group action.\\
%

\begin{defi}~\label{de71}
{\em 
(1) A (possibly unbounded) operator $\hM$ with dense domain $\sD$
on a separable Hilbert space $\sH$ is said to {\em commute}
with a bounded operator $A$ on $\sH$ if $\hA$ maps $\sD$ into $\sD$
and 
\beql{800}
\hM (\hA \bv) = \hA ( \hM \bv),   ~~\mbox{for all}~~ \bv \in \sD.
\eeq

(2) Let $G$ be a locally compact group, with (left) 
Haar measure $dg$.
 A (possibly unbounded) operator $\hM$ with dense domain $\sD$
on $L^2( G, dg)$ is said to {\em commute with $G$} if it
commutes with all unitary operators $\{ \hU(g) :~ g \in G \}$,
given by $\hU(g)(f)(h):= f(hg)$. 
}
\end{defi}

If $G$ is a  locally compact abelian group, then the closed operators that
commute with $G$ can be characterized as multiplication
operators on $\hat{G}$, as follows.  Let
$a(\chi)$ denote a Borel-measurable function on $\hat{G}$, not 
necessarily bounded. Let $\sD_a \subset L^2(G, dg)$ be the
domain of (equivalence-classes of) square-integrable functions
$\varphi(g)$ on $G$ such that $a(\chi) \sF_G(\varphi)(\chi)$ belongs
to $L^2(\hat{G}, d \chi)$. Let $(\tilde{\hM}_a, \sD_a)$ denote the 
(possibly unbounded) closed operator on
$ L^2(G, dg)$ with domain $\sD_a$ acting by
$\tilde{\hM}_a(\varphi) := \sF_G^{-1} \circ \hM_a \circ \sF_G (\varphi).$
where $\hM_a$ is the multiplication operator by $a(\chi)$
on $L^2(\hat{G}, d\chi)$. We call $a(\chi)$ its associated
{\em spectral multiplier function}. In the case $G=\RR^{\ast}$  
for a closed dilation-invariant
operator $T$ we use the alternate  notations  $a_0(\tau,T), a_1(\tau, T)$
$( -\infty < \tau < \infty)$ or
just $a_0(\tau), a_1(\tau)$ to denote the spectral multiplier functions
associated to the characters $|y|^{i \tau}$
and $sgn(y) |y|^{i \tau}$, respectively.\\
%

\begin{prop}~\label{prop81} {\rm (Burnol)}
Let $G$ be a locally compact abelian group with
two-sided  Haar measure $dg$.
 
(1) For each measurable function $a(\chi)$
on $\hat{G}$ the (spectral) multiplication operator 
$(\tilde{M}_a, \sD_a)$ on $L^2(G, dg)$ is closed, has $\sD_a$ as
a dense domain, and  
commutes with $G$. If $(\tilde{M}_b, \sD_b)$ extends
$(\tilde{\hM}_a, \sD_a)$, then $a(\chi) = b(\chi)$ (up to sets of
measure zero) and $(\tilde{M}_b, \sD_b)= (\tilde{\hM}_a, \sD_a)$.
The adjoint of $(\tilde{M}_a, \sD_a)$ is 
 $(\tilde{\hM}_{\bar{a}}, \sD_{\bar{a}})$, with
$\sD_{\bar{a}}= \sD_a$.

(2) Suppose that $(\hM, \sD)$ is
a (possibly unbounded)  operator with dense domain
in $L^2(G, dg)$ which is closed
and commutes with the elements of $G$. Then 
$(\hM, \sD) \equiv (\tilde{\hM}_a, \sD_a)$ for
a measurable function $a(\chi)$ on $\hat{G}$, which is
unique up to a set of measure zero.

(3) The operator $(\tilde{M}_a, \sD_a)$
is bounded if and only if $a(\chi)$ is essentially bounded.

(4) The adjoint of the 
operator $(\tilde{M}_a, \sD_a)$ 
is $(\tilde{M}_{\bar{a}}, \sD_{\bar{a}})$, with $\sD_{\bar{a}}= \sD_a$.
It is self-adjoint if and only if $a(\chi)$ is 
(essentially) real-valued.
\end{prop}

\begin{proof} (1) This is Lemma 2.3 in Burnol~\cite{Bu01}.

(2)  This is Theorem 2.4 in 
Burnol~\cite{Bu01}.

(3) This follows from the definition of the domain $\sD_a$. 

(4) This is part of  Lemma 2.3 in Burnol~\cite{Bu01}.
\end{proof}

We next recall another result of Burnol in which the domain
$\sD$ is not maximal.
%

\begin{prop}~\label{prop82}{\rm (Burnol)}
Let $G$ be a locally compact abelian group with
two-sided  Haar measure $dg$. Suppose that $(\hM, \sD)$ is
a (possibly unbounded)  operator with a dense domain 
in $L^2(G, dg)$ which is symmetric
and commutes with $G$. Then $(\hM, \sD)$ is essentially
self-adjoint, and if it is closed then  it is
self-adjoint.
\end{prop}

\begin{proof} This is Corollary 2.6  in 
Burnol~\cite{Bu01}. The closure of such an operator
is a spectral multiplication operator,  with
a real-valued multiplier function $a(\chi)$, by Proposition~\ref{prop81}.
\end{proof}

%
%
\subsection{Spectral Decomposition of Dilation-Invariant Operators on $L^2 (\RR, dx)$ }\label{secB2}

We now specialize to the 
case of the dilation group $(G, dg) \equiv (\RR^{\ast}, \frac{dy}{|y|}).$

Its unitary dual $\widehat{G} := \widehat{\RR}^{\ast}$ consists of 
two one-parameter families of unitary characters,
$\chi_{\tau}^{+}(t) = |t|^{i \tau}$ and
$\chi_{\tau}^{-}(t) = \sgn (t) |t|^{i \tau}$,
with $\tau \in \RR$,
and can be identified with $\widehat{G} = \RR \oplus \ZZ/ 2 \ZZ.$ 
The corresponding Fourier transform $\sF_G$ 
is a variant of the Mellin transform,
which uses  the {\em  two-sided Mellin transforms,} given for $k=0,1$ by 
\beql{800b}
\sM_k(f)(s) := \int_{\RR^{\ast}} f(t) (sgn (t))^k |t|^{s} \frac{dt}{|t|}.
\eeq
The transform is as follows.

%
%
\begin{lemma}\label{lemB4a}
The Mellin transform integrals $\sM_k(f)(s)$ are well-defined for
Schwartz functions $f(t)$ on $\RR^{\ast}$ on the imaginary axis $s \in iR$
and the map 
$$ 
f(t) \in \sS(\RR^{\ast}) \longmapsto 
\sF_G(\tau):= 
\left(\sM_0(f)(i\tau), \sM_1(f)(i\tau)\right).
$$ 
extends to an isometry
$\sF_G:  L^2(\RR^{\ast}, \frac{dy}{|y|}) \to L^2(\widehat{\RR}^{\ast}, d \chi)$
where 
$$
L^2(\widehat{\RR^{\ast}}, d \chi) \equiv 
L^2(\RR \oplus \ZZ/ 2 \ZZ, \frac{d \tau}{2\sqrt{2} \pi})
$$
corresponds to the action of the Mellin transform on
the imaginary axis. 
\end{lemma}

Now consider a closed dilation-invariant operator $\hM$ on
$L^2(\RR^{\ast}, \frac{dx}{|x|})$. We use  a special notation for the spectral multiplier
functions $a(\chi)$ above of such a function 
We write them as two families
$a_0(\tau, \hM), a_1(\tau, \hM)$ 
$( -\infty < \tau < \infty)$, or
just $a_0(\tau), a_1(\tau)$ when $\hM$ is understood,
associated to the characters $|y|^{i \tau}$
and $sgn(y) |y|^{i \tau}$, respectively.

These results for $L^2(\RR^{\ast}, \frac{dy}{|y|})$
carry over to $L^2(\RR, dx)$
with the action $U(t)(f)(x) = |t|^{\frac{1}{2}} f(tx)$, 
using the (inverse of the)
isometry $\varphi: L^2(\RR, dx) \to L^2(\RR^{\ast}, \frac{dy}{|y|})$
given by $f(x) \longmapsto \varphi(f)(y) := |y|^{1/2} f(y).$
Operators on $L^2(\RR, dx)$ that commute with the unitary
operators $U(t)$ push forward
to operators on  $L^2(\RR^{\ast}, \frac{dy}{|y|})$ that
commute with the $\RR^{\ast}$-action. 

The isometry 
$$ 
\sF_G \circ \varphi: L^2(\RR, dx) \to
 L^2( \RR  \oplus\ZZ/ 2 \ZZ, \frac{d \tau}{2\sqrt{2} \pi})
$$
has an inverse  given by an integral formula
(inverse Mellin transform) 
for all sufficiently nice functions
$f(x) \in L^2(\RR, dx).$ 
For Schwartz functions $f(x) \in \sS(\RR)$
the inversion formula is
$$
f(x) =  \int_{-\infty}^\infty 
\sM_0(f)( \frac{1}{2} + i \tau)
|x|^{-\frac{1}{2} -i\tau} \frac{d\tau}{2\sqrt{2}\pi} + 
  \int_{-\infty}^\infty \sM_1(f)( \frac{1}{2} + i \tau)
sgn(x) |x|^{-\frac{1}{2} -i\tau} \frac{d\tau}{2\sqrt{2}\pi}.
$$

In part II we will tabulate 
spectral multiplier functions for
various  dilation-invariant operators on $\sH_{N,d}(\chi)$.
The dilation operators on  $\sH_{N,d}(\chi)$  are 
$\hV(t):= \sW_{N,d}(\chi) \circ \hU(t) \circ \sW_{N,d}(\chi)^{-1}.$
Note that the reflection operator 
$$\hR^2 (F) (\rra, \rrc, z) := F(- \rra, -\rrc, z )$$
has  $\hR^2= \hV(-1)= \chi(-1) \hT_{-1}$, and its spectral 
multiplier function
is $a_0(\tau)\equiv 1$, $a_1(\tau)\equiv -1$.

%
%
\subsection{Continuous spectrum for $x \frac{d}{dx} + \frac{1}{2}$
on $L^2(\RR, dx)$}\label{secB3}

Proposition~\ref{prop81} specifies
for each  unbounded operator above 
a (maximal) domain on which is it closed and commutes
with $\RR^{\ast}$. For some purposes it is useful to have
a smaller dense domain $\Phi$ which is left invariant by the
operator, with the maximal domain being recoverable by
taking the closure of this operator. 
This is 
relevant in describing continuous spectra, which fall outside
the Hilbert space. Continuous spectra can sometimes be described as 
generalized functions (distributions), using
the dense domain $\Phi$ as a space of allowed test functions. 

We recall that a  {\em rigged Hilbert space}
(or {\em Gelfand triple}) consists of $(\Phi, \sH, \Phi')$
in which $\Phi$ is a dense vector subspace of a separable Hilbert space
$\sH$ endowed with a Frechet topology finer than the Hilbert space
topology, and $\Phi'$ is the dual space to $\Phi$, 
viewed as a space of generalized functions so that one
has the inclusions $\Phi \subset \sH \subset \Phi'$.
The original definition requires  that $\Phi$ be
a nuclear space in the sense of Grothendieck 
and we impose  this requirement, although some  authors
do not,
e.g.  Wickramasehara and Bohm \cite{WB02}, \cite{WB03}.
(In the physics literature, such as \cite{WB02} 
the space $\Phi'$ is a space of conjugate-linear functionals,
rather than linear functionals as in our definition.)

For the dilation invariant operator $D= x \frac{d}{dx} + \frac{1}{2}$ 
on $\sH= L^2(\RR, dx)$ we take
$\Phi= \sS(\RR),$ one natural domain
is  the Schwartz space with its Fr\'{e}chet  topology,
in which case  $\Phi' =  \sS'(\RR)$ is the space of tempered distributions.
The  Schwartz space is relevant because it
is the space of smooth vectors for the
Schr\"{o}dinger representation of the Heisenberg group acting
on $L^2(\RR, dx)$, see Howe~\cite{Ho80a}, \cite[p. 827]{Ho80}. 
The Schwartz space is  invariant under all
dilations $\{U(t): t \in \RR\}$, under the Fourier transform $\sF$
and under additive translations $\hT(t)f(x)= f(x+t)$.
However it is not invariant under the inversion 
$If(x) = \frac{1}{|x|} f(\frac{1}{x})$, or under
the modified Poisson operator $P$ or the co-Poisson operator $P'$.
For the dilation-invariant differential operator 
$D= x \frac{d}{dx} + \frac{1}{2}$
with this domain we have the following facts.
%

\begin{prop}~\label{prop83} 
(1) The operator $D= x \frac{d}{dx} + \frac{1}{2}$ leaves
the domain $\sS(\RR)$ invariant, and on  this domain it
is essentially skew-adjoint.

(2) It has a purely absolutely continuous
spectrum, with generalized eigenfunctions $f_{k, \tau}(x)$
parametrized by $(k, \tau) \in  \ZZ/ 2\ZZ \oplus \RR $, as
$$
f_{0, \tau}(x) = |x|^{-\frac{1}{2} + i \tau}, ~~~
f_{1, \tau}^{1}(x) = sgn(x) |x|^{-\frac{1}{2} + i \tau},
$$
viewed as tempered distributions. The spectral measure
is $\frac{d\tau}{2\sqrt{2}\pi}$ on both real components of
the continuous spectrum, with
spectral multiplier functions
$$ 
a_0(\tau) = -i \tau~~~\mbox{and}~~~ a_1(\tau) = -i \tau.
$$

(3)  For all elements $f(x) \in \sS(\RR)$ the following
two formulae converge absolutely:
$$
f(x) = 
\int_{-\infty}^\infty \sM_0 (f) (\frac{1}{2} + i \tau)
 |x|^{-\frac{1}{2} - i \tau} \frac{d\tau}{2\sqrt{2} \pi}  +
\int_{-\infty}^\infty \sM_1 (f) (\frac{1}{2} + i \tau)
 sgn(x)|x|^{-\frac{1}{2} - i \tau} \frac{d\tau}{2\sqrt{2} \pi},
$$
and
$$
Df(x) = 
 -\int_{-\infty}^\infty i \tau \sM_0 (f) (\frac{1}{2} + i \tau)
 |x|^{-\frac{1}{2} - i \tau} \frac{d\tau}{2\sqrt{2} \pi}  -
\int_{-\infty}^\infty i \tau \sM_1 (f) (\frac{1}{2} + i \tau)
)sgn (x)|x|^{-\frac{1}{2} - i \tau} \frac{d\tau}{2 \sqrt{2}\pi},
$$
\end{prop}

\begin{proof}
The Schwartz space is a nuclear Fr\'{e}chet space with its
usual topology using seminorms, and the result on
self-adjointness and continuous spectrum follows from the 
Nuclear Spectral theorem of Gel'fand-Maurin (Gelfand and Vilenkin \cite{GV61},
 Maurin \cite{Mau72}, see Bohm and Gadella~\cite[p. 25]{BG89}) 
applied to the operator $iD$.

The first formula is the inverse Mellin transform, separated into
even and odd function parts. (The  scaling factor  in the
measure divides by an extra  factor of $2$ coming from the definition of $\sM_j$
integrating over the whole real line.)  The second formula identifies
the spectral multipliers and is obtained
by differentiating the first formula under the integral sign.
\end{proof}

%

%
%
%


\end{document}